\documentclass[a4paper, 11pt]{amsart}
\usepackage[
  margin=1.9cm,
  includefoot,
  footskip=30pt,
]{geometry}
\usepackage{amsmath,amssymb,amscd}
\usepackage{amsfonts}
\usepackage[english]{babel}
\usepackage[leqno]{amsmath}
\usepackage{amssymb,amsthm}
\usepackage{mathrsfs} 
\usepackage{amscd}
\usepackage{enumerate}
\usepackage{epsfig}
\usepackage{relsize}
\usepackage{layout}
\usepackage{lscape}
\usepackage[usenames,dvipsnames]{xcolor}
\usepackage[backref=page]{hyperref}
\usepackage{tikz}
\hypersetup{
 colorlinks,
 citecolor=Green,
 linkcolor=Red,
 urlcolor=Blue}
\usepackage[matrix,arrow,tips,curve]{xy}
\input{xy}
\xyoption{all}
\definecolor{darkgreen}{rgb}{0.0, 0.7, 0.0}
\definecolor{purple}{rgb}{0.5, 0.0, 0.5}
\definecolor{red}{rgb}{0.8, 0.2, 0.0}

\newtheorem{thm}{Theorem}[section]
\newtheorem{bthm}{Theorem}
\newtheorem{bcor}{Corollary}
\newtheorem{lemma}[thm]{Lemma}

\newtheorem{prop}[thm]{Proposition}

\newtheorem{claim}[thm]{Claim}
\newtheorem{subclaim}[thm]{Subclaim}

\numberwithin{equation}{section}
\theoremstyle{definition}
\newtheorem{defi}[thm]{Definition}

\newtheorem{notation}[thm]{Notation}

\theoremstyle{remark}
\newtheorem{remark}[thm]{Remark}

\newtheorem{example}[thm]{Example}
\newcommand{\Z}{\mathbb{Z}}
\newcommand{\C}{\mathbb{C}}

\newcommand{\R}{\mathbb{R}}

\newcommand{\Pic}{\operatorname{Pic}}

\newcommand{\cHom}{\mathcal Hom}
\DeclareMathOperator{\Hom}{{Hom}}
\DeclareMathOperator{\Ext}{{Ext}}

\def \Im{{\rm Im}}

\DeclareMathOperator{\id}{id}

\def \P{\mathbb{P}}

\def \F{\mathcal F}
\def \N {\mathcal N}

\def\I{\mathcal I}
\def \L{\mathcal L}
\def \E{\mathcal E}
\def \G{\mathcal G}
\def \H{\mathcal H}
\def \U{\mathcal U}
\def\O{\mathcal O}

\def\M0{\mathcal M^0}

\DeclareMathOperator{\Proj}{{Proj}}
\DeclareMathOperator{\Sym}{{Sym}}

\title{On the classification of non-big Ulrich vector bundles on fourfolds}

\author[A.F. Lopez, R. Mu\~{n}oz and J.C. Sierra]{Angelo Felice Lopez*, Roberto Mu\~{n}oz and Jos\'e Carlos Sierra}

\address{\hskip -.43cm Angelo Felice Lopez, Dipartimento di Matematica e Fisica, Universit\`a di Roma
Tre, Largo San Leonardo Murialdo 1, 00146, Roma, Italy. e-mail {\tt lopez@mat.uniroma3.it}}

\address{\hskip -.43cm Roberto Mu\~{n}oz, Departamento de Matem\'atica Aplicada a las TIC, ETSISI Universidad Polit\'ecnica de Madrid. C/ Alan Turing s/n. 28031, Madrid, Spain. email: {\tt roberto.munoz@upm.es}}

\address{\hskip -.43cm Jos\'e Carlos Sierra, Departamento de Matem\'aticas Fundamentales, Facultad de Ciencias, UNED,  C/ Juan del Rosal 10, 28040 Madrid, Spain. e-mail {\tt jcsierra@mat.uned.es}}

\thanks{* Research partially supported by  PRIN ``Advances in Moduli Theory and Birational Classification'' and GNSAGA-INdAM}


\begin{document} 

\begin{abstract} 
We give an almost complete classification of non-big Ulrich vector bundles on fourfolds. This allows us to classify them in the case of Picard rank one fourfolds, of Mukai fourfolds and in the case of del Pezzo $n$-folds for $n \le 4$. We also classify Ulrich bundles with non-big determinant on del Pezzo and Mukai $n$-folds, $n \ge 2$.
\end{abstract}

\maketitle

\section{Introduction}

Let $X \subseteq \P^N$ be a smooth irreducible complex variety of dimension $n \ge 1$. A vector bundle $\E$ on $X$ is Ulrich if $H^i(\E(-p))=0$ for all $i \ge 0$ and $1 \le p \le n$. We refer for example to \cite{es, b1, cmp} and references therein for the importance of Ulrich vector bundles and the relation with properties of $X$. 

One often useful geometrical consequence of the existence of a non-big Ulrich vector bundle $\E$ on $X$ is that $X$ is covered by linear spaces, as shown in \cite[Thm.~2]{ls}. This property gives very strong conditions on the geometry of $X$ especially in low dimension, thus allowing the classification of non-big Ulrich vector bundles for $n \le 3$, achieved in \cite{lo, lm}. On the other hand, already for $n=4$, the classification of varieties covered by lines is far more incomplete. For example, for $n=4$, one can have scrolls $X \to Y$ which are a projective bundle only over a proper open subset of a threefold $Y$. Nevertheless, denoting by $H$ a hyperplane section of $X$, observe that still a lot of information on the pair $(X,H)$ is available in dimension $n \ge 2$. We have (see \cite{bs2} and \S \ref{adj}) the nef value $\tau = \tau(X,H)$ and the nef value morphism
$$\phi_{\tau} =  \phi_{\tau}(X,H):= \varphi_{m(K_X+\tau H)} : X \to X'.$$ 
Classical adjunction theory allows to divide the pairs $(X,H)$ into different cases, depending on the behaviour of $\phi_{\tau}$ (this is especially useful for $n=4$, see Lemma \ref{aggiu}). In this paper we study this morphism in the presence of a non-big rank $r$ Ulrich vector bundle $\E$, as follows. We have a morphism $\Phi : X \to {\mathbb G}(r-1, \P H^0(\E))$ and its lift with connected fibers in the Stein factorization $\widetilde \Phi : X \to \widetilde{\Phi(X)}$. When $c_1(\E)^n = 0$, we first prove a useful technical result, the Dichotomy Lemma (see Lemma \ref{dico}), that allows to compare the fibers of $\Phi$ (or of $\widetilde \Phi$) and of a given morphism starting from $X$. A nice application of it is given in two standard cases arising in adjunction theory, a quadric fibration in Proposition \ref{qf} and a blow-up of a point in Proposition \ref{caso(e)}. Moreover when $n=4$, applying the Dichotomy Lemma to $\phi_{\tau}$ in several cases, together with the results in \cite{ls} and \cite{lms} in the case $c_1(\E)^n > 0$, we get a pretty complete classification of non-big Ulrich vector bundles, as stated below. In many cases we have a linear Ulrich triple, in the sense of Definition \ref{not4}. The cases in the theorem below are also listed in Tables \ref{tab1} and \ref{tab2} at the end of the paper.

\begin{bthm}
\label{main3}

\hskip 3cm

Let $X \subseteq \P^N$ be a smooth irreducible variety of dimension $4$ and let $\E$ be a rank $r$ vector bundle on $X$. If $\E$ is Ulrich not big then $(X,\O_X(1),\E)$ is one of the following.

If $c_1(\E)^4 = 0$:
\begin{itemize}
\item [(i)] $(\P^4, \O_{\P^4}(1), \O_{\P^4}^{\oplus r})$.
\item [(ii1)] $(X,\O_X(1),\E)\cong(\P^1 \times \P^3, \O_{\P^1}(1) \boxtimes \O_{\P^3}(1), p^*(\O_{\P^3}(1))^{\oplus r})$, where $p=\Phi: \P^1 \times \P^3 \to \P^3$ is the second projection.
\item [(ii2)] $(X,\O_X(1),\E)$ is a linear Ulrich triple with $p = \phi_{\tau} = \widetilde \Phi$ and $b=1$. In particular $\E$ is pull-back of a twisted Ulrich vector bundle on $B$.
\item [(iii)] $(\P^2 \times \P^2, \O_{\P^2}(1) \boxtimes \O_{\P^2}(1), p^*(\O_{\P^2}(2))^{\oplus r})$, where $p: \P^2 \times \P^2 \to \P^2$ is one of the two projections.
\item [(iv)] $(\P^1 \times Q, \O_{\P^1}(1) \boxtimes \O_Q(1), p^*(\mathcal S(1))^{\oplus (\frac{r}{2})})$, where $p : \P^1 \times Q \to Q=Q_3$ is the second projection.
\item [(v1)] $(X,\O_X(1))$ is a hyperplane section of $\P^2 \times \P^3$ under the Segre embedding and $\E \cong q^*(\O_{\P^3}(2))^{\oplus r}$, where $q : X \to \P^3$ is the restriction of the second projection.
\item [(v2)] $(X,\O_X(1),\E)$ is a linear Ulrich triple with $p = \phi_{\tau} = \widetilde \Phi$ and $b=2$.
\item [(v3)] $(X,\O_X(1),\E)$ is a linear Ulrich triple with $p = \widetilde \Phi, b=3$ and $X$ has a morphism to a smooth curve with all fibers $\P^1 \times \P^2$ embedded by the Segre embedding.
\item [(vi1)] $(X,\O_X(1))$ is a del Pezzo fibration over a smooth curve with every fiber the blow-up of $\P^3$ in a point  and $(X,\O_X(1),\E)$ is a linear Ulrich triple with $p = \widetilde \Phi : \P(\F) \to B, b=3, \phi_{\tau} = h \circ p$ where $h : B \to X'$ and $(B, \det \F)$ is a del Pezzo fibration over $X'$ with every fiber $\P^2$.
\item [(vi2)] $(X,\O_X(1))$ is a del Pezzo fibration over a smooth curve with smooth fibers $\P^1 \times \P^1 \times \P^1$, singular fibers $\P^1 \times \mathcal Q$, where $\mathcal Q \subset \P^3$ is a quadric cone and $(X,\O_X(1),\E)$ is a linear Ulrich triple with $p = \widetilde \Phi : \P(\F) \to B, b=3, \phi_{\tau} = h \circ p$ where $h : B \to X'$ and $(B, \det \F)$ is a del Pezzo fibration over $X'$ with smooth fibers $\P^1 \times \P^1$ and singular fibers $\mathcal Q$.
\item [(vi3)] $(X,\O_X(1))$ is a del Pezzo fibration over a smooth curve with smooth fibers $\P(T_{\P^2})$, singular fibers the tautological image of $\P(\F)$, where $\F = \O_{\mathbb F_1}(C_0+f) \oplus \O_{\mathbb F_1}(C_0+2f)$, that is a hyperplane section of the Segre embedding $\P^2 \times \P^2 \subset \P^8$. Also $\phi_{\tau} = h \circ \widetilde \Phi$ where $h : \widetilde{\Phi(X)} \to X'$ is a fibration with general fiber $\P^2$. In particular $\E_{|\P(T_{\P^2})}$ is pull-back of a vector bundle on $\P^2$.
\item [(vii)] $(X,\O_X(1))$ is a quadric fibration with equidimensional fibers over a smooth surface and $(X,\O_X(1),\E)$ is a linear Ulrich triple with $p = \widetilde \Phi, b=3$, $\phi_{\tau}$ factorizes through $\widetilde \Phi$ and every fiber of $\phi_{\tau}$ is a disjoint union of linear spaces.
\item [(viii)] $(X,\O_X(1))$ is a linear $\P^1$-bundle over a smooth threefold, $(X,\O_X(1),\E)$ is a linear Ulrich triple with $p = \phi_{\tau} = \widetilde \Phi$ and $b=3$.
\item [(ix)] $(X,\O_X(1))$ is a scroll over a normal threefold with non-equidimensional fibers and $\phi_{\tau}$ factorizes through $\widetilde \Phi$ via a generically finite degree $1$ map and every fiber of $\phi_{\tau}$ is a disjoint union of linear spaces. In particular, on the general fiber $\P^1$ of $\phi_{\tau}$, we have that $\E_{|\P^1}$ is trivial.
\item [(x1)] $(\P^1 \times M, \O_{\P^1}(1) \boxtimes L, p^*(\G(L)))$, where $M$ is a Fano $3$-fold of index $2$, $K_M=-2L$, $p$ is the second projection and $\G$ is a rank $r$ Ulrich vector bundle for $(M, L)$.
\item [(x2)] $(\P^1 \times \P(T_{\P^2}), \O_{\P^1}(1) \boxtimes \O_{\P(T_{\P^2})}(1), p^*(\G \otimes (\O_{\P^1}(2)\boxtimes \O_{\P^2}(3))))$, where $p :  X \cong \P(\O_{\P^1}(1)\boxtimes T_{\P^2}) \to \P^1 \times \P^2$ is the projection map and $\G$ is a rank $r$ vector bundle on $\P^1 \times \P^2$ such that $H^j(\G \otimes S^k (\O_{\P^1}(-1)\boxtimes \Omega_{\P^2}))=0$ for $j \ge 0, 0 \le k \le 2$.
\item [(x3)] $X$ is a hyperplane section of $\P^2 \times Q_3$ under the Segre embedding, $\O_X(1) = (\O_{\P^2}(1) \boxtimes \O_{Q_3}(1))_{|X}$ and $\E \cong q^*(\mathcal S(2))^{\oplus (\frac{r}{2})}$, where $q : X \to Q_3$ is the restriction of the second projection.
\item [(x4)] $X$ is twice a hyperplane section of $\P^3 \times \P^3$ under the Segre embedding, $\O_X(1) = (\O_{\P^3}(1) \boxtimes \O_{\P^3}(1))_{|X}$ and $\E \cong q^*(\O_{\P^3}(3))^{\oplus r}$, where $q : X \to \P^3$ is the restriction of one of the two projections.
\item [(x5)] $X \cong \P(\mathcal S)$ where $\mathcal S$ is the spinor bundle on $Q_3$, $\O_X(1) \cong \O_{\P(\mathcal S)}(1)\otimes p^*(\O_{Q_3}(1))$ and $\E \cong p^*(\G(3))$, where $p: X \to Q_3$ is the projection and $\G$ is a rank $r$ vector bundle on $Q_3$ such that $H^j(\G(-2k) \otimes S^k \mathcal S)=0$ for $j \ge 0, 0 \le k \le 2$.
\end{itemize}

If $c_1(\E)^4 > 0$:
\begin{itemize}
\item [(xi)] $(Q_4, \O_{Q_4}(1))$ and $\E \cong \mathcal S', \mathcal S'', \mathcal S' \oplus \mathcal S''$, where $\mathcal S', \mathcal S''$ are the spinor bundles on $Q_4$.
\item [(xii)] $(X,\O_X(1))$ is a linear $\P^3$-bundle over a smooth curve $p: X \to B$ and on any fiber $f$ of $p$, $\E_{|f}$ is either $T_{\P^3}(-1)\oplus \O_{\P^3}^{\oplus (r-3)}$ or $\Omega_{\P^3}(2) \oplus \O_{\P^3}^{\oplus (r-3)}$ or $\N(1) \oplus \O_{\P^3}^{\oplus (r-2)}$, where $\N$ is a null-correlation bundle or is a quotient of type $0 \to \O_{\P^3}(-1)^{\oplus 2} \to \O_{\P^3}^{\oplus (r+2)} \to \E_{|f} \to 0$.
\item [(xiii)] $(X,\O_X(1))$ is a linear $\P^2$-bundle over a smooth surface $p: X \to B$ and on any fiber $f$ of $p$, $\E_{|f} \cong T_{\P^2}(-1)\oplus \O_{\P^2}^{\oplus (r-2)}$.
\item [(xiv)] $(X,\O_X(1))$ is a quadric fibration over a smooth curve $p: X \to B$ and $\E$ is a rank $2$ relative spinor bundle, that is $\E_{|f}$ is a spinor bundle on a general fiber $f$ of $p$.
\end{itemize}
Vice versa all $\E$ in (i)-(xi), with the exceptions of (vi3) and (ix), are Ulrich not big.
\end{bthm}

Note that almost all cases in Theorem \ref{main3} are possible, see Examples \ref{altro}-\ref{settimo}. The exceptions are case (vi1) and case (xii) with restriction $\N(1) \oplus \O_{\P^3}^{\oplus (r-2)}$. Moreover, in some cases, namely (vi3) and (ix), the information obtained on $\E$ is not complete, we can only obtain a factorization of $\phi_{\tau}$ through $\widetilde \Phi$.
 
As a consequence of Theorem \ref{main3}, we can classify completely non-big Ulrich vector bundles in the following classes of varieties.

\begin{bcor}
\label{main4}

\hskip 3cm

Let $X \subseteq \P^N$ be a smooth irreducible variety of dimension $2 \le n \le 4$ and let $\E$ be a rank $r$ vector bundle on $X$. We have:
\begin{itemize}
\item [(i)] If $\rho(X) =1$, then $\E$ is Ulrich not big if and only if $(X,\O_X(1),\E)$ is one of the following:
\begin{itemize}
\item [(i1)] $(\P^n, \O_{\P^n}(1), \O_{\P^n}^{\oplus r})$.
\item [(i2)] $(Q_n, \O_{Q_n}(1))$ and $\E$ is one of $(\mathcal S')^{\oplus r}, (\mathcal S'')^{\oplus r}$ for $n=2$, $\mathcal S$ for $n=3$ and as in (xi) of Theorem \ref{main3} for $n=4$.
\end{itemize}
\item [(ii)] If $(X, \O_X(1))$ is a del Pezzo variety, then $\E$ is Ulrich not big if and only if $(X,\E)$ is one of the following:
\begin{itemize}
\item [(ii1)] $(\P^1 \times \P^1 \times \P^1, \O_{\P^1}(1) \boxtimes \O_{\P^1}(1) \boxtimes \O_{\P^1}(1), q^*((\O_{\P^1}(1) \boxtimes \O_{\P^1}(2))^{\oplus s} \oplus (\O_{\P^1}(2) \boxtimes \O_{\P^1}(1))^{\oplus (r-s)})$ for $0 \le s \le r$ and $q : \P^1 \times \P^1 \times \P^1 \to \P^1 \times \P^1$ is one of the three projections.
\item [(ii2)] $(\P(T_{\P^2}),  \O_{\P(T_{\P^2})}(1), p^*(\O_{\P^2}(2))^{\oplus r})$, where $p: \P(T_{\P^2}) \to \P^2$ is one of the two projections.
\item [(ii3)] $(\P^2 \times \P^2, \O_{\P^2}(1) \boxtimes \O_{\P^2}(1), p^*(\O_{\P^2}(2))^{\oplus r})$, where $p: \P^2 \times \P^2 \to \P^2$ is one of the two projections.
\end{itemize}
\item [(iii)] If $(X, \O_X(1))$ is a Mukai variety, then $\E$ is Ulrich not big if and only if $(X,\E)$ is as in (x1)-(x5) of Theorem \ref{main3}.
\end{itemize}
\end{bcor}

Moreover, when $\det \E$ is not big, we can classify Ulrich vector bundles on del Pezzo or Mukai $n$-folds.

\begin{bcor}
\label{main5}

\hskip 3cm

Let $X \subseteq \P^N$ be a smooth irreducible variety of dimension $n \ge 2$ and let $\E$ be a rank $r$ vector bundle on $X$ such that $\det \E$ is not big. We have:
\begin{itemize}
\item [(i)] If $(X, \O_X(1))$ is a del Pezzo $n$-fold, then $\E$ is Ulrich if and only if $(X,\E)$ is as in (ii1)-(ii3) of Corollary \ref{main4}.
\end{itemize}
\begin{itemize}
\item [(ii)] If $(X, \O_X(1))$ is a Mukai $n$-fold, then $\E$ is Ulrich if and only if $(X,\E)$ is either as in (x1)-(x5) of Theorem \ref{main3} or is one of the following:
\begin{itemize}
\item [(ii1)] $(\P^3 \times \P^3, p^*(\O_{\P^3}(3))^{\oplus r})$, where $p: \P^3 \times \P^3 \to \P^3$ is one of the two projections.
\item [(ii2)] $(\P(T_{\P^3}), p^*(\O_{\P^3}(3))^{\oplus r})$, where $p: \P(T_{\P^3}) \to \P^3$ is one of the two projections.
\item [(ii3)] $(\P^2 \times Q_3, q^*(\mathcal S(2))^{\oplus (\frac{r}{2})})$, where $q : \P^2 \times Q_3 \to Q_3$ is the second projection.
\end{itemize}
\end{itemize}
\end{bcor}

\section{Notation and standard facts about (Ulrich) vector bundles}

Throughout this section we will let $X \subseteq \P^N$ be a smooth irreducible complex variety of dimension $n \ge 1$, degree $d$ and $H$ a hyperplane divisor on $X$.

\begin{defi}
We say that $(X, \O_X(1))$ as above is a {\it linear $\P^k$-bundle} over a smooth variety $B$ if $(X, \O_X(1))=(\P(\F), \O_{\P(\F)}(1))$, where $\F$ is a very ample vector bundle on $B$ of rank $k+1$.

We say that $(X, \O_X(1))$ as above is a {\it scroll (respectively a quadric fibration; respectively a del Pezzo fibration) over a normal variety $Y$ of dimension $m$} if there exists a surjective morphism with connected fibers  $\phi: X \to Y$ such that $K_X+(n-m+1)H=\phi^*\L$ (respectively $K_X+(n-m)H=\phi^*\L$; respectively $K_X+(n-m-1)H=\phi^*\L$), with $\L$ ample on $Y$.
\end{defi}

\begin{defi}
\label{spin}
For $n \ge 2$ we let $Q_n \subset \P^{n+1}$ be a smooth quadric. We let $S$ ($n$ odd), and $S', S''$ ($n$ even), be the vector bundles on $Q_n$, as defined in \cite[Def.~1.3]{o}. The {\it spinor bundles} on $Q_n$ are ${\mathcal S}={\mathcal S}_n = S(1)$ if $n$ is odd and ${\mathcal S}'={\mathcal S}'_n = S'(1)$, ${\mathcal S}''={\mathcal S}''_n = S''(1)$, if $n$ is even. They all have rank $2^{\lfloor \frac{n-1}{2} \rfloor}$.
\end{defi}

\begin{notation}
For $k \in \Z : 1 \le k \le n$ we denote by $F_k(X)$ the Fano variety of $k$-dimensional linear subspaces of $\P^N$ that are contained in $X$. For $x \in X$, we denote by $F_k(X,x) \subset F_k(X)$ the subvariety of $k$-dimensional linear subspaces passing through $x$.
\end{notation}

The following fact is well known (see for example \cite[Prop.s~2.2.1 and 2.3.9]{ru}) and will be often used without mentioning.
\begin{remark}
\label{russo}
Let $x \in X$ be a general point. Then $F_1(X,x)$ is smooth and $\dim_{[L]} F_1(X,x) = - K_X \cdot L -2$ for every $[L] \in F_1(X,x)$.
\end{remark}

\begin{defi}
Given a nef line bundle $\L$ on $X$ we denote by
$$\nu(\L) = \max\{k \ge 0: c_1(\L)^k \ne 0\}$$ 
the {\it numerical dimension} of $\L$.
\end{defi}
Recall that when $\L$ is globally generated $\nu(\L)$ is the dimension of the image of the morphism induced by $\L$.

\begin{defi}
\label{not}
Let $\E$ be a rank $r$ vector bundle on $X$. We denote by $c(\E)$ its Chern polynomial and by $s(\E)$ its Segre polynomial. We set $\P(\E) = \Proj(\Sym(\E))$ with projection map $\pi : \P(\E) \to X$ and tautological line bundle $\O_{\P(\E)}(1)$. We say that $\E$ is {\it nef (big, ample, very ample)} if $\O_{\P(\E)}(1)$ is nef (big, ample, very ample). If $\E$ is nef, we define the {\it numerical dimension} of $\E$ by $\nu(\E):=  \nu(\O_{\P(\E)}(1))$. When $\E$ is globally generated we define the map determined by $\E$ as
$$\Phi=\Phi_{\E} : X \to {\mathbb G}(r-1, \P H^0(\E)).$$
For any point $x \in X$ we will denote the fiber of $\Phi$ by
$$F_x = \Phi^{-1}(\Phi(x))$$ 
and we set $\phi(\E)$ for the dimension of the general fiber of $\Phi_{\E}$. Moreover, we set
$$\varphi=\varphi_{\E} = \varphi_{\O_{\P(\E)}(1)} : \P(\E) \to \P H^0(\E)$$
$$\Pi_y = \pi(\varphi^{-1}(y)), y \in \varphi(\P(\E))$$
and  
$$P_x = \varphi(\P(\E_x)).$$
\end{defi}
Note that $\Phi(x)= [P_x]$ is the point in  ${\mathbb G}(r-1,\P H^0(\E))$ corresponding to $P_x$.

We recall that, considering the map
$$\lambda_{\E} : \Lambda^r H^0(\E) \to H^0(\det \E)$$
one gets a commutative diagram
\begin{equation}
\label{muk}
\xymatrix{X \ar[d]^{\varphi_{|\Im \lambda_{\E}|}} \ar[r]^{\hskip -.9cm \Phi_{\E}} &  {\mathbb G}(r-1,\P H^0(\E)) \ar@{^{(}->}[d]^{P_{\E}} \\ \P \Im \lambda_{\E} \ar@{^{(}->}[r] & \P \Lambda^r H^0(\E)}
\end{equation}
where $P_{\E}$ is the Pl\"ucker embedding. In particular this implies that if $c_1(\E)^n=0$, then $\dim F_x \ge 1$ for every $x \in X$. We will often use this fact without further mentioning.

\begin{defi}
Let $\E$ be a vector bundle on $X \subseteq \P^N$. We say that $\E$ is an {\it Ulrich vector bundle} if $H^i(\E(-p))=0$ for all $i \ge 0$ and $1 \le p \le n$.
\end{defi}

The following properties will be often used without mentioning.

\begin{remark}
\label{gen}
Let $\E$ be a rank $r$ Ulrich vector bundle on $X \subseteq \P^N$ and let $d = \deg X$. Then
\begin{itemize}
\item [(i)] $\E$ is $0$-regular in the sense of Castelnuovo-Mumford, hence $\E$  is globally generated (by \cite[Thm.~1.8.5]{laz1}).
\item [(ii)] $h^0(\E)=rd$ (by \cite[Prop.~2.1]{es} or \cite[(3.1)]{b1}).
\item [(iii)] $\E$ is arithmetically Cohen-Macaulay (ACM), that is $H^i(\E(j))=0$ for $0 < i <n$ and all $j \in \Z$ (by \cite[Prop.~2.1]{es} or \cite[(3.1)]{b1}).
\item [(iv)] $\E_{|Y}$ is Ulrich on a smooth hyperplane section $Y$ of $X$ (by \cite[(3.4)]{b1}).
\end{itemize}
\end{remark}

\begin{remark}
\label{kno1}
On $(\P^n, \O_{\P^n}(1))$ the only rank $r$ Ulrich vector bundle is $\O_{\P^n}^{\oplus r}$ by \cite[Prop.~2.1]{es}, \cite[Thm.~2.3]{b1}. 
\end{remark}

We also collect here some properties that follow by \cite[Thm.~2]{ls}.

\begin{lemma}
\label{not2}
Let $\E$ be an Ulrich vector bundle on $X$. Then $F_x$ is a linear space contained in $X \subseteq \P^N$ for every $x \in X$. Moreover if 
\begin{equation}
\label{diag}
\xymatrix{X \ar[dr]_{\Phi} \ar[r]^{\hskip -.3cm \widetilde \Phi} & \widetilde{\Phi(X)} \ar[d]^g \\ & \Phi(X)}
\end{equation}
is the Stein factorization of $\Phi = \Phi_{\E}$, then, for every $x \in X$,
$$\widetilde F_x := \widetilde \Phi^{-1}(\widetilde \Phi(x))=F_x$$ 
and there is a vector bundle $\H$ on $\widetilde{\Phi(X)}$ such that $\E \cong \widetilde \Phi^*\H$.
\end{lemma}
\begin{proof}
For every $x \in X$ we have that $F_x$ is a linear space by \cite[Thm.~2]{ls}. This implies that $g$ is bijective and $\widetilde F_x =F_x$. As is well known, there is a rank $r$ vector bundle $\U$ on $\Phi(X)$ such that $\E \cong \Phi^* \U$ and therefore also $\E \cong \widetilde \Phi^*\H$ with $\H = g^*\U$.
\end{proof}

\begin{defi}
\label{not4}
Let $\E$ be a vector bundle on $X$. We say that $(X,\O_X(1),\E)$ is a {\it linear Ulrich triple} if there are a smooth irreducible variety $B$ of dimension $b \ge 1$, a very ample vector bundle $\F$ and a rank $r$ vector bundle $\G$ on $B$ such that 
\begin{equation}
\label{trip}
(X,\O_X(1),\E)=(\P(\F), \O_{\P(\F)}(1), p^*(\G(\det \F)))
\end{equation}
where $p: X \cong \P(\F) \to B$ is the projection and 
\begin{equation}
\label{van-trip}
H^j(\G \otimes S^k \F^*)=0 \ \hbox{for all} \ j \ge 0, 0 \le k \le b-1.
\end{equation}
\end{defi}
Note that when $(X,\O_X(1),\E)$ is a linear Ulrich triple, then $\E$ is an Ulrich vector bundle on $X$ by \cite[Lemma 4.1]{lo}. Moreover observe that, in the case $b=1$, we have $H^j(\G)=0$ for all $j \ge 0$, hence $\G \otimes \L$ is an Ulrich vector bundle on $B$ for any very ample line bundle $\L$. Thus, in this case, $\E$ is pull-back of a twisted Ulrich vector bundle on $B$.

\begin{lemma}
\label{ii}
Let $\E$ be an Ulrich vector bundle on $X$ such that $\widetilde \Phi$ has equidimensional fibers and suppose that $(X,\O_X(1))\ne (\P^n, \O_{\P^n}(1))$. Then $(X,\O_X(1),\E)$ is a linear Ulrich triple with $p = \widetilde \Phi, B=\widetilde{\Phi(X)}$ and $b = n-\phi(\E)$.
\end{lemma}
\begin{proof}
This easily follows by \cite[Prop.~3.2.1]{bs2}, Lemma \ref{not2} and \cite[Lemma 4.1]{lo}. 
\end{proof}

In several cases we will study Ulrich vector bundles on a variety $X$ that has some standard structure morphism. As it will be clear in the sequel, this will naturally distinguish two different cases. We will use the following 
\begin{notation}
\label{not3}
Given a morphism $h : X \to X'$ we set 
$$f_x = h^{-1}(h(x)), x \in X.$$ 
\end{notation}
Given an Ulrich vector bundle $\E$ on $X$, we have $x \in F_x \cap f_x$. Now

\begin{lemma} (Dichotomy Lemma) 
\label{dico}

Let $\E$ be an Ulrich vector bundle on $X$ such that $c_1(\E)^n=0$ and let $h : X \to X'$ be a morphism. Then only one of the following cases occurs for $h$:
\begin{itemize}
\item[(fin)] $\dim F_x \cap f_x = 0$ for every $x \in X$, or
\item[(fact)] $F_x \subseteq f_x, h$ factorizes through $\widetilde \Phi$ and $f_x$ is a disjoint union of linear spaces, fibers of $\Phi$, for every $x \in X$.
\end{itemize} 
\end{lemma}
\begin{proof}
Since $c_1(\E)^n=0$, it follows that $\dim F_x \ge 1$ for every $x \in X$. Now suppose that there is an $x_0 \in X$ such that $\dim(F_{x_0} \cap f_{x_0}) \ge 1$. Then the morphism 
$$h_{|F_{x_0}} : F_{x_0} \to X'$$ 
has a positive dimensional fiber, namely $F_{x_0} \cap f_{x_0}$. Since $F_{x_0} = \P^k$ it follows that $h_{|F_{x_0}}$ is constant, that is $F_{x_0} \subseteq f_{x_0}$. Now $\widetilde F_{x_0} = F_{x_0} \subset f_{x_0}$, hence \cite[Lemma 1.15(a)]{de} implies that $F_x = \widetilde F_x \subset f_x$ for general $x$. Then, by semicontinuity, $\dim_x(F_x \cap f_x) \ge 1$ for every $x \in X$. Again $h_{|F_x}$ is constant, that is $\widetilde F_x = F_x \subseteq f_x$. Therefore $h$ factorizes through $\widetilde \Phi$ by \cite[Lemma 1.15(b)]{de}. Since the fibers of $\widetilde \Phi$ are linear spaces, we get that $f_x$ is a disjoint union of linear spaces for every $x \in X$.
\end{proof}
\begin{defi}
Let $\E$ be an Ulrich vector bundle on $X$. Given a morphism $h : X \to X'$, we define the subcase
$$({\rm emb}) \ \ F_x \cap f_x = \{x\} \ {\rm scheme-theoretically, for \ every} \ x \in X.$$
We will say, that {\it case (emb) (or (fin), or (fact)) holds for $h$}, referring to the above or to the Dichotomy Lemma. 
\end{defi}
We now analyze how the cases (fin), (emb) or (fact) occur in some special cases.

\begin{lemma}
\label{fe}
Let $\E$ be an Ulrich vector bundle on $X$ such that $c_1(\E)^n=0$ and let $h : X \to X'$ be a morphism. We have:
\begin{itemize}
\item[(i)] If case {\rm (fin)} holds for $h$ and $\dim X'=1$, then $F_x$ is a line for every $x \in X$ and $X' \cong \P^1$. 
\item[(ii)] If $h$ is a linear $\P^k$-bundle or a quadric fibration over a smooth curve and case {\rm (fin)} holds for $h$, then case {\rm (emb)} holds for $h$.
\item[(iii)] Assume that case {\rm (fact)} holds for $h$. If $f_x$ is integral and $\Pic(f_x) \cong \Z$ for some $x \in X$, then $F_x=f_x$. Moreover if $h_*\O_X \cong \O_{X'}, f_x$ is integral and $\Pic(f_x) \cong \Z$ for every $x \in X$, then $h=\widetilde \Phi$.
\end{itemize} 
\end{lemma}
\begin{proof}
Since $c_1(\E)^n=0$ we have that $F_x = \P^k, k \ge 1$ for every $x \in X$. Suppose that case (fin) holds in one of (i)-(iii). To see (i), if $\dim F_x \ge 2$, then $h_{|F_x} : F_x=\P^k \to X'$ is constant, contradicting case (fin). Then $F_x$ is a line and, for every $v \in F_x$, we have that $\dim f_v \cap F_x = \dim f_v \cap F_v = 0$, so that $F_x$ dominates $X'$ and then $X' \cong \P^1$. This proves (i). Now (ii) is clear if $h$ is a linear $\P^k$-bundle. If $h$ is a quadric fibration over a smooth curve, let $x \in X$, so that $F_x$ is a line by (i). If $F_x \cdot f_x \ge 2$, then, for a general $x' \in X$, $F_x \cdot f_{x'} = F_x \cdot f_x \ge 2$. This implies that $x \in F_x \subset \langle f_{x'} \rangle = \P^n$. But then $X = \P^n$, a contradiction. Hence $f_x \cdot F_x=1$ for every $x \in X$ and case (emb) holds. This proves (ii). To see (iii), assume that $f_x$ is integral and $\Pic(f_x) \cong \Z$ for some $x \in X$. Since $\Phi_{|f_x} : f_x \to {\mathbb G}(r-1,\P H^0(\E))$ contracts $F_x$ to a point, it must be constant, thus $f_x = F_x = \widetilde F_x$. Now if $f_x$ is integral and $\Pic(f_x) \cong \Z$ for every $x \in X$, then $f_x = F_x = \widetilde F_x$ for every $x \in X$, hence also $\widetilde \Phi$ factorizes through $h$ by \cite[Lemma 1.15(b)]{de} and we deduce that $h=\widetilde \Phi$.
\end{proof}

The following general results, applied to some standard cases arising in adjunction theory, illustrate the power of the Dichotomy Lemma. 

\begin{prop}
\label{qf} 
Let $\E$ be an Ulrich vector bundle on $X$ such that $c_1(\E)^n=0$ and suppose that $n \ge 4$. Let $h : X \to X'$ be a quadric fibration over a smooth curve. Then $(X,\O_X(1),\E)=(\P^1 \times Q, \O_{\P^1}(1) \boxtimes \O_Q(1), p^*(\mathcal \H(1)))$, where $p : \P^1 \times Q \to Q=Q_{n-1}$ is the second projection and $\H$ is a direct sum of spinor bundles on $Q$. 
\end{prop}
\begin{proof}
Let $H_Q \in |\O_Q(1)|$. By hypothesis $f_x \cong Q$ and $H_{|f_x} \cong H_Q$ for general $x$. Case (fact) does not hold for $h$, since otherwise Lemma \ref{fe}(iii) would give the contradiction $\P^k= F_x=f_x$. Now the Dichotomy Lemma, Lemma \ref{fe}(i) and (ii) give that $F_x$ is a line for every $x \in X$ and case (emb) holds for $h$. It follows by Lemma \ref{ii} that $(X,\O_X(1),\E)$ is a linear Ulrich triple with $p = \widetilde \Phi, B=\widetilde{\Phi(X)}$ and $b = n-1$. Moreover case (emb) implies that, for general $x$, there is a closed embedding $p_{|f_x} : f_x \to B$ and therefore $B \cong Q$. Hence $p_{|f_x}$ is an isomorphism and $(p^*H_Q)_{|f_x} \cong H_{|f_x}$. Set $\det \F = \O_Q(c)$ for some $c \in \Z$. Then
$$(1-n) H_Q = K_Q = K_{f_x} = (K_X+f_x)_{|f_x} = (-2H+(1-n+c)p^*H_Q)_{|f_x} = (c-n-1)H_Q$$
so that $c=2$. Therefore $\F$ is a very ample rank $2$ vector bundle on $Q$ with $\det \F=2H_Q$ and it follows that $\F \cong \O_Q(1)^{\oplus 2}$ (see for example \cite[Prop.~1.2]{aw}). Therefore $(X,\O_X(1))=(\P^1 \times Q, \O_{\P^1}(1) \boxtimes \O_Q(1))$. Now $\E \cong p^*(\G(2))$ where $\G$ is a rank $r$ vector bundle on $Q$ such that $H^j(\G \otimes S^k \F^*)=0$ for $j \ge 0, 0 \le k \le n-2$. Hence $\H:=\G(1)$ is an Ulrich vector bundle on $Q$ and we get that $\E \cong p^*(\H(1))$, where $p : \P^1 \times Q \to Q$ is the second projection and $\H$ is a direct sum of spinor bundles on $Q$ by \cite[Lemma 3.2(iv)]{lms}.
\end{proof}

\begin{prop}
\label{caso(e)} 
Let $\E$ be an Ulrich vector bundle on $X$ such that $c_1(\E)^n=0$ and suppose that $n \ge 2$. Then $(X, \O_X(1))$ is not a blow-up $h: X \to X_1$ of a smooth $n$-fold at a point with exceptional divisor $E$ such that $\O_X(1)_{|E} \cong \O_{\P^{n-1}}(1)$.
\end{prop}
\begin{proof}
Assume that $(X, \O_X(1))$ is a blow-up as stated and apply the Dichotomy Lemma to $h$. Since $f_x = \{x\}$ for general $x$, we get that case (fact) does not hold for $h$, so that we are in case (fin). Moreover $f_u$ is a linear space for every $u \in X$, hence (emb) holds for $h$. It follows that $\Phi_{|E} : E \to \Phi(X)$ is a closed embedding. On the other hand, we know that $\dim \Phi(X) \le n-1$ and therefore $\P^{n-1} \cong E \cong \Phi(E)=\Phi(X)$. Hence for every $u \in X$ we have that $F_u=F_{u_0}$ for some $u_0 \in E$. Then $E=f_{u_0}$ and $F_{u_0} \cup f_{u_0} \subset T_{u_0}X$. If $\dim F_u \ge 2$ then $\dim F_{u_0} \ge 2$ and $\dim T_{u_0}X \ge \dim \langle F_{u_0} \cup f_{u_0} \rangle \ge n+1$, contradicting the smoothness of $X$. Therefore $F_u$ is a line for every $u \in X$. It follows that $\Phi : X \to \P^{n-1}$ has equidimensional fibers and \cite[Prop.~3.2.1]{bs2} implies that $\Phi$ is a linear $\P^1$-bundle. Hence there is a very ample rank $2$ vector bundle $\F$ on $\P^{n-1}$ such that $(X, \O_X(1)) \cong (\P(\F), \O_{\P(\F)}(1)), \Phi$ is the bundle projection $p : \P(\F) \to \P^{n-1}$ and $\E \cong p^*(\G(\det \F))$ for some rank $r$ vector bundle $\G$ on $\P^{n-1}$. Let $R=p^*(\O_{\P^{n-1}}(1))$ and $\det \F = \O_{\P^{n-1}}(c)$, for some $c \in \Z$. Then
$$K_E = (K_X+E)_{|E} = (-2H+(c-n)R-H)_{|E}$$
that is $\O_{\P^{n-1}}(-n) \cong \O_{\P^{n-1}}(c-n-3)$, so that $c=3$. Therefore, since $\F$ is very ample, it has splitting type $(1,2)$ on any line in $\P^{n-1}$ and it follows by \cite[Thm.]{v} that either $\F \cong  \O_{\P^{n-1}}(1) \oplus \O_{\P^{n-1}}(2)$ or $n=3$ and $\F \cong T_{\P^2}$. The latter case is excluded since $\P(T_{\P^2})$ does not contain linear $\P^2$'s. Therefore we are in the first case and \cite[Lemma 4.1]{lo} gives in particular that $H^i(\G(-s)) = 0$ for all $i \ge 0$ and $0 \le s \le n-2$. Hence $\G(1)$ is an Ulrich vector bundle for $(\P^{n-1}, \O_{\P^{n-1}}(1))$, so that $\G \cong \O_{\P^{n-1}}(-1)^{\oplus r}$ by Remark \ref{kno1}. But this gives the contradiction $0 =  H^{n-1}(\G(-n+1)) = H^{n-1}(\O_{\P^{n-1}}(-n)^{\oplus r}) \ne 0$.
\end{proof}

\section{Non-big Ulrich vector bundles on fourfolds}

In this section we will prove Theorem \ref{main3}, Corollary \ref{main4} and Corollary \ref{main5}. 

One guide will be given by the following.

\subsection{Fourfolds and adjunction theory}
\label{adj}

We collect some definitions and standard facts in adjunction theory, that we recall for completeness' sake. 

Let $X$ be a smooth irreducible variety such that $K_X$ is not nef and let $H$ be a very ample divisor. Consider the nef value of $(X,H)$ (see \cite[Def.~1.5.3]{bs2})
$$\tau = \tau(X,H) = \min\{t \in \R : K_X+tH \ \hbox{is nef}\}$$
and the nef value morphism, defined for $m \gg 0$ by
$$\phi_{\tau} =  \phi_{\tau}(X,H):= \varphi_{m(K_X+\tau H)} : X \to X'.$$ 
We recall that $(\phi_{\tau})_*\O_X \cong \O_{X'}$, see \cite[Def.~1.5.3]{bs2}.

\begin{lemma}
\label{aggiu}
Let $X \subseteq \P^N$ be a smooth irreducible fourfold covered by lines and let $H \in |\O_X(1)|$. Let $\tau$ be the nef value of $(X,H)$ and let $\phi_{\tau}$ be the nef value morphism. Then $(X,\O_X(1))$ is only one of the following:
\begin{itemize}
\item [(a)] $(\P^4, \O_{\P^4}(1))$.
\item [(b.1)] $(Q_4, \O_{Q_4}(1))$.
\item [(b.2)] A linear $\P^3$-bundle under $\phi_{\tau} : X \to X'$ over a smooth curve with $\tau=4$.
\item [(c.1)] A del Pezzo $4$-fold, that is $K_X=-3H$.
\item [(c.2)] A quadric fibration under $\phi_{\tau} : X \to X'$ over a smooth curve with $\tau=3$. 
\item [(c.3)] A linear $\P^2$-bundle under $\phi_{\tau} : X \to X'$ over a smooth surface with $\tau=3$.
\item [(d.1)] A Mukai variety, that is $K_X=-2H$.
\item [(d.2)] A del Pezzo fibration under $\phi_{\tau} : X \to X'$  over a smooth curve with $\tau=2$.
\item [(d.3)] A quadric fibration with equidimensional fibers under $\phi_{\tau} : X \to X'$ over a smooth surface with $\tau=2$.
\item [(d.4)] A linear $\P^1$-bundle under $\phi_{\tau} : X \to X'$ over a smooth threefold with $\tau=2$.
\item [(d.5)] A scroll under $\phi_{\tau} : X \to X'$ over a normal threefold with non-equidimensional fibers with $\tau=2$.
\item [(e)] The blow-up $\phi_{\tau} : X \to X'$ of a smooth fourfold at $t \ge 1$ points, with exceptional divisors $E_i \cong \P^3$ such that $H_{|E_i} \cong \O_{\P^3}(1), 1 \le i \le t$ and $\tau=3$.
\end{itemize}
\end{lemma}
\begin{proof}
Let $x \in X$ be a general point and let $L \in F_1(X,x)$. Then
$$0 \le \dim_{[L]} F_1(X,x) = -K_X \cdot L - 2$$
so that $K_X \cdot L \le - 2$ and $\tau \ge 2$. By \cite[Prop.~7.2.2]{bs2} we have that either we are in case (a), or $\tau=4$ and we are in cases (b.1) or (b.2) or $\tau \le 4$ and $K_X+4H$ is big and nef. In the latter case $K_X+4H$ is ample by \cite[Prop.~7.2.3]{bs2} and $\tau \le 3$ by \cite[Prop.~7.2.4]{bs2}. Moreover \cite[Prop.~7.3.2]{bs2} gives that $K_X+3H$ is ample unless $\tau = 3$ and either we are in one of the cases (c.1)-(c.3) or (e) (for (c.3) use also \cite[Thm.~0.2]{sv} and for (e) use also \cite[Thm.~0.3 and Rmk.~1]{sv}). Next if $K_X+3H$ is ample then $\tau < 3$ and $(X,\O_X(1))$ is isomorphic to its first reduction (see \cite[Def. 7.3.3]{bs2}). Therefore \cite[Prop.~7.3.4]{bs2} implies that $\tau = 2$, so that $K_X \cdot L = - 2$, hence $K_X+2H$ is nef and not big. It follows by \cite[Prop.~7.5.3 and Thm.~14.2.3]{bs2} that either we are in one of the cases (d.1)-(d.3) or $(X,\O_X(1))$ is a scroll under $\phi_{\tau} : X \to X'$ over a normal threefold. Finally in the latter case if $\phi_{\tau}$ has equidimensional fibers, then we are in case (d.4) by \cite[Prop.~3.2.1]{bs2}, otherwise we are in case (d.5).
\end{proof}

\subsection{Proofs}

To this end we will use the notation in \ref{not3} and in Definition \ref{not}.

\renewcommand{\proofname}{Proof of Theorem \ref{main3}}
\begin{proof}
Suppose that $(X,\O_X(1),\E)$ is as in (i)-(xi). Then, using, in the corresponding cases, Remark \ref{kno1}, \cite[Prop.~3.3(iii)]{lms}, \cite[Lemma 4.1]{lo}, \cite[(3.5)]{b1} and Remark \ref{gen}(iv), we see that $\E$ is Ulrich not big.

Now assume that $\E$ is Ulrich not big.

It follows by \cite[Thm.~1]{lo} that $X$ is covered by lines, hence Lemma \ref{aggiu} gives that $(X,\O_X(1))$ belongs to one of the cases (a)-(e) in Lemma \ref{aggiu}. We will divide the proof according to these cases.

If $(X,\O_X(1))$ is as in (a), we are in case (i) by Remark \ref{kno1}. 

If $(X,\O_X(1))$ is as in (b.1), we are in case (xi) by \cite[Prop.~3.3(iii)]{lms}.

Therefore we can assume from now on that $(X,\O_X(1))$ is neither as in (a) nor as in (b.1). 

For the rest of the proof $x \in X$ will denote a general point. 

We will now divide the proof into several subcases and claims.

\noindent {\bf Case (A)}: $c_1(\E)^4 > 0$.

We claim that
\begin{equation}
\label{ellesse}
\dim F_1(X,x) \ge r+3-\nu(\E) \ge 1.
\end{equation}
In fact, set $k=r+3-\nu(\E)$. It follows from \cite[Cor.~2]{ls} that we can find a $1$-dimensional family $T$ of $k$-dimensional linear spaces $M_t, t \in T$, with $x \in M_t \subseteq X$. Consider the incidence correspondence 
$$\mathcal J =\{([L], t) \in F_1(X,x) \times T : L \subseteq M_t\}.$$ 
The second projection shows that $\mathcal J$ is irreducible of dimension $k$. Next, let $B=\bigcap_{t \in T}M_t$. Since $B \subseteq M_t$ for every $t \in T$ and $\dim T=1$, then $\dim B < k$. On the other hand, $x \in B$, hence choosing a point $x' \in M_t \setminus B$, we can find a line $L_0 =  \langle x, x' \rangle$, with $[L_0] \in F_1(X,x)$ and such that $L_0 \not\subseteq B$. Therefore, since $\dim T=1$, there are finitely many $t \in T$ such that $L_0 \subset M_t$. Thus, the first projection $p: \mathcal J \to F_1(X,x)$ has $0$-dimensional general fibers over $p(\mathcal J)$. Therefore $\dim F_1(X,x) \ge \dim p(\mathcal J)=k$. This proves \eqref{ellesse}. 

Now we study the cases (b.2) and (c.3). To unify notation, we denote $\phi_{\tau}$ by $p:X \to B$. First, observe that \cite[Thm.~1.4]{lp} and \eqref{ellesse} give that
\begin{equation}
\label{acc}
r+1 \le \nu(\E) \le r+2 \ \hbox{in case (b.2) and} \ \nu(\E) = r+2  \ \hbox{in case (c.3)}.
\end{equation}
\renewcommand{\proofname}{Proof}
We first show that we can apply \cite[Lemma 4.4]{lms}, that we recall here for the reader's sake.

\begin{lemma}
\label{nu=b+r}
Let $(X,\O_X(1))=(\P(\F), \O_{\P(\F)}(1))$, where $\F$ is a rank $n-b+1$ very ample vector bundle over a smooth irreducible variety $B$ of dimension $b$ with $1 \le b \le n-1$. Let $\E$ be a rank $r$ Ulrich vector bundle on $X$, let $p : X \to B$ be the projection morphism and suppose that
\begin{equation}
\label{cos}
p_{|\Pi_y} : \Pi_y \to B \ \hbox{is constant for every} \ y \in \varphi(\P(\E)).
\end{equation}
Then, for every fiber $f$ of $p$ we have $\nu(\E)=b+\dim \varphi(\pi^{-1}(f)) \ge b+r-1$. Moreover we have the following two extremal cases:
\begin{itemize}
\item [(i)] If $\nu(\E) = b+r-1$ there is a rank $r$ vector bundle $\G$ on $B$ such that $\E \cong p^*(\G(\det \F))$ and
$H^j(\G \otimes S^k \F^*)=0$ for all $j \ge 0, 0 \le k \le b-1$. 
\item [(ii)] If $\nu(\E) = b+r$ then either $b=n-1$ and $\E$ is big or $b \le n-2$ and $\E_{|f} \cong T_{\P^{n-b}}(-1)\oplus \O_{\P^{n-b}}^{\oplus (r-n+b)}$ for any fiber $f=\P^{n-b}$ of $p$.
\end{itemize}
\end{lemma}

\begin{claim} 
\label{sato}
In cases (b.2) and (c.3) we have that \eqref{cos} holds for $p:X \to B$.
\end{claim}
\begin{proof} 
If $\nu(\E) = r+1$ we are in case (b.2) and for every $y \in \varphi(\P(\E))$ we get that 
$$\dim \Pi_y \ge r+3-\nu(\E) = 2.$$ 
Hence  \eqref{cos} holds for $p : X \to B$. 

Therefore we can assume that $\nu(\E) = r+2$ by \eqref{acc}. 

Arguing by contradiction assume that there is a $y_0 \in \varphi(\P(\E))$ such that $\Pi_{y_0}$ is not contained in a fiber of $p$. Then the same holds for a general $y \in \varphi(\P(\E))$: In fact, if $\Pi_y$ is contained in a fiber of $p$, then, by specialization $\dim \Pi_{y_0} \cap f_{x_0} \ge \dim \Pi_y \cap f_x = 1$, where $y_0 \in P_{x_0}$ and $y \in P_x$. Now $\Pi_{y_0}$ is a linear space of positive dimension by \cite[Thm.~2]{ls} and $p_{|\Pi_{y_0}} : \Pi_{y_0} \to B$ contracts $\Pi_{y_0} \cap f_{x_0}$ to a point, hence is constant, that is $\Pi_{y_0} \subseteq f_{x_0}$, a contradiction. Therefore $\Pi_y \not\subset f_x$ for a general $y \in \varphi(\P(\E))$. Then $F_1(X,x)$ has at least two irreducible components, namely $W$ made of lines in $f_x$ through $x$ and $W'$ made of lines of type $\Pi_y$, with $\dim W' \ge 1$ by \cite[Cor.~2]{ls}. Moreover $F_1(X,x)$ is smooth, hence $W \cap W' = \emptyset$. As is well known, $F_1(X,x) \subset \P^3 = \P(T_xX)$. In case (b.2) we have that $W$ is a plane, thus giving a contradiction. In case (c.3) we have that $W$ is a line and $W'$ is a curve. We now use the fact that $F_1(X,x)$ is contained in the base locus of the linear system $|II|$ given by the second fundamental form of $X$ at $x$. Observe that $X$ is not defective by \cite[Thm.~3.3(b)]{ei}, hence \cite[Thm.~7.3(i)]{lan} gives that there is a smooth quadric $Q \in |II|$. Since $W \sqcup W' \subset Q$ we find that $W'$ is a union of lines. But $X \subset \P^N$ and a line $L$ component of $W'$ is also a line in ${\mathbb G}(1,N)$, hence the union of the lines representing points of $L$ gives a plane $M_x$ such that $x \in M_x \subset X$. It follows by \cite[Main Thm.]{sa2} that $X$ is a linear $\P^k$-bundle $p' : X \to B'$ over a smooth $B'$ so that its general fibers contain the $M_x$'s. On the other hand, it cannot be that $k \ge 3$ for otherwise on any $\P^k$ we would have that $p_{|\P^k} : \P^k \to B$ is constant, thus giving the contradiction that $p$ has a fiber of dimension $k$. Therefore $k=2$ and, by construction, the $\P^2$-bundle structure $p'$ is different from $p$. Now on any fiber $f'$ of $p'$ we have that $p_{|f'} : f' \to B$ cannot be constant, otherwise $f'$ is also a fiber of $p$, and it follows by \cite[Thm.~4.1]{laz2} that $B \cong \P^2$. Similarly, $B' \cong \P^2$. But then \cite[Thm.~A]{sa1} implies that $(X,\O_X(1))=(\P^2 \times \P^2, \O_{\P^2}(1) \boxtimes \O_{\P^2}(1))$, a contradiction.
\end{proof} 
\renewcommand{\proofname}{Proof}

\begin{claim} 
\label{p3bdle}
In case (b.2) we have that $(X, \O_X(1), \E)$ is as in (xii).
\end{claim}
\begin{proof}
Note that  \eqref{cos} holds for $p : X \to B$ by Claim \ref{sato}.

If $\nu(\E) = r+1$ we can apply Lemma \ref{nu=b+r} and we are in case (xii).

By \eqref{acc} it remains to study the case $\nu(\E) = r+2$.
 
Then, Lemma \ref{nu=b+r} gives 
$$\dim \varphi(\pi^{-1}(f)) = \nu(\E)-1=r+1$$ 
for every fiber $f$ of $p$. Next note that the morphism $\Phi_{|f} : f=\P^3 \to {\mathbb G}(r-1, \P H^0(\E))$ cannot be constant, for otherwise $P_x=P_{x'}$ for any $x, x' \in f$, giving the contradiction $\varphi(\pi^{-1}(f)) = P_x = \P^{r-1}$.  Hence $\Phi_{|f}$ is finite onto its image and therefore $\varphi(\pi^{-1}(f)) \subseteq \P H^0(\E)$ is swept out by a family $\{P_x, x \in f\}$ of dimension $3$ of linear $\P^{r-1}$'s. We will now show, along the lines of \cite[Lemma p.~44]{se}, that 
\begin{equation}
\label{sev}
\hbox{either} \ \varphi(\pi^{-1}(f))=\P^{r+1} \ \hbox{or} \ \varphi(\pi^{-1}(f))=Q_{r+1}, \ \hbox{a quadric of rank at least} \ 5.
\end{equation}
Set $Y_f= \varphi(\pi^{-1}(f))$. Since $\dim Y_f = r+1$, if $y \in Y_f$ is general, there is a $1$-dimensional family $\{P_t, t \in T\}$ of $\P^{r-1}$'s through $y$. Let $\mathcal C_y = \bigcup_{t \in T} P_t$ be the corresponding cone. Then $\dim \mathcal C_y = r$ and set $c = \deg \mathcal C_y$. Let $H_1, \ldots, H_{r-1}$ be general hyperplanes. Then the surface $S_f:= Y_f \cap H_1 \cap \ldots \cap H_{r-1}$ is such that there are $c$ lines contained in $S_f$ and passing through its general point. As is well known, it follows that either $c=1$ and $S_f$ is a scroll or a plane or $c=2$ and $S_f$ is a quadric. Assume now that $Y_f \ne \P^{r+1}$. When $c=1$ we get that $\mathcal C_y=\P^r$, hence $Y_f$ is a scroll in $\P^r$'s. Also, when $c=2$ we have that $Y_f=Q$ is a quadric. If $Y_f$ is a scroll, then the composition of $\Phi_{\E_{|f}}$ with the Pl\"ucker embedding maps $f = \P^3$ to another scroll in $\P^r$, say $Z_f$, of dimension $r+1$ over some curve $\Gamma$, whose $\P^r$'s are the dual of the ones in $Y_f$. Thus we get a map $g : f=\P^3 \to Z_f$ and $g$ lifts to the normalization $\nu: \P(\G) \to Z_f$, where $\G$ is a vector bundle over $\Gamma$. But this gives the contradiction that $f= \P^3$ dominates $\Gamma$. The same argument can be applied if ${\rm rk} Q \le 4$ since we can see it as a scroll in $\P^r$ over $\P^1$. This proves \eqref{sev}.

We now claim that $c_1(\E_{|L})>0$ for every line $L \subset f$. To see the latter, since $\E$ is globally generated, we can suppose that there is a line $L_0 \subset f$ such that $c_1(\E_{|L_0})=0$. Then the same holds on any line $L \subset f$, hence $\Phi_{\E_{|L}}$ is constant. On the other hand, since $\dim \varphi(\pi^{-1}(f)) = r+1$, there exist $x_1, x_2 \in f$ such that $P_{x_1} \ne P_{x_2}$ hence on $L' = \langle x_1, x_2 \rangle$ we have that $\Phi_{\E_{|L'}}$ is not constant. This proves that $c_1(\E_{|L})>0$ for every line $L \subset f$. Now for every line $L \subset f$ we have that $H^0(\E) \to H^0(\E_{|L})$ is surjective by \cite[Lemma 3.2]{ls}, hence $\varphi(\pi^{-1}(L))$ is a rational normal scroll. Then $h^0(\E_{|L}) \le r+3$ by \eqref{sev} and $1 \le c_1(\E_{|L}) \le 3$. Furthermore, if $c_1(\E_{|L})=3$ then $\varphi(\pi^{-1}(L))$ has codimension $1$ in $Q = \varphi(\pi^{-1}(f))$. Note that $r \ge 2$ since $\E$ is not big and $c_1(\E)^4>0$, hence intersecting with general hyperplanes $H_i, 1 \le i \le r-2$ we get a surface $\varphi(\pi^{-1}(L)) \cap H_1 \cap \ldots \cap H_{r-2}$ of degree $3$ inside a smooth quadric in $\P^4$, a contradiction. This gives that $1 \le c_1(\E_{|L}) \le 2$. 

If $c_1(\E_{|L}) = 1$ then $\E_{|L} \cong \O_{\P^1}(1) \oplus \O_{\P^1}^{\oplus (r-1)}$, hence $c_1(\E_{|f}) = 1$ and \cite[Prop.~IV.2.2]{e} implies that $\E_{|f}$ is either $T_{\P^3}(-1) \oplus \O_{\P^3}^{\oplus (r-3)}, r \ge 3$ or $\O_{\P^3}(1) \oplus \O_{\P^3}^{\oplus (r-1)}$ or $\Omega_{\P^3}(2) \oplus \O_{\P^3}(1)^{\oplus (r-3)}, r \ge 3$. The first case does not occur since then \eqref{sev} implies the contradiction $r+2 \le h^0(\E_{|f}) = h^0(T_{\P^3}(-1) \oplus \O_{\P^3}^{\oplus (r-3)})=r+1$. In the second case we have that $\E_{|f}$ is big, contradicting the fact that $\dim \varphi(\pi^{-1}(f)) = r+1$. Also $c_1(\Omega_{\P^3}(2) \oplus \O_{\P^3}(1)^{\oplus (r-3)}) = r-1 \ge 2$, hence the third case does not occur. It follows that $c_1(\E_{|f}) = 2$ and therefore $\E_{|f}$ is as in (i)-(vii) of \cite[Thm.~1]{su}. Now $r+2 \le h^0(\E_{|f}) \le r+3$ by \eqref{sev}, hence cases (i)-(iii) and (vii) of \cite[Thm.~1]{su} are excluded. Thus $(X, \O_X(1),\E)$ is as in (xii).
\end{proof} 
\renewcommand{\proofname}{Proof}

\begin{claim} 
\label{p2bdle}
In case (c.3) we have that $(X, \O_X(1), \E)$ is as in (xiii).
\end{claim}
\begin{proof}
By \eqref{acc} and Claim \ref{sato} we can apply Lemma \ref{nu=b+r}(ii) and we are in case (xiii). 
\end{proof} 
\renewcommand{\proofname}{Proof}

To conclude the proof of case (A) we can assume that we are neither in case (b.2) nor in case (c.3). 

Again \cite[Thm.~1.4]{lp}, the classification of del Pezzo $4$-folds (see for example \cite[\S 1]{lp}, \cite{fu1}) and \eqref{ellesse} give that $\nu(\E) = r+2, \dim F_1(X,x)=1$ and either we are in case (c.2) or in one of the following cases for $(X,\O_X(1))$:
\begin{itemize}
\item [(A.1)] a cubic hypersurface $\P^5$.
\item [(A.2)] a complete intersection of two quadrics in $\P^6$.
\item [(A.3)] a linear section $\mathbb G(1,4) \cap H_1 \cap H_2$, where $\mathbb G(1,4) \subset \P^9$ in the Pl\"ucker embedding and $H_i  \subset \P^9$ are hyperplanes $i=1,2$.
\item [(A.4)] $(\P^2 \times \P^2, \O_{\P^2}(1) \boxtimes \O_{\P^2}(1))$.
\end{itemize}  

Moreover note that in cases (A.1)-(A.3), or in case (c.2), we have that $F_2(X,x) = \emptyset$. In fact, if not, being a closed condition, we would get that $F_2(X,x') \ne \emptyset$ holds on any point $x' \in X$. Now \cite[Main Thm.]{sa2} gives that $X$ is a linear $\P^2$-bundle over a smooth surface, hence \cite[Main Thm.]{fu2} implies that  we are in case (c.3), a contradiction. Therefore we can apply \cite[Prop.~4.6]{lms} and deduce that there is a morphism $\psi: \P^{r-1} \to F_1(X,x)$ that is finite onto its image and 
$$1 = \dim F_1(X,x) \ge r-1$$ 
hence $r \le 2$. On the other hand, $\E$ is not big and therefore $r=2$. 

\begin{claim} 
\label{a1-3}
Cases (A.1), (A.2), (A.3) and (A.4) do not occur.
\end{claim}
\begin{proof} 
Case (A.4) does not occur by \cite[Thm.~3]{lms}. In cases (A.1) and (A.2) it is well known (see for example \cite[Ex.~2.3.11]{ru}) that $F_1(X,x)$ is a smooth complete intersection of type $(2,3)$ or $(2,2)$ in $\P^3$, dominated by $\P^1$ via $\psi$, a contradiction. 

In case (A.3) note that $\det \E=2H$ by \cite[Lemma 3.2]{lo}. Let $Y$ be a smooth hyperplane section of $X$. By Remark \ref{gen}(iv), $\E_{|Y}$ is Ulrich on $Y$ and it is indecomposable since $\Pic(Y) \cong \Z$ and there are no Ulrich line bundles on $Y$. It follows by \cite[Thm.~3.4]{ac} and Remark \ref{gen}(iii) that $\E_{|Y} \cong S_L(l)$ or $S_C(l)$ or $S_E(l)$ for some $l \in \Z$ (see \cite[Ex.~3.1, 3.2, 3.3]{ac} for the definition of these sheaves). Since $\det (\E_{|Y})=2H_{|Y}$ we get that $\E_{|Y} \cong S_L(1)$ or $S_E(1)$. Now $h^0(\E_{|Y})= 2 \deg X = 10$ while $h^0(S_L(1))=12$, hence this case is excluded. Therefore $\E_{|Y} \cong S_E(1)$ and we deduce that $c_2(\E) \cdot H^2 = c_2(S_E(1)) \cdot H_{|Y} = 7$. We know that $N^2(X)$ is generated by the classes of two planes $M_1, M_2$ with $M_1^2=1, M_2^2=2, M_1 \cdot M_2=-1$ by \cite[Cor.~4.7 and proof]{pz}. Hence $c_2(\E) = a_1[M_1]+a_2[M_2]$ for some $a_1, a_2 \in \Z$ and
$$7 =  c_2(\E) \cdot H^2 = (a_1[M_1]+a_2[M_2]) \cdot H^2 = a_1+a_2.$$
Therefore
$$s_4(\E^*) = c_1(\E)^4 - 3 c_1(\E)^2 \cdot c_2(\E) + c_2(\E)^2 = 5a_1^2-42a_1+94 > 0$$
giving that $\E$ is big, a contradiction. 
\end{proof} 
\renewcommand{\proofname}{Proof}

\begin{claim} 
\label{a4}
In case (c.2) we have that $(X, \O_X(1), \E)$ is as in (xiv).
\end{claim}
\begin{proof} 
Let $Q$ be a general fiber of $\phi_{\tau}$. Then $(\det \E)_{|Q}=eH_{|Q}$ for some $e \ge 0$. 
\begin{subclaim} 
\label{uni}
For any line $L \subset Q$ we have that 
\begin{equation}
\label{spl}
\E_{|L} \cong \O_{\P^1} \oplus \O_{\P^1}(e).
\end{equation}
\end{subclaim}
\begin{proof} 
We first prove that \eqref{spl} holds for a general line $L \subset Q$.

There is a nonempty open subset $U \subseteq X$ such that $F_1(X,x)$ is smooth, $\dim F_1(X,x)=1$ and $Q_x:=\phi_{\tau}^{-1}(\phi_{\tau}(x))$ is smooth irreducible for any $x \in U$. By \cite[Prop.~4.6]{lms} there is a $1$-dimensional family of lines $Z_x:=\psi(\P^1) \subset F_1(X,x)$, hence $Z_x$ is an irreducible component of $F_1(X,x)$, disjoint from other components. For every $[L] \in Z_x$ we have that $\dim_{[L]} F_1(X,x) = \dim_{[L]} Z_x = 1$, hence $K_X \cdot L = -3$. Therefore $(K_X+3H) \cdot L =0$ and then $L \subset Q_x$. Hence $Z_x=F_1(Q_x,x)$ for every $x \in U$. 

Moreover, let us see that for every $[L] \in Z_x$ property \eqref{spl} holds. In fact, if $[L] \in Z_x$ then we have that $L=\psi(y)=\Pi_y$ for some $y \in P_x$. Since $\varphi^{-1}(y)$ is a curve contracted by $\varphi$ and $x \in L=\Pi_y$, we deduce that there is a $0$-dimensional subscheme $Z$ of $L$ of length $2$ such that $H^0(\E) \to H^0(\E_{|Z})$ is not surjective. On the other hand, the restriction map $H^0(\E) \to H^0(\E_{|L})$ is surjective by \cite[Lemma 3.3]{ls}, hence $H^0(\E_{|L}) \to H^0(\E_{|Z})$ is not surjective. Therefore $\E_{|L}$ is not ample and \eqref{spl} holds. 

Consider now the incidence correspondence
$$\I = \{ (u, [L]) \in Q \times F_1(Q) : u \in L \}$$
together with its surjective projections 
$$\xymatrix{& \I \ar[dl]_{p_1} \ar[dr]^{p_2} & \\ \ \ \ \ Q & & \ F_1(Q) \ \ \ .}$$
Note that $\I$ is irreducible. Since $p_1^{-1}(U \cap Q)$ dominates $F_1(Q)$, there is a nonempty open subset $V$ of $F_1(Q)$ such that $V \subset p_2(p_1^{-1}(U \cap Q))$. Now for any line $[L] \in V$ we have that $L = p_2(x,L)$ for some $(x,L) \in p_1^{-1}(U \cap Q)$, so that $x \in L \subset Q$ and $x \in U \cap Q$. Therefore $[L] \in F_1(Q,x)=F_1(Q_x,x)=Z_x$. Hence \eqref{spl} holds for a general line $L \subset Q$. Now for any $[L] \in F_1(Q)$ we have that $\E_{|L} \cong \O_{\P^1}(a_1) \oplus \O_{\P^1}(a_2)$ with $a_1 \ge a_2 \ge 0$, since $\E$ is globally generated. By semicontinuity we have that $h^0(\E_{|L}(-e)) \ge 1$ since \eqref{spl} holds for a general line $L_t \subset Q$. Therefore we get that $a_1 \ge e$. Now $a_1+a_2=c_1(\E) \cdot L=c_1(\E) \cdot L_t=e$, so that $a_2 = e-a_1 \le 0$ and therefore $a_2=0$ and $a_1=e$.
\end{proof} 
We now continue with the proof of Claim \ref{a4}.
 
Note that $\E_{|Q}$ cannot split as $\O_Q \oplus \O_Q(e)$. In fact, if such a splitting holds, choosing a point $x \in U \cap Q$ (see proof of Subclaim \ref{uni}) we get that $x \in L=\Pi_y$, hence there is $z \in \varphi^{-1}(y)$ such that $x = \pi(z)$. On the other hand, $z \in \varphi^{-1}(y) = \P(\O_L) \subset \P(\O_Q) \subset \P(\E)$. But $\P(\O_Q)$ is contracted to a point by $\varphi$, contradicting the fact that $\varphi^{-1}(y)$ is a curve.

Now Subclaim \ref{uni} gives that $\E_{|Q}$ is a uniform rank $2$ vector bundle on $Q$ satisfying \eqref{spl}. Therefore $\E_{|Q}$ is indecomposable and \cite[Cor.~6.7]{mos} gives that $\E_{|Q}$ is a spinor bundle, that is we are in case (xiv).
%
%
This proves Claim \ref{a4}.
\end{proof} 

This concludes the proof of Theorem \ref{main3} in Case (A).

\noindent {\bf Case (B)}: $c_1(\E)^4 = 0$.

Note that this implies that $\rho(X) \ge 2$. Recall that $(X,\O_X(1))$ belongs to one of the cases (a)-(e) in Lemma \ref{aggiu} and we are assuming that $(X,\O_X(1))$ is neither as in (a) nor as in (b.1). In many cases we will apply the notation and results of the Dichotomy Lemma with respect to the nef value morphism $\phi_{\tau}$ in Lemma \ref{aggiu}.

\begin{claim} 
In case (b.2), either we are in case (fin) and $(X,\O_X(1),\E)$ is as in (ii1) or we are in case (fact) and $(X,\O_X(1),\E)$ is as in (ii2). 
\end{claim}
\begin{proof}
In case (fact) just apply Lemmas \ref{fe}(iii) and \ref{ii} to get case (ii2). In case (fin) we are in (emb) by Lemma \ref{fe}(ii). Now the same proof of \cite[Cor.~5, Case 2]{ls} applies and we get that $(X,\O_X(1),\E)$ is as in (ii1).
\end{proof}

\begin{claim}  
In case (c.1) we have that $(X,\O_X(1),\E)$  is as in (iii).
\end{claim}
\begin{proof}
Since $\rho(X) \ge 2$, using the classification of del Pezzo $4$-folds (see for example \cite[\S 1]{lp}, \cite{fu1}), we see that $X=\P^2 \times \P^2 \subset \P^8$. Then we are in case (iii) by \cite[Thm.~3]{lms}. 
\end{proof}

\begin{claim} 
In case (c.2), $(X,\O_X(1),\E)$ is as in (iv). 
\end{claim}
\begin{proof}
Just apply Proposition \ref{qf}.
\end{proof}

\begin{claim} 
\label{serve}
In case (c.3), either we are in case (fact) and $(X,\O_X(1),\E)$ is as in (v2) with $p = \phi_{\tau} = \widetilde \Phi$ and $b=2$ or we are in case (emb) and either $(X,\O_X(1),\E)$ is as in (v3) with $p = \widetilde \Phi, b=3$, or $(X,\O_X(1),\E)$ is as in (v1).
\end{claim}
\begin{proof}
Set $\phi = \phi_{\tau} : X \to X'$. In case (fact) just apply Lemmas \ref{fe}(iii) and \ref{ii} to get $(X,\O_X(1),\E)$ is as in (v2) with $p = \phi_{\tau} = \widetilde \Phi$ and $b=2$. In case (fin) we are in (emb) by Lemma \ref{fe}(ii). Note that it cannot be that there is a fiber $F_{x_0}=\P^3$, for then $\phi_{|F_{x_0}} : F_{x_0}=\P^3 \to X'$ is constant, hence $F_{x_0} \subset f_{x_0}$, a contradiction. It follows that if $F_x = \P^2$ then $F_u=\P^2$ for every $u \in X$, hence $\widetilde \Phi$ has equidimensional fibers and gives a linear $\P^2$-bundle over a smooth $\widetilde{\Phi(X)}$. On the other hand, $\phi_{|F_x} : F_x \to X'$ is a closed embedding, giving that $X' \cong \P^2$. Also $\widetilde \Phi_{|f_x} : f_x \to \widetilde{\Phi(X)}$ is a closed embedding, giving that $\widetilde{\Phi(X)} \cong \P^2$. Thus $X$ has two different $\P^2$-bundle structures over $\P^2$ and \cite[Thm.~A]{sa1} implies that $(X, \O_X(1)) \cong (\P^2 \times \P^2, \O_{\P^2}(1) \boxtimes \O_{\P^2}(1)$. But then $K_X=-3H$, a contradiction since we are in case (c.3). 

Therefore $F_x$ is a line and $c_1(\E)^3 \ne 0$.

Now we have two possibilities: either there is a fiber $F_{x_0}=\P^2$ or $F_u$ is a line for every $u \in X$. 

Consider the second case. We have that $\widetilde \Phi$ has equidimensional fibers and then $(X,\O_X(1),\E)$ is a linear Ulrich triple with $p = \widetilde \Phi$ and $b=3$ by Lemma \ref{ii}. Since we are in case (emb), we also have an unsplit family of smooth rational curves $R = \{\phi(F_u), u \in X\}$ covering $X'$. For each $r \in R$ let $C_r \subset X'$ be the corresponding rational curve. For every $z \in X'$ we have a morphism 
$$\gamma_z : \P^2 \cong \phi^{-1}(z) \to R$$ 
defined by $\gamma_z(u)=\phi(F_u)$. Hence $\gamma_z$ is either finite onto its image or constant. We claim that the first case does not occur. In fact, consider the incidence correspondence
$$\I=\{(z, r) \in X' \times R : z \in C_r\}$$
together with its two projections $\pi_1 : \I \to X'$ and $\pi_2 : \I \to R$. Note that $\pi_1^{-1}(z) \cong \Im \gamma_{z}$ has dimension $0$ or $2$ for every $z \in X'$. Now $\pi_2$ is surjective and $\pi_2^{-1}(r) \cong C_r$ for every $r \in R$, so that $\I$ is irreducible and $\dim \I = \dim R + 1$. Also $\pi_1$ is surjective and therefore $\dim R = 1 + \dim \pi_1^{-1}(z_1)$ for $z_1 \in X'$ general. If $\dim \pi_1^{-1}(z_1) = 2$, we get that $\dim R = 3$, contradicting the bend-and-break lemma (see for example \cite[Prop.~3.2]{de}). Therefore $\dim \pi_1^{-1}(z_1)=0$ and we get that $\dim R = 1$. Now for any $z \in X'$ we have that $\dim \pi_1^{-1}(z) = \dim \Im \gamma_z \le 1$, so that $\dim \pi_1^{-1}(z)=0$ and $\gamma_{z}$ is constant. 

For every $u \in X$ let $Y_u = \phi^{-1}(\phi(F_u))$. Then $\phi_{|Y_u} : Y_u \to \P^1 \cong \phi(F_u)$ exhibits $Y_u$ as a linear $\P^2$-bundle over $\P^1$ with fibers $f_{u'}, u' \in Y_u$. Now observe that, for every $u' \in Y_u$ we have that $F_{u'} \subset Y_u$: in fact, since $u' \in Y_u$ we have that $z':=\phi(u') \in \phi(F_u)$, so that there is a $u'' \in F_u$ such that $z'=\phi(u'')$. Hence $u', u'' \in \phi^{-1}(z')$ and therefore $\gamma_{z'}(u')=\gamma_{z'}(u'')$, that is $\phi(F_{u'})=\phi(F_{u''})=\phi(F_u)$ since $F_{u''}=F_u$. Now for any $u_1 \in F_{u'}$ we have that $\phi(u_1) \in \phi(F_{u'})=\phi(F_u)$ and therefore $u_1 \in Y_u$. Thus $F_{u'} \subset Y_u$ for every $u' \in Y_u$. 
Set  $h_u := \widetilde \Phi_{|Y_u} : Y_u \to \widetilde{\Phi(X)}$. Then $F_{u'} = h_u^{-1}(h_u(u'))$ for every $u' \in Y_u$. This gives that $h_u(Y_u)$ has dimension $2$. On the other hand for any $u' \in Y_u$ we have that ${h_u}_{|f_{u'}}$ is a closed embedding, and therefore $h_u(Y_u) \cong \P^2$ and $h_u$ exhibits $Y_u$ as a linear $\P^1$-bundle over $\P^2$ with fibers $F_{u'}, u' \in Y_u$.
Finally observe that $Y_u$ is smooth, since $X \cong \P(\F)$ and $Y_u \cong \P(\F_{|\phi(F_u)})$. It follows by \cite[Thm.~A]{sa1} that, for every $u \in X$, $Y_u \cong \P^1 \times \P^2$ embedded by the Segre embedding in $X \subset \P^N$. Moreover since $\gamma_{z}$ is constant for every $z \in X'$ it follows that there is a unique $r_z \in R$ such that $z \in C_{r_z}$. This defines a morphism $X' \to R$ with all fibers the curves $C_r$. Passing to the Stein factorization we get a $\P^1$-bundle $X' \to \widetilde R$ onto a smooth curve. Finally, the fibers of the composition $X \to X' \to  \widetilde R$ are exactly the $Y_u$, that is $\P^1 \times \P^2$ embedded by the Segre embedding. This concludes the proof in the case that $F_u$ is a line for every $u \in X$, and gives that $(X,\O_X(1),\E)$ is as in (v3) with $p = \widetilde \Phi$ and $b=3$.

Assume now that there is a fiber $F_{x_0}=\P^2$, so that, as above $X' \cong \P^2$. Recall that we have a very ample rank $3$ vector bundle $\F$ over $\P^2$ such that $(X, \O_X(1)) \cong (\P(\F),\O_{\P(\F)}(1))$. As is well known, $F_1(X,x)$ is smooth, hence $L:=F_x$ belongs to a unique irreducible component $W$ of $F_1(X,x)$. For any line $[L'] \in W$ we have that $\det \E \cdot L' = \det \E \cdot L = 0$, hence $L'$ is contracted by $\Phi$, that is $L'=L$ and therefore $\dim W=0$. Hence $0 = \dim_{[L]} F_1(X,x) = - K_X \cdot L -2$, so that $K_X \cdot L = -2$. 

Set $c_i(\F)=a_iH_{\P^2}^i$ for some $a_i \in \Z, i=1,2$. Let $\xi = \O_{\P(\F)}(1)$ and $R = \phi^*(\O_{\P^2}(1))$, so that $\xi^2 \cdot R^2 = 1$ and $R^3=0$. We claim that 
$$\xi^4 = a_1^2-a_2 \ \hbox{and} \ \xi^3 \cdot R = a_1.$$ 
In fact, from the standard relation
$$\sum\limits_{i=0}^3 (-1)^i \phi^*(c_i(\F)) \xi^{3-i}=0$$
we get
$$\xi^3 = a_1 \xi^2 \cdot R - a_2 \xi \cdot R^2$$
hence $\xi^3 \cdot R = a_1$ and $\xi^4 = a_1 \xi^3 \cdot R - a_2 \xi^2 \cdot R^2 = a_1^2-a_2$. 

Now there are $a, b \in Z$ such that 
$$\det \E = a \xi + b R.$$
Note that it cannot be that $a=0$, for otherwise we would get that $c_1(\E)^3 = 0$, a contradiction. Also, the condition $c_1(\E)^4 = 0$ gives
\begin{equation}
\label{sec}
6a^2b^2 + 4a^3ba_1 + a^4(a_1^2-a_2)=0.
\end{equation}
Now $\xi \cdot L = 1$ and $\det \E \cdot L = 0$, giving
\begin{equation}
\label{sec2}
a + b R \cdot L = 0.
\end{equation}
Therefore $K_X \cdot L = -2$ gives
$$(-3\xi + \phi^*(K_{\P^2}+ \det \F)) \cdot L = -2$$
that is
$$(a_1-3) R \cdot L = 1$$
and therefore
$$a_1=4 \ \hbox{and} \ R \cdot L = 1.$$
We get from \eqref{sec2} that $b=-a$ and from \eqref{sec} that $a_2=6$. 

Let $M$ be any line in $\P^2$. Then $\F_{|M}$ is ample and $c_1(\F_{|M}) = c_1(\F) \cdot M = 4$. This implies that $\F_{|M} \cong  \O_{\P^1}(2) \oplus \O_{\P^1}(1)^{\oplus 2}$. It follows by \cite[Prop.~5.1]{el} that $\F \cong  \O_{\P^2}(2)  \oplus \O_{\P^2}(1)^{\oplus 2}$ or $T_{\P^2} \oplus \O_{\P^2}(1)$. The first case is excluded since it has $a_2=5$. Therefore  $\F \cong T_{\P^2} \oplus \O_{\P^2}(1)$ and $\det \E = a(\xi-R)$. Since $|\xi-R|$ is base-point free and defines a morphism $q : X \to \P^3$, it follows by \cite[Lemma 5.1]{lo} that there is a rank $r$ vector bundle $\G$ on $\P^3$ such that $\E \cong q^*\G$. We now claim that $\G \cong \O_{\P^3}(2)^{\oplus r}$. To this end  observe that we can see $X$ as a hyperplane section of $\P^2 \times \P^3$ under the Segre embedding given by $L=\O_{\P^2}(1) \boxtimes \O_{\P^3}(1)$. Moreover $\O_X(1) = L_{|X}$ and $q={p_2}_{|X} : X \to \P^3$, where $p_2: \P^2 \times \P^3 \to \P^3$ is the second projection. Consider the exact sequence
$$0 \to (p_2^*\G)(-(p+1)L) \to (p_2^*\G)(-pL) \to \E(-pH) \to 0.$$
Since $H^i(\E(-pH))=0$ for all $i \ge 0$ and $1 \le p \le 4$ we deduce that
$$H^i((p_2^*\G)(-(p+1)L)) \cong H^i((p_2^*\G)(-pL)) \ \hbox{for all} \ i \ge 0 \ \hbox{and} \ 1 \le p \le 4.$$
Now the K\"unneth formula gives that
$$h^0(\O_{\P^2}(p-2))h^{i-2}(\G(-p-1))=h^0(\O_{\P^2}(p-3))h^{i-2}(\G(-p)) \ \hbox{for} \ i \ge 2 \ \hbox{and} \ 1 \le p \le 4$$
and one easily sees that this implies that $H^j(\G(-2)(-s))=0$ for all $j \ge 0$ and $1 \le s \le 3$. But then $\G(-2)$ is an Ulrich vector bundle for $(\P^3, \O_{\P^3}(1))$ and therefore $\G \cong \O_{\P^3}(2)^{\oplus r}$ by Remark \ref{kno1}. Thus $(X,\O_X(1),\E)$ is as in (v1).
\end{proof}

In the sequel we will use the standard notation $V_7$ for the blow up of $\P^3$ in a point.

In order to study cases (d.2)-(d.5) we first observe the following.

\begin{claim}
\label{fe2}
In cases (d.2), (d.3), (d.4) and (d.5), case {\rm (fin)} does not hold for $\phi_{\tau}$.
\end{claim}
\begin{proof}
Suppose that case (fin) holds for $\phi_{\tau}$ and let $L$ be any line  such that $x \in L \subseteq F_x$. Since $\tau = 2$ we have that $(K_X+2H) \cdot L \ge 0$, that is $K_X \cdot L \ge -2$. On the other hand, we have that
$$0 \le \dim_{[L]} F_1(X,x) = -K_X \cdot L - 2$$
so that $K_X \cdot L = - 2$ and therefore $\phi_{\tau}(L)$ is a point, so that $L \subseteq f_x \cap F_x$, a contradiction.
\end{proof}

\begin{claim} 
\label{liscie}
In case (d.2) we have that every smooth fiber is isomorphic to only one of $V_7, \P^1 \times \P^1 \times \P^1$ or $\P(T_{\P^2})$.
\end{claim}
\begin{proof}
By Claim \ref{fe2} we are in case (fact) for $\phi_{\tau}$, hence $\phi_{\tau}$ factorizes through $\widetilde \Phi$. Let $f_u$ be a smooth fiber. If $\Pic(f_u) \cong \Z$, Lemma \ref{fe}(iii) gives that $f_u = F_u$, hence $f_u \cong \P^3$ and $\O_X(1)_{|f_u} \cong \O_X(1)_{|F_u} \cong \O_{\P^3}(1)$.  But $\phi_{\tau} : X \to X'$ is a del Pezzo fibration for $(X, \O_X(1))$, so that ${K_X}_{|f_u} = - 2 H_{|f_u}$, giving the contradiction 
$$-4H_{\P^3}=K_{f_u} = {K_X}_{|f_u} = - 2 H_{|f_u}=-2H_{\P^3}.$$ 
Now the classification of del Pezzo $3$-folds (see for example \cite[\S 1]{lp}, \cite{fu1}) implies that $f_u$  is either $V_7, \P^1 \times \P^1 \times \P^1$ or $\P(T_{\P^2})$. In the first case $f_u$ is a del Pezzo $3$-fold of degree $7$ and then $f_u \cong V_7$. In the other two cases observe that $\rho(f_u) = \rho(f_x)$ by \cite[Thm.~1.4]{jr} and therefore $f_u \cong \P^1 \times \P^1 \times \P^1$ when $f_x \cong \P^1 \times \P^1 \times \P^1$ and $f_u \cong \P(T_{\P^2})$ when $f_x \cong \P(T_{\P^2})$. 
\end{proof}

\begin{claim} 
In case (d.2), if $f_x \cong V_7$, then $(X,\O_X(1),\E)$ is as in (vi1).
\end{claim}
\begin{proof}
By \cite[(4.7)]{fu3} we have that $f_u \cong V_7$ for every $u \in X$. We claim that $F_u$ is a line for every $u \in X$. Assume to the contrary that there is a $u \in X$ such that $F_u$ is not a line. We know that $F_u \subset f_u$. Then it cannot be that $F_u=\P^3$, for otherwise, $\P^3 \cong V_7$, a contradiction. Therefore $F_u = \P^2$. Let $E$ be the exceptional divisor of the blow-up $\varepsilon : f_u \cong V_7 \to \P^3$ in a point and let $\widetilde H$ be the pull back of a plane. Observe that since $K_{f_u} = (K_X + f_u)_{|f_u} = {K_X}_{|f_u} = -2H_{|f_u}$ and, as is well known, $\Pic(V_7)$ has no torsion, we get that $H_{|f_u} = 2\widetilde H -E$. Since $E$ is the only linear plane contained in $V_7$ (with respect to $2\widetilde H-E$), we deduce that $F_u = E$. Let now $u' \in f_u \setminus E$. Then $F_{u'}$ is not contained in $E$, and therefore, by what we have just proved, $F_{u'}$ is a line. As above $F_{u'} \cap E \ne \emptyset$, for otherwise $\O_{\P^1}(1) = (2\widetilde H-E)_{|F_{u'}} = 2\widetilde H_{|F_{u'}}$, a contradiction. But then $F_{u'} \cap F_u = F_{u'} \cap E \ne \emptyset$ and therefore $F_{u'} = F_u=E$, a contradiction. Hence $F_u$ is a line for every $u \in X$ and Lemma \ref{ii} implies that $(X,\O_X(1),\E)$ is a linear Ulrich triple with $p = \widetilde \Phi : X \cong \P(\F) \to B$ and $b = 3$. As we know, $\phi_{\tau} = h \circ p$ where $h : B \to X'$. Also there is an ample line bundle $\L$ on $X'$ such that
$$mp^*(K_B + \det \F) = m (K_X + 2H) = \phi_{\tau}^* \L = p^*(h^*\L)$$
and therefore $m(K_B + \det \F) = h^*\L$, so that $(B, \det \F)$ is a del Pezzo fibration over $X'$. Also, for every $x' \in X'$, we have that $V_7 \cong \P(\F_{|h^{-1}(x')})$, hence $h^{-1}(x')$ is a smooth surface and we have a surjective morphism $p_{|V_7} : V_7 \to h^{-1}(x')$. Now $E$ is a plane and the fibers $F_{u'}, u' \in V_7$ of $p_{|V_7}$ are lines. As above $F_{u'} \cap E \ne \emptyset$. On the other hand, it cannot be that $F_{u'} \subset E$, for otherwise the morphism $\Phi_{|E} : \P^2 \cong E \to \Phi(X)$, that contracts $F_{u'}$ to a point, would be constant, therefore implying the contradiction $F_{u'} = E$. Hence $F_{u'} \cap E$ is a point for every $u' \in V_7$ and $p_{|E} : E \to h^{-1}(x')$ is a closed embedding, thus giving that $h^{-1}(x') \cong \P^2$.
\end{proof}

\begin{claim} 
\label{dp6}
In case (d.2), if $f_x \cong \P^1 \times \P^1 \times \P^1$, then $(X,\O_X(1),\E)$ is as in (vi2), while if $f_x \cong \P(T_{\P^2})$, then $(X,\O_X(1),\E)$ is as in (vi3).
\end{claim}
\begin{proof}
By Claim \ref{fe2}  we are in case (fact) and, for every $u \in X$, $f_u$ is a disjoint union of fibers $F_v$ of $\Phi$. Also, $K_{f_x} = -2H_{|f_x}$, hence $H^3 \cdot f_u =  H^3 \cdot f_x = 6$. 
\begin{subclaim} 
\label{normale}
$f_u$ is normal for every $u \in X$.
\end{subclaim}
\begin{proof}
Note that $f_u$ is integral by \cite[(4.6)]{fu3}, hence there is no linear $\P^3$ contained in $f_u$. Moreover, $f_u$ is not a cone by \cite[\S 1, p.~232]{fu3}. Assume that $f_u$ is not normal.  It follows by \cite[Thm.~2.1(a)]{fu4} and \cite[Thm.~II]{fu5} that $f_u$ is the projection $\pi_O : \Sigma \to f_u$ of a rational normal threefold scroll $\Sigma \subset \P^8$ of degree $6$ from a point $O \not\in \Sigma$. In particular a general curve section $C$ of $f_u$ is a rational curve of degree $6$ in $\P^5$ and with arithmetic genus $1$, because so is the arithmetic genus of a general curve section of $f_x$. Therefore $C$ has a unique double point and this implies that ${\rm Sing}(f_u)$ is a plane. Then there is a quadric $Q \subset \Sigma$ such that $\pi_O(Q) = {\rm Sing}(f_u)$ and it is easily checked that then $\Sigma$ is the embedding of $\P(\O_{\P^1}(4) \oplus \O_{\P^1}(1) \oplus \O_{\P^1}(1))$ by the tautological line bundle and $Q$ is the embedding of $\P(\O_{\P^1}(1) \oplus \O_{\P^1}(1))$. Let $p : \Sigma \cong \P(\O_{\P^1}(4) \oplus \O_{\P^1}(1) \oplus \O_{\P^1}(1)) \to \P^1$ be the projection map. Then any plane contained in $\Sigma$ must be a fiber of $p$. Also, since $p_{|Q} : Q \to \P^1$ is surjective, it follows that the planes $M_z := p^{-1}(z), z \in \P^1$ intersect $Q$, and therefore $L_z := M_z \cap Q$ is an effective divisor on $Q$. On the other hand, $M_z \not\subset \langle Q \rangle = \P^3$, for otherwise $L_z$ would be a hyperplane section of $Q$ and then, for $z \ne z' \in \P^1$ we would get the contradiction $\emptyset \ne L_z \cap L_{z'} \subset M_z \cap M_{z'} = \emptyset$. Therefore, since $L_z \subset M_z \cap \langle Q \rangle$, it follows that $L_z$ is a line for every $z \in \P^1$. Now $\pi_O(L_z)$ and $\pi_O(L_{z'})$ are two lines in the plane ${\rm Sing}(f_u)$, hence they intersect. Therefore $\pi_O(M_z)\cap \pi_O(M_{z'}) \ne \emptyset$ for any $z, z' \in \P^1$. 

Let $u' \in f_u$ be a general point. If $F_{u'} = \P^2$ then, for general $u'' \in f_u$, there are two planes $M_{z'}, M_{z''} \subset \Sigma$ such that  $\pi_O(M_{z'}) =F_{u'}, \pi_O(M_{z''}) =F_{u''}$. But this gives the contradiction 
$$\emptyset = F_{u'} \cap F_{u''} = \pi_O(M_{z'})\cap \pi_O(M_{z''}) \ne \emptyset.$$ 
Assume now that $F_{u'}$ is a line. Observe that $F_{u'} \not\subset \pi_O(M_z)$ for any $z \in \P^1$, for otherwise $\Phi_{|\pi_O(M_z)} : \P^2 = \pi_O(M_z) \to \Phi(X)$ must be constant, since it contracts the line $F_{u'}$ to a point. But then $F_{u'} = \pi_O(M_z)$, a contradiction. Hence there is a line $L \subset \Sigma$ such that $F_{u'} = \pi_O(L)$ and $L \not\subset Q$. But lines in $\Sigma$ not contained in $Q$ must be contained in a plane $M_z$, thus giving a contradiction. 
\end{proof}

We assume from now on that
$$V:= f_u \ \hbox{is a singular fiber of} \ \phi_{\tau}.$$
Note that $K_V = -2H_{|V}$ and $V \subset \P^7 = \P H^0(H_{|V})$ is of degree $6$ by \cite[(1.5)]{fu3}. 

Let us recall some notation. Let $a_s \ge \ldots \ge a_0 \ge 0$ be integers and let $S=S(a_0,\ldots,a_s)= \varphi_{\xi}(\P(\G)) \subset \P^{\sum_{i=0}^s a_i+s}$ be the rational normal scroll, where $\G = \O_{\P^1}(a_0) \oplus \ldots \oplus \O_{\P^1}(a_s)$, $p : \P(\G) \to \P^1$ and $\xi$ is the tautological line bundle. We denote, for every $t \in \P^1$, by $R=R_t$ a ruling on $\P(\G)$ and by $G=G_t$ its image on $S$. When some $a_i$ are zero, $S(a_0,\ldots,a_s)$ is a cone with vertex $S_0$ and inverse image $W_0$ on $\P(\G)$. If $a_k=1$ for some $k$, let
$l(S) = \max\{k \ge 0: a_k=1\}$, let $S_1=S(a_0,\dots,a_{l(S)})$, $W_1 =\P(\O_{\P^1}(a_0) \oplus \ldots \oplus \O_{\P^1}(a_{l(S)}))$ with tautological line bundle $\xi_1$ and ruling $R_1$. We recall that if a line $L \subset S$ is not contained in a ruling, then $L \subset S_1$. For every subvariety $Y \subset S$ such that $Y \not\subseteq S_0$ we denote its strict transform on $\P(\G)$ by $\widetilde Y := \overline{\varphi_{\xi}^{-1}(Y \setminus S_0)}$. 

\begin{subclaim}
\label{c1}
Let $S_V \subset\P^7$ be one of $S(0,0,1,3), S(0,0,2,2)$ or $S(0,1,1,2)$ and let $\G_V$ be one of $\O_{\P^1}^{\oplus 2} \oplus \O_{\P^1}(1) \oplus \O_{\P^1}(3), \O_{\P^1}^{\oplus 2} \oplus \O_{\P^1}(2)^{\oplus 2}$ or $\O_{\P^1} \oplus \O_{\P^1}(1)^{\oplus 2} \oplus \O_{\P^1}(2)$. Then $V \subset S_V, \widetilde V \sim 2\xi - 2 R$ on $\P(\G_V), W_0 \subset \widetilde V, S_0 \subset V$ and $\varphi_{\xi}^{-1}(V) = \widetilde V$.  
\end{subclaim}
\begin{proof} 
Let $v_0 \in {\rm Sing}(V)$, let $\pi_{v_0} : V \dashrightarrow \P^6$ be the projection with center $v_0$ and let $V' \subset \P^6$ be its image. Since $V$ is not a cone by \cite[p.~232]{fu3}, it follows that $V'$ is an irreducible non-degenerate threefold of degree at most $4$. Hence $\deg V' = 4$ and \cite[Thm.~1]{eh2} implies that either $V'$ is a rational normal scroll or a cone over the Veronese surface $S \subset \P^5$. In the second case, we have that $V \subset \mathcal C$ the cone with vertex a line $L$ over $S \subset \P^5$. Let $\pi_L : \P^7 \dashrightarrow \P^5$ be the projection with center $L$. Note that any plane $M \subset \mathcal C$ is of type $\langle p, L \rangle, p \in S$. Now, for any point $z \in M \setminus L$ we have that $p=\pi_L(z) \in S$ and $M=\langle p, L \rangle$. Also, any line $L' \subset \mathcal C$ must intersect $L$, for otherwise $\pi_L(L')$ is a line in $S$, a contradiction. If there is a point $v \in V$ such that $F_v$ is a plane, then $L \subset F_v$ and picking any $v' \in V \setminus F_v$ we find the contradiction 
$$\emptyset \ne F_{v'} \cap L \subseteq F_{v'} \cap F_v =  \emptyset.$$
Therefore $F_v$ is a line for every $v \in V$. This means that $\dim \Phi(V) = 2$ and $V$ is disjoint union of the $2$-dimensional family of lines $\{\Phi^{-1}(y), y \in \Phi(V)\}$. Since they all intersect $L$, it follows that there is a point $v' \in L$ contained in infinitely many such lines, a contradiction. 

Therefore  $V'$ is a rational normal scroll of degree $4$, hence it can be one of $S(0,1,3), S(0,2,2), S(0, 0, 4)$ or $S(1,1,2)$ and, as a consequence, $V \subset S_V \subset\P^7$, where $S_V$ is one of $S(0,0,1,3), S(0,0,2,2), S(0, 0, 0, 4)$ or $S(0,1,1,2)$. Let $C$ be a general curve section of $V$ so that $C$ is a smooth irreducible elliptic curve of degree $6$. If $S_V = S(0, 0, 0, 4)$ then $C \subset S(0,4)$, which is not possible since $S(0,4)$ does not contain an elliptic curve of degree $6$. Thus $S_V$ is one of $S(0,0,1,3), S(0,0,2,2)$ or $S(0,1,1,2)$. Hence the vertex of $S_V$ is a point or a line and we can consider $\widetilde V \sim a\xi+bR$ on $\P(\G_V)$. It follows that $V$ is a Weil divisor linearly equivalent to $aH+bG$ on $S_V$ and $C \sim aH_{|Y}+bG_{|Y}$, where $Y$, the general surface section of $S_V$, is a non-degenerate smooth irreducible surface of degree $4$ in $\P^5$, that is $Y = S(1,3)$ or $S(2,2)$. Then this gives that $a=2$ and $b= -2$. Now $W_0$ is covered by curves $\Gamma$ contracted by $\varphi_{\xi}$, hence $\Gamma \equiv c(\xi^3-4\xi^2R), c > 0$. Therefore $\Gamma \cdot \widetilde V = c(\xi^3-4\xi^2R) \cdot (2\xi - 2 R) = -2c<0$, hence $\Gamma \subset \widetilde V$ and it follows that $W_0 \subset \widetilde V, S_0 \subset V$ and $\widetilde V = \varphi_{\xi}^{-1}(V)$.
\end{proof}

\begin{subclaim}
\label{c3}
Let $S_V = S(0,1,1,2)$. Then $V \cap S_1 = M_1 \cup M_2$, where $M_1$ and $M_2$ are two planes. Moreover either $S_0 \subset M_1=M_2$ or $M_1 \cap M_2 = S_0$ and both $M_1 \cap G_t$ and $M_2 \cap G_t$ are lines for every $t \in \P^1$. Furthermore $\langle S(1,2), M_1 \rangle = \langle S(1,2), M_2 \rangle = \P^7$.
\end{subclaim}
\begin{proof} 
We have $S_V = \varphi_{\xi}(\G)$ where $\G = \O_{\P^1} \oplus \O_{\P^1}(1)^{\oplus 2} \oplus \O_{\P^1}(2)$. It  easily follows that $V \cap S_1 =  \varphi_{\xi}(\widetilde V \cap W_1)$. Moreover $W_0 \subset \widetilde V \cap W_1$ by Subclaim \ref{c1}, hence $S_0 \subset V \cap S_1$. Now let $D = \P(\O_{\P^1}(1)^{\oplus 2} \oplus \O_{\P^1}(2))$. Note that $D \sim \xi$ and $D \cap W_0 = \emptyset$, hence $\varphi_{\xi}(D)$ is a hyperplane section of $S$ not passing through the vertex $S_0$. Let $Q = \P(\O_{\P^1}(1)^{\oplus 2})$, with tautological line bundle $\xi_Q$ and fiber $R_Q$. Note that $Q \subseteq D \cap W_1$. Also, since $W_1 \sim \xi-2R$, we have that $\xi^2 \cdot D \cdot W_1=\xi^3 \cdot (\xi-2R)=2 = \xi^2 \cdot Q$ and therefore $Q = D \cap W_1$. Now $\widetilde V \cap W_1 \cap D = \widetilde V \cap Q \sim 2\xi_Q-2R_Q$, hence $\widetilde V \cap W_1 \cap D$ is a curve of type $(0,2)$ (or $(2,0)$) on $Q$, that is the union of two, possibly coincident, lines $L_1$ and $L_2$ on $Q$. Also note that $\widetilde V \cap W_1 \sim 2\xi_1 - 2R_1$ on $W_1$ and $H^0(\xi_1 - \widetilde V \cap W_1) = H^0(-\xi_1+2R_1)=0$ and therefore $V \cap S_1 = \varphi_{\xi}(\widetilde V \cap W_1)$ is a non-degenerate degree $2$ surface in $\P^4 = \P(H^0(\xi_1))$. This implies that $V \cap S_1$ is reducible, hence $V \cap S_1 = M_1 \cup M_2$, where $M_1$ and $M_2$ are two, possibly coincident, planes. Also, in case $M_1 \ne M_2$, they must intersect in a point, hence in $S_0$. Now assume that $M_1 \subset G_t$. Let $\widetilde M_i, i = 1,2$ be the strict transforms, so that $\widetilde V \cap W_1 = \widetilde M_1 \cup \widetilde M_2$ and $\widetilde M_1 \subset R_t$. Then $\widetilde M_1 = W_1 \cap R_t$, hence $\widetilde M_1 \sim R_1$ and therefore $\widetilde M_2 \sim 2\xi_1 - 3R_1$, a contradiction since $H^0(2\xi_1 - 3R_1)=0$. Hence $\widetilde M_1 \not\subset R_t$ for every $t \in \P^1$ and then $\pi_{|\widetilde M_1} : \widetilde M_1 \to \P^1$ is surjective. Hence $\widetilde M_1 \cap R_t \ne \emptyset$ for every $t \in \P^1$ and then it is a divisor on $\widetilde M_1$. It follows that both $M_1 \cap G_t$ and $M_2 \cap G_t$ are lines for every $t \in \P^1$. Also let $\widetilde M_i \sim a_i \xi_1 + b_iR_1, i = 1, 2$, so that $a_i \ge 0, a_1 + a_2 = 2$ and $1 = \xi_1^2 \cdot (a_1 \xi_1 + b_1R_1) = 2a_1+b_1$. If $a_1 = 0$ then $b_1=1$ and we get the same contradiction as above. Similarly if $a_2=0$. Therefore $\widetilde M_1 \sim \xi_1 - R_1$. Since $\langle S(1,2) \rangle = \P^4$, to prove that $\langle S(1,2), M_1 \rangle = \P^7$, we just need to prove that $H^0(\I_{\widetilde M_1 \cup W_{12}/\P(\G)}(\xi))=0$, where $W_{12} = \P(\O_{\P^1}(1) \oplus \O_{\P^1}(2)) \subset \P(\G)$. If not, there is a $D_1 \in |\xi|$ such that $D_1 \cap W_1 = \widetilde M_1 + W_1 \cap W_{12}$. But then $W_1 \cap W_{12} \sim R_1$, hence it maps to a point in $\P^1$. But this is a contradiction since $W_1 \cap W_{12}$ contains $\P(\O_{\P^1}(1))$.
\end{proof}

\begin{subclaim}
\label{smq}
Let $t \in \P^1$ and let $\Sigma_t = V \cap G_t$. Then $\Sigma_t$ is a smooth quadric in $\P^3=G_t$ such that $S_0 \subset \Sigma_t$. In particular, $\dim({\widetilde \Phi}(V))=2$.
\end{subclaim}
\begin{proof}
Note that $\Sigma_t = \varphi_{\xi}(\widetilde V \cap R_t)$ and, by Subclaim \ref{c1}, $\deg \Sigma_t = \xi^2 \cdot (2\xi-2R) \cdot R = 2$, so that $\Sigma_t $ is a quadric in $\P^3=G_t$. Also $S_0 \subset V$ by Subclaim \ref{c1}, hence $S_0 \subset \Sigma_t$. To prove that $\Sigma_t$ is smooth, let $Z = S_1$ when $S_V = S(0,0,1,3)$ or $S(0,1,1,2)$ and $Z=S_0$ when $S_V=S(0,0,2,2)$. Observe that $V \not \subset Z$, for otherwise we would have that either $V = Z = S(0,0,1)=\P^3$ or $V = Z = S(0,1,1)$ a quadric cone in $\P^4$, a contradiction. Let $\tilde v \in \widetilde V \setminus \varphi_{\xi}^{-1}(Z)$ and $v = \varphi_{\xi}(\tilde v)$. Note that $v \not\in S_0$, for otherwise $\varphi_{\xi}(\tilde v) = v \in S_0 \subseteq Z$. Let $L$ be any line such that $v \in L \subseteq F_v$. Then $L \not\subset S_0$ and $\tilde v \in \widetilde L$, because $\tilde v \not\in W_0$. Since $L \not\subset S_1$ we get that $L \subset G_{\pi(\tilde v)}$ and therefore $F_v \subset G_{\pi(\tilde v)}$. Hence $\dim_v F_v \cap G_{\pi(\tilde v)} = \dim_v F_v \ge 1$ and it follows by semicontinuity that $\dim_v F_v \cap G_{\pi(\tilde v)} \ge 1$ for every $\tilde v \in \widetilde V$. We have that
$$\Sigma_t = \bigsqcup\limits_{y \in \Phi(V)} \Phi^{-1}(y) \cap G_t$$
and if $v \in \Phi^{-1}(y) \cap G_t$, then $\Phi^{-1}(y)=F_v$ and there is $\tilde v \in R_t$ such that $v = \varphi_{\xi}(\tilde v)$, so that $t = \pi(\tilde v)$ and $\dim \Phi^{-1}(y) \cap G_t = \dim F_v \cap G_{\pi(\tilde v)} \ge 1$. Since $\Phi^{-1}(y)=\P^k, k = 1, 2$, two cases are possible. If there is a $y \in \Phi(V)$ such that $\dim \Phi^{-1}(y) \cap G_t = 2$ then $\Phi^{-1}(y) \cap G_t = \Phi^{-1}(y) = \P^2$ and $\Sigma_t$ is either reducible or a double plane. Since $\Sigma_t \subset G_t = \P^3$ and $\Sigma_t \subset X$, we can apply the same method in the proof of Claim \ref{nome} and get a contradiction. Therefore $\Phi^{-1}(y) \cap G_t$ is a line for every $y \in \Phi(V)$ such that $\Phi^{-1}(y) \cap G_t \ne \emptyset$. Hence $\Sigma_t$ is covered by a family of disjoint lines, so that $\Sigma_t$ is smooth. Moreover note that a general fiber $F_v, v \in V$ cannot be a plane, for otherwise the argument above would show that $\dim_v F_v \cap G_{\pi(\tilde v)} \ge 2$ and then we would have an $y \in \Phi(V)$ such that $\dim \Phi^{-1}(y) \cap G_t = 2$, a contradiction. Therefore a general $F_v$ is a line and $\dim \Phi(V)=2$. Since $\widetilde \Phi$ and $\Phi$ have the same fibers, we get that $\dim({\widetilde \Phi}(V))=2$.
\end{proof}

\begin{subclaim}
\label{c4}
The case $S_V = S(0,0,1,3)$ does not occur. In the other cases we have that $V = \varphi_{\O_{\P(\F)}(1)}(\P(\F))$, where $\F = \O_{\mathbb F_0}(f) \oplus \O_{\mathbb F_0}(2C_0+f)$ if $S_V = S(0,0,2,2)$, while $\F = \O_{\mathbb F_1}(C_0+f) \oplus \O_{\mathbb F_1}(C_0+2f)$ if $S_V = S(0,1,1,2)$.
\end{subclaim}
\begin{proof}
Let $S_V$ be $S(0,0,1,3), S(0,0,2,2)$ or $S(0,1,1,2)$ and let $Y$ be respectively $S(1,3), S(2,2)$ or $S(1,2)$. We have that $S_V \subset \P^7$, and, in the first two cases, $S_V$ is a cone with vertex the line $S_0$ over $Y \subset \P^5$, while in the third case $Y \subset \P^4$ and $\langle Y \rangle \cap M_1 = \emptyset$, where  $M_1 \subset V$ is the plane given in Subclaim \ref{c3}. Set $N= S_0$ in the first two cases and $N=M_1$ in the third case. Let $z \in Y$ and let $t \in \P^1$ be such that $z \in G_t$. Consider the plane $M_z = \langle z, S_0 \rangle$ in the first two cases and $M_z = \langle z, L_1 \rangle$ in the third case, where $L_1 = M_1 \cap G_t$. Since $M_z \subset G_t$, we get by Subclaim \ref{smq} that $M_z \cap \Sigma_t$ is the union of $S_0$ (or $L_1$) and a line $L_z$, meeting $S_0$ (or $L_1$, hence also $M_1$) in a point. This defines a morphism $Y \to \mathbb G(1,7)$ which in turn gives rise to a rank $2$ globally generated vector bundle $\F$ on $Y$ such that 
$$V = \bigcup\limits_{z \in Y} L_z = \varphi_{\xi_{\F}}(\P(\F)).$$ 
Moreover note that there is a unique such line passing through the general point of $V$, for the lines $L_z$ are the fibers of the projection $\pi_N : V \dashrightarrow Y$. In particular $\varphi_{\xi_{\F}}$ is birational and we find that
\begin{equation}
\label{conti}
h^0(\xi_{\F})=8 \ \hbox{and} \ \xi_{\F}^3= \deg V = 6.
\end{equation}
Also, since the lines $L_z$ meet $N$ in a point, it follows that $N$ gives rise to a section $\P(\L_0) \subset \P(\F)$ of $\P(\F) \to Y$, where $\L_0$ is a line bundle quotient of $\F$. Hence $\L_0$ is globally generated and defines a morphism $\varphi_{\L_0} : Y \to N$. Thus we have, for some integers $a$ and $b$, an exact sequence
\begin{equation}
\label{seq}
0 \to \O_Y(aC_0+bf) \to \F \to \L_0 \to 0.
\end{equation}
Now, when $S_V = S(0,0,2,2)$ we have that $Y \cong {\mathbb F_0}$ and $\L_0$ must be $\O_{\mathbb F_0}(f)$.
By \eqref{conti} we get that $a \ge 0$ and $b \ge 0$, for otherwise \eqref{seq} gives that $8 = h^0(\xi_{\F})= h^0(\F) \le 2$. Hence $H^1(\O_{\mathbb F_0}(aC_0+bf))=0$ and \eqref{seq} gives that
$$8 = h^0(\F) = (a+1)(b+1)+2.$$ 
Moreover, using the well-known fact that $\xi_{\F}^3 = c_1(\F)^2-c_2(\F)$, we get from \eqref{seq} that 
$$a(2b+1)=6$$ 
and it follows that $a=2, b=1$. Since $\Ext^1(\O_{\mathbb F_0}(f),\O_{\mathbb F_0}(2C_0+f)) \cong H^1(\O_{\mathbb F_0}(2C_0))=0$, we get that \eqref{seq} splits and $\F \cong \O_{\mathbb F_0}(f) \oplus \O_{\mathbb F_0}(2C_0+f)$.

Next we exclude the case $S_V = S(0,0,1,3)$. We have that $Y \cong {\mathbb F_2}$ and it is easily seen that $\L_0 \cong \O_{\mathbb F_2}(f)$. Then \eqref{seq} gives an exact sequence
$$0 \to \O_f(a) \to \F_{|f} \to \O_f \to 0$$
with $a = c_1(\F) \cdot f \ge 0$, since $\F$ is globally generated. But then the sequence splits and we find that 
$\F_{|f} \cong \O_{\P^1} \oplus \O_{\P^1}(a)$. On the other hand if we let $G_t = \langle f, S_0 \rangle$ we see that the quadric $\Sigma_t = \varphi_{\xi_{\F_{|f}}}(\P(\F_{|f}))$ is a cone, contradicting Subclaim \ref{smq}. Thus this case does not occur.

Finally assume that $S_V=S(0,1,1,2)$. Then $Y \cong {\mathbb F_1}$ and it is easily seen that $\L_0$ must be $\O_{\mathbb F_1}(C_0+f)$. By \eqref{conti} we get that $a \ge 0$ and $b \ge 0$, for otherwise $8 = h^0(\xi_{\F})= h^0(\F) \le 3$. Now $\xi_{\F}^3 = c_1(\F)^2-c_2(\F)$ becomes $a^2-2ab-b+5=0$, that can be rewritten as
$$(2a+1)(4b-2a+1)=21.$$
The possible integer solutions are $(a, b)= (0, 5), (1, 2), (3, 2)$ and $(10, 5)$. As it is easily seen, the cases $(0, 5)$ and $(3, 2)$ have $H^1(\O_{\mathbb F_1}(aC_0+bf))=0$ and $h^0(\O_{\mathbb F_1}(aC_0+bf)) \ne 5$, while the case $(10, 5)$ has $h^0(\O_{\mathbb F_1}(10C_0+5f)) > 8$, so they all contradict \eqref{seq}. Thus we have that $a=1, b=2$. Since $\Ext^1(\O_{\mathbb F_1}(C_0+f),\O_{\mathbb F_1}(C_0+2f)) \cong H^1(\O_{\mathbb F_1}(f))=0$, we get that \eqref{seq} splits and $\F = \O_{\mathbb F_1}(C_0+f) \oplus \O_{\mathbb F_1}(C_0+2f)$.
\end{proof}

\begin{subclaim}
\label{c5}
Let $\mathcal Q \subset \P^3$ be the quadric cone. In the case $S_V = S(0,0,2,2)$, we have that $V \cong \P^1 \times \mathcal Q \subset \P^7$ embedded by $\O_{\P^1}(1) \boxtimes \O_{\mathcal Q}(1)$.
\end{subclaim}
\begin{proof}
By Subclaim \ref{c4} we have that $V = \varphi_{\xi_{\F}}(\P(\F))$, where $\F = \O_{\mathbb F_0}(f) \oplus \O_{\mathbb F_0}(2C_0+f)$. Let us first consider the exceptional locus ${\rm Exc}(\varphi_{\xi_{\F}})$ of $\varphi_{\xi_{\F}}$. We claim that
\begin{equation}
\label{luogoecc}
{\rm Exc}(\varphi_{\xi_{\F}})=\P(\O_{\mathbb F_0}(f)).
\end{equation}
Let $p : \P(\F) \to \mathbb F_0$ be the projection map. From Grothendieck's relation 
$$\xi_{\F}^2 = \xi_{\F} p^*c_1(\F) - p^*c_2(\F)$$ 
we deduce, setting $R = p^*(C_0) p^*(f)$, that
\begin{equation}
\label{gro}
\xi_{\F}^2 = 2 \xi_{\F}p^*C_0  + 2\xi_{\F} p^*f - 2R
\end{equation}
and therefore that $\xi_{\F}^2p^*C_0 = \xi_{\F}^2p^*f = 2$. Also \eqref{gro} gives that $N^2(\P(\F))$ is generated by $\xi_{\F}p^*C_0, \xi_{\F}p^*f$ and $R$. Let $C \subset \P(\F)$ be an irreducible curve contracted by $\varphi_{\xi_{\F}}$. Then $C \equiv a\xi_{\F}p^*C_0+b\xi_{\F}p^*f+cR$ for some integers $a,b,c$. Since $p^*C_0$ and $p^*f$ are nef, we get that $0 \le C \cdot p^*C_0 = b$ and $0 \le C \cdot p^*f = a$. Now
$$0 = \xi_{\F} \cdot C = 2a+2b+c$$
gives that $c = -2a-2b$, hence either $a>0$ or $b>0$. Also $\P(\O_{\mathbb F_0}(f)) \sim \xi_{\F} -2p^*C_0-p^*f$ and therefore
$$C \cdot \P(\O_{\mathbb F_0}(f)) = [a\xi_{\F}p^*C_0+b\xi_{\F}p^*f-(2a+2b)R] \cdot (\xi_{\F} -2p^*C_0-p^*f)=-a-2b < 0$$
so that $C \subset \P(\O_{\mathbb F_0}(f))$. On the other hand $H^0(\F) \cong H^0(\xi_{\F}) \to H^0(\xi_{\O_{\mathbb F_0}(f)}) \cong H^0(\O_{\mathbb F_0}(f))$ is surjective, hence $\varphi_{\xi_{\F}}$ restricts to $\varphi_{\O_{\mathbb F_0}(f)}$ on $\P(\O_{\mathbb F_0}(f)) \cong \mathbb F_0$, a morphism that contracts all curves $C \sim f$. Since $V$ is  normal by Subclaim \ref{normale}, this proves \eqref{luogoecc} and moreover the proof of Subclaim \ref{c4} shows that  $\P(\O_{\mathbb F_0}(f))$ is contracted to the line $S_0$. Now for every $t \in \P^1$ let $f_t$ be the corresponding line on $\P(\O_{\mathbb F_0}(f)) \cong \mathbb F_0$. We have an exact sequence
$$0 \to \F(-f_t) \to \F \to \F_{|f_t} \to 0$$
with $H^1(\F(-f_t)) = H^1(\O_{\mathbb F_0} \oplus \O_{\mathbb F_0}(2C_0))=0$, showing that $H^0(\F) \cong H^0(\xi_{\F}) \to H^0(\xi_{\F_{|f_t}}) \cong H^0(\F_{|f_t})$ is surjective. Since $\F_{|f_t} \cong \O_{\P^1} \oplus \O_{\P^1}(2)$ we deduce that $\varphi_{\xi_{\F}}$ maps $\P(\F_{|f_t})$ onto a quadric cone $\mathcal Q_t \subset \P^3$. On the other hand $\P(\F_{|f_t}) \cap \P(\O_{\mathbb F_0}(f))$ is a curve on $\P(\O_{\mathbb F_0}(f))$ isomorphic to $f_t \subset \mathbb F_0$ and this curve is contracted to the point in $S_0$ corresponding to $t$. Now any point $v \in V \setminus S_0$ belongs to a unique cone $\mathcal Q_t$. On the other hand, if $v \in S_0$ then $v$ is the vertex of the cone $\mathcal Q_t$ where $t \in \P^1$ is the image of $f_t$ on $\P(\O_{\mathbb F_0}(f))$. This clearly gives an isomorphism $V \cong \P^1 \times \mathcal Q$.
\end{proof}
 
\begin{subclaim}
\label{c5bis}
If $S_V = S(0,1,1,2)$, then $\rho(V)=2$ and in fact $V$ is a hyperplane section of the Segre embedding $\P^2 \times \P^2 \subset \P^8$.
\end{subclaim}
\begin{proof}
By Subclaim \ref{c4} we have that $V = \varphi_{\xi_{\F}}(\P(\F))$, where $\F = \O_{\mathbb F_1}(C_0+f) \oplus \O_{\mathbb F_1}(C_0+2f)$. Consider the isomorphism $\P(\O_{\mathbb F_1}(C_0+f)) \cong \mathbb F_1$ and let $\widetilde C_0$ be the curve on $\P(\O_{\mathbb F_1}(C_0+f))$ isomorphic to $C_0 \subset \mathbb F_1$. We first claim that $\widetilde C_0$ is the unique curve contracted by $\varphi_{\xi_{\F}}$. To see the latter, let $p : \P(\F) \to \mathbb F_1$ be the projection map. From Grothendieck's relation 
$$\xi_{\F}^2 = \xi_{\F} p^*c_1(\F) - p^*c_2(\F)$$ 
we deduce, setting $R = p^*(C_0) p^*(f)$, that
\begin{equation}
\label{gro2}
\xi_{\F}^2 = 2 \xi_{\F}p^*C_0  + 3\xi_{\F} p^*f - 2R
\end{equation}
and therefore that $\xi_{\F}^2p^*C_0 = 1, \xi_{\F}^2p^*f = 2$. Also \eqref{gro2} gives that $N^2(\P(\F))$ is generated by $\xi_{\F}p^*C_0, \xi_{\F}p^*f$ and $R$. Let $C \subset \P(\F)$ be an irreducible curve contracted by $\varphi_{\xi_{\F}}$. Then $C \equiv a\xi_{\F}p^*C_0+b\xi_{\F}p^*f+cR$ for some integers $a,b,c$. Since $p^*(C_0+f)$ and $p^*f$ are nef, we get that $0 \le C \cdot p^*(C_0+f) = b$ and $0 \le C \cdot p^*f = a$. Now
$$0 = \xi_{\F} \cdot C = a+2b+c$$
gives that $c = -a-2b$, hence either $a>0$ or $b>0$. Also $\P(\O_{\mathbb F_1}(C_0+f)) \sim \xi_{\F} -p^*C_0-2p^*f$ and therefore
$$C \cdot \P(\O_{\mathbb F_1}(C_0+f)) = [a\xi_{\F}p^*C_0+b\xi_{\F}p^*f-(a+2b)R] \cdot (\xi_{\F} -p^*C_0-2p^*f)=-a-b < 0$$
so that $C \subset \P(\O_{\mathbb F_1}(C_0+f))$. On the other hand, via the isomorphism $\P(\O_{\mathbb F_1}(C_0+f)) \cong \mathbb F_1$ we have that 
$$(\xi_{\F})_{|\P(\O_{\mathbb F_1}(C_0+f))} = \xi_{\O_{\mathbb F_1}(C_0+f)} \cong C_0+f$$
so that
$$0 = \xi_{\F} \cdot C = (C_0+f) \cdot p(C)$$
and therefore $p(C)=C_0$, hence $C=\widetilde C_0$. This proves the above claim.

Since $V$ is normal by Subclaim \ref{normale}, we have that ${\rm Exc}(\varphi_{\xi_{\F}})=\widetilde C_0$ and therefore $\varphi_{\xi_{\F}}$ is an isomorphism outside $\widetilde C_0$ and contracts the latter to a singular point of $V$, namely to $S_0$. It follows that $(\varphi_{\xi_{\F}})_* \O_{\P(\F)} \cong \O_V$ and therefore $(\varphi_{\xi_{\F}})_* \O_{\P(\F)}^* \cong \O_V^*$. Moreover $R^1 (\varphi_{\xi_{\F}})_* \O_{\P(\F)}^*$ is a skyscraper sheaf supported on the point $S_0$ and on $S_0$ it is $H^1(\widetilde C_0, \O_{\widetilde C_0}^*)$, so that 
$$H^0(V,R^1 (\varphi_{\xi_{\F}})_* \O_{\P(\F)}^*) \cong H^1(\widetilde C_0, \O_{\widetilde C_0}^*) \cong H^1(\P^1, \O_{\P^1}^*) \cong \Z.$$ 
Now the Leray spectral sequence gives rise to the exact sequence
$$0 \to H^1(V, \O_V^*) \to H^1(\P(\F), \O_{\P(\F)}^*) \to H^0(V, R^1 (\varphi_{\xi_{\F}})_* \O_{\P(\F)}^*) \to 0.$$
Since $H^1(\P(\F), \O_{\P(\F)}^*) \cong \Z^3$, we deduce that $\Pic(V)$ has rank $2$.

Now let us see that $V$ is a hyperplane section of the Segre embedding $\P^2 \times \P^2 \subset \P^8$. Note that this also proves, by Lefschetz's theorem, that $\rho(V)=2$.

Under the morphism $\varphi_{\xi_{\F}}$ we see that $\P(\O_{\mathbb F_1}(C_0+f))$ gets mapped onto $\P^2$ with $\widetilde C_0$ contracted to a point $P \in \P^2$, while $\P(\O_{\mathbb F_1}(C_0+2f))$ gets mapped isomorphically onto the rational normal surface scroll $S(1,2) \subset \P^4$. Moreover the fibers of $p$ not meeting $\widetilde C_0$, give rise to a family of disjoint lines meeting $S(1,2)$ and $\P^2$ in a point and giving an isomorphism $S(1,2) \setminus S(1) \cong \P^2-\{P\}$, while the fibers meeting $\widetilde C_0$, give rise to lines meeting $S(1)$ and passing through $P$. One can put coordinates so that $V$, that is the union of these lines, is a hyperplane section of the Segre embedding $\P^2 \times \P^2 \subset \P^8$. To see this let $P=(1:0:0) \in \P^2$, consider the line $(0:s:t)$ and parametrize the lines through $P$ with coordinates $(a:b)$, so that a point in $\P^2$ has coordinates $(a:bs:bt)$. Parametrize the points of $S(1,2)$, join of the line $(s:0:0:t:0:0)$ and the conic $(0:s^2:st:0:st:t^2)$, by $(as:bs^2:bst:at:bst:bt^2)$, inside the hyperplane $Z_2-Z_4=0$ in $\P^5$. 
Now a point in a line joining a point of $S(1,2)$ and of $\P^2$ has coordinates $(ma:mbs:mbt:nas:nbs^2:nbst:nat:nbst:nbt^2)$. It follows that $V$, that is the locus of these points, is a hyperplane section of the Segre embedding $\P^2 \times \P^2 \subset \P^8$: if we have coordinates $(X_0:X_1:X_2)$ and $(Y_0:Y_1:Y_2)$ this can be seen by setting $m=X_0, ns=X_1, nt=X_2, a=Y_0, bs=Y_1$ and $bt=Y_2$.
\end{proof}

\begin{subclaim}
\label{c6}
If $S_V=S(0,1,1,2)$ then $f_x \cong \P(T_{\P^2})$, while if $S_V = S(0,0,2,2)$ then $f_x \cong \P^1 \times \P^1 \times \P^1$.
\end{subclaim}
\begin{proof}
Let $V=f_u$ be such that $S_V=S(0,1,1,2)$. Then $V = \varphi_{\O_{\P(\F)}(1)}(\P(\F))$ by Subclaim \ref{c4} and therefore $\rho(V)=2$ by Subclaim \ref{c5bis}. Also the only singular point of $V$ is an ordinary double point by \cite[Thm.~2.9]{fu4}, hence in particular is terminal. It follows by \cite[Thm.~1.4]{jr} that $\rho(f_x)=2$ and therefore that $f_x \cong \P(T_{\P^2})$. 

Assume now that $S_V = S(0,0,2,2)$ so that we know by Subclaim \ref{c5} that  $V = \P^1 \times \mathcal Q \subset \P^7$. Moreover $T:= {\widetilde \Phi}(V)$ is a surface by Subclaim \ref{smq}. We now show that $T \cong \mathcal Q$. In fact, let $U$ be the open subset of $\widetilde{\Phi(X)}$ such that the fibers of $\widetilde \Phi$ have dimension $1$. For every $v \in V$ we know that $F_v \subset f_v=f_u=V$, hence $F_v$ is a line, since $\P^1 \times \mathcal Q$ does not contain planes. Then $V \subset \widetilde \Phi^{-1}(U)$ and, similarly, $f_x \subset \widetilde \Phi^{-1}(U)$. Now there is a rank $2$ vector bundle $\G$ on $U$ such that $\widetilde \Phi^{-1}(U) \cong \P(\G)$ and $\O_X(1)_{|\P(\G)} \cong \O_{\P(\G)}(1)$. Therefore $V \cong \P(\G_{|T})$ and, since $V$ is normal by Subclaim \ref{normale}, so must be $T$.  But it is easily seen that $\P^1 \times \mathcal Q$ has only one $\P^1$-bundle structure over a normal surface, namely the second projection $\P^1 \times \mathcal Q \to \mathcal Q$ and therefore $T \cong \mathcal Q$. Now \eqref{diag} gives rise to the commutative diagram
$$\xymatrix{\widetilde \Phi^{-1}(U) \cong \P(\G) \ar[dr]_{{\phi_{\tau}}_{|\P(\G)}} \ar[r]^{\hskip 1cm \widetilde \Phi_{|\P(\G)}} & U \ar[d]^{g_{|U}} \\ & X'}.$$
In particular $g_{|U}$ gives a deformation of $\mathcal Q$ to $T_x:= \widetilde \Phi(f_x)$. Also we know that there is an ample line bundle $\L$ on $X'$ such that $K_X+2H = \phi_{\tau}^*\L$ and therefore, restricting to $\widetilde \Phi^{-1}(U) \cong \P(\G)$ we get, setting $p=\widetilde \Phi_{|\P(\G)}$, that $p^*((g_{|U})^*\L) =  
p^*(K_U + \det \G)$, so that $K_U + \det \G = (g_{|U})^*\L$. In particular this shows that both $K_{\mathcal Q}$ and $K_{T_x}$ are restrictions of a line bundle on $U$. But then $K_{T_x}^2 = K_{\mathcal Q}^2=8$. Since $\P^1 \times \P^1 \times \P^1$ and $\P(T_{\P^2})$ have as only $\P^1$-bundle structures over a smooth surface the projections to $\P^1 \times \P^1$ or to $\P^2$, we have that $T_x$ is either $\P^1 \times \P^1$ when $f_x \cong \P^1 \times \P^1 \times \P^1$ or $\P^2$ when $f_x \cong \P(T_{\P^2})$. Thus $f_x \cong \P^1 \times \P^1 \times \P^1$. 
\end{proof}
To finish the proof of Claim \ref{dp6} observe that we already know the smooth fibers of $\phi_{\tau}$ by Claim \ref{liscie}. On the other hand, if $V=f_u$ is a singular fiber, Subclaims \ref{c1}, \ref{c4} and Subclaim \ref{c6} imply that all singular fibers are either all of type $S(0,0,2,2)$ or all of type $S(0,1,1,2)$. Hence, when $f_x \cong \P^1 \times \P^1 \times \P^1$ we get that all singular fibers are $\P^1 \times \mathcal Q$ by Subclaims \ref{c6} and \ref{c5}. In particular $F_u$ is a line for every $u \in X$ and it follows by Lemma \ref{ii} that $(X, \O_X(1), \E)$ is a linear Ulrich triple with $p = \widetilde \Phi$ and $b = 3$. This gives case (vi2). Finally if $f_x \cong \P(T_{\P^2})$ then we are in case (vi3) by Subclaim \ref{c4}. This completes the proof of Claim \ref{dp6}.
\end{proof}

\begin{claim} 
\label{nome}
In case (d.3), $(X,\O_X(1),\E)$ is as in (vii).
\end{claim}
\begin{proof}
By Claim \ref{fe2} we are in case (fact) for $\phi_{\tau}$, hence $\phi_{\tau}$ factorizes through $\widetilde \Phi$. Suppose first that there is an $x_0 \in X$ such that $F_{x_0}$ is a linear $\P^k$ with $2 \le k \le 3$. Since $F_{x_0} \subset f_{x_0}$ and $\dim f_{x_0} = 2$, it follows that $\P^2 = F_{x_0} \subset f_{x_0}$, so that $f_{x_0}$ is a reducible quadric. If $f_{x_0} = F_{x_0} \cup M$ with $M$ a plane distinct from $F_{x_0}$, note that $f_{x_0}$ spans a $\P^3$: If not, then $F_{x_0} \cap M$ is a point but then a general hyperplane section of $X$ gives a conic fibration over a curve with a fiber union of two disjoint lines, contradicting the fact that all fibers must have arithmetic genus $0$. Now $\Phi_{|M} : M \to \mathbb G(r-1,\P H^0(\E))$ contracts the line $F_{x_0} \cap M$ to a point, hence it is constant, so that $\Phi(F_{x_0})=\Phi(M)$ and therefore $M \subset F_{x_0}$, a contradiction. Hence $f_{x_0}$ is a double plane with $(f_{x_0})_{red}=F_{x_0}$. Again $f_{x_0}$ spans a $\P^3$: In fact, a general hyperplane section of $X$ gives a conic fibration over a curve with a fiber which is a double line with arithmetic genus $0$. But such a double line spans a plane, hence $f_{x_0}$ spans a $\P^3$. Since $F_{x_0}$ is a fiber of $\Phi$ we have that $(\det \E)_{|F_{x_0}} \cong \O_{\P^2}$. Now consider the exact sequence (see for example \cite[Proof of Prop.~4.1]{be})
$$0 \to \O_{\P^2}(-1) \to \O_{f_{x_0}}^* \to \O_{F_{x_0}}^* \to 1.$$
Since $H^1(\O_{\P^2}(-1))=0$ we get that the restriction map 
$$\Pic(f_{x_0})=H^1( \O_{f_{x_0}}^*) \to H^1(\O_{F_{x_0}}^*)=\Pic(F_{x_0})$$ 
is injective, hence $(\det \E)_{|f_{x_0}} \cong \O_{f_{x_0}}$. Now $f_{x_0}$ is a quadric in $\P^3$, hence $h^0(\O_{f_{x_0}})=1$.
Therefore $\Phi(f_{x_0})$ is a point and this gives that, scheme-theoretically, $f_{x_0} = F_{x_0}$. But this contradicts \cite[Thm.~2]{ls}.
It follows that $\dim F_u=1$ for every $u \in X$ and then we can apply Lemma \ref{ii}. 
\end{proof}

\begin{claim} 
In case (d.4), $(X,\O_X(1),\E)$ is as in (viii).
\end{claim}
\begin{proof}
By Claim \ref{fe2} we are in case (fact) for $\phi_{\tau}$, hence $F_u = f_u$ for every $u \in X$. Thus $\phi_{\tau} = \widetilde \Phi$ and we just apply Lemma \ref{ii}.
\end{proof}

\begin{claim} 
In case (d.5), $(X,\O_X(1),\E)$ is as in (ix).
\end{claim}
\begin{proof}
By Claim \ref{fe2} we are in case (fact) for $\phi_{\tau}$, hence there is a morphism $\psi : \widetilde{\Phi(X)} \to X'$ such that $\phi_{\tau} = \psi \circ \widetilde \Phi$. Now $F_x$ is a line, hence $f_x = F_x = \widetilde F_x$. Therefore the general fiber of $\psi$ must be a point, that is $\psi$ is generically finite of degree $1$. 
\end{proof} 

\begin{claim} 
Case (e) does not occur.
\end{claim}
\begin{proof}
This follows by Proposition \ref{caso(e)}. 
\end{proof}

We are therefore left with case (d.1) in which $(X,\O_X(1))$ is a Mukai fourfold, that is $K_X = -2H$. We will use Mukai's classification \cite[Thm.~7]{mu}, \cite[Table 0.3]{wi1}. 

We remark that we can assume that $c_1(\E)^3 \ne 0$, for otherwise \cite[Cor.~4]{ls} implies that $(X,\O_X(1))$ is a linear $\P^{4-b}$-bundle over a smooth variety of dimension $b \le 2$, contradicting $K_X = -2H$.

We will divide the proof in two subcases of case (d.1):
\begin{itemize}
\item [(B.1)] there exists $x_0 \in X$ such that $\dim F_{x_0} \ge 2$.
\item [(B.2)] $\dim F_u = 1$ for every $u \in X$.
\end{itemize} 

In the second case we have the following general fact.

\begin{claim}
\label{casoii-muk}
In case (B.2) we have that $(X,\O_X(1),\E)$ is a linear Ulrich triple with $p = \widetilde \Phi, B=\widetilde{\Phi(X)}$ and $b = 3$. Moreover the $\P^1$-bundle structure $p: X \cong \P(\F) \to B$ occurs in the following cases:
\begin{itemize}
\item [(i)] $X= \P^1 \times M$ and either $\F \cong L^{\oplus 2}$ where $K_M=-2L$, or $B=\P^1 \times \P^2$ and either $M=V_7$ and $\F=\O_{\P^1}(1)\boxtimes (\O_{\P^2}(1) \oplus \O_{\P^2}(2))$ or $M=\P(T_{\P^2})$ and $\F=\O_{\P^1}(1)\boxtimes T_{\P^2}$. 
\item [(ii)] $(X, \O_X(1))$ is as in Examples 6, 8 or 9 in Mukai's classification and the $\P^1$-bundle structure is the one given in \cite[Thm.~7 and Ex.~1-9]{mu} (see also \cite[Table 0.3]{wi1}).
\end{itemize}
\end{claim}
\begin{proof} 
The first fact follows by Lemma \ref{ii}. As for the $\P^1$-bundle structure, this is a well-known fact that can be easily proved either by following the proof of \cite[Thm.~0.1]{wi2} or simply using the fact that the morphism $p$ is given by a globally generated line bundle $\L$ on $X$ with $\L^4=0$ (most calculations of this type are done in the course of the proofs). 
\end{proof}

\begin{claim} 
\label{b21}
If $X$ is a Mukai fourfold of product type, then $(X, \O_X(1),\E)$ is as in (x1) or (x2).
\end{claim}
\begin{proof} 
We have that $X \cong \P^1 \times M$ with $M$ a Fano threefold of even index, hence $\O_X(1) \cong \O_{\P^1}(1)\boxtimes L$ where $K_M=-2L$. Let $p_i, i=1, 2$ be the two projections. Then $\det \E = p_1^*(\O_{\P^1}(a)) + p_2^*N$ for some $a \in \Z$ and some line bundle $N$ on $M$. Now
$$0 = c_1(\E)^4 = 4aN^3.$$
If $a=0$, then $N$ is globally generated and \cite[Lemma 5.1]{lo} gives that there is a vector bundle $\H$ on $M$ such that $\E \cong p_2^*\H$. Hence we are in case (x1) by \cite[Lemma 4.1]{lo}. Suppose now that $N^3=0$. Then $\rho(M) \ge 2$, for otherwise we have that then $N \cong \O_M$ and therefore $c_1(\E)^2=0$, a contradiction.
Hence the only possibility is that either $M= V_7, \P^1\times \P^1\times \P^1$ or a hyperplane section of the Segre embedding of $\P^2 \times \P^2$. 

First, assume that we are in case in case (B.1).

We have $\P^k = F_{x_0} \subset X$ for $k=2$ or $3$. Since ${p_1}_{|F_{x_0}} : \P^k \to \P^1$ must be constant, it follows that $F_{x_0}=\{y\} \times Z$, where $y \in \P^1$ and $Z$ is a linear $\P^k$ contained in $M$. In most cases we will write this as $F_{x_0} \subset M$. Therefore $k=2$ and $L_{|F_{x_0}}=\O_{\P^2}(1)$.

If $M= V_7$, let $\varepsilon : V_7 \to \P^3$ be the blow-up map with exceptional divisor $E$. Let $\widetilde H$ be the pull back of a plane, so that $L = 2\widetilde H -E$. Note that it cannot be that $E \cap Z = \emptyset$, for otherwise $\O_{\P^2}(1) = L_{|Z} = 2\widetilde H_{|Z}$. Therefore $\dim E \cap Z \ge 1$ and the morphism $\varepsilon_{|Z} : \P^2=Z \to \P^3$ contracts $E \cap Z$ to a point, hence it is constant, so that $Z = E$. Now $h=\id_{\P^1} \times \varepsilon : X \cong \P^1 \times V_7 \to \P^1 \times \P^3$ contracts $F_{x_0}$ to a point. But then the Dichotomy Lemma implies that $F_x \subset h^{-1}(h(x))=\{x\}$, a contradiction. 

If $M= \P^1\times \P^1\times \P^1$ observe that if $p_i, 1 \le i \le 3$ is a projection, then ${p_i}_{|F_{x_0}}: \P^k=F_{x_0} \to \P^1$ must be constant, giving a contradiction. 

If $M$ a hyperplane section of the Segre embedding of $\P^2 \times \P^2$, then, as is well known, $F_{x_0}  = \P^2 \times \{y\}$ or $\{z\} \times \P^2$. In the first case, if $p_2: \P^2 \times \P^2 \to \P^2$ is the second projection, then ${p_2}_{|M}(F_{x_0})  = \{y\}$, hence $F_{x_0} \subset ({p_2}_{|M})^{-1}(y)$. But this is a contradiction since $\dim ({p_2}_{|M})^{-1}(y) = 1$, as one can see using the isomorphism $M \cong \P(T_{\P^2})$ and that ${p_2}_{|M} : \P(T_{\P^2}) \to \P^2$ is the projection map. In the second case a similar contradiction can be obtained. 

Next assume that we are in case in case (B.2), so that we can apply Claim \ref{casoii-muk}.

If $\F \cong L^{\oplus 2}$, setting $\G' = \G(L)$ we get that  $\E \cong p^*(\G'(L))$ and this gives again case (x1).

Now when $M=\P(T_{\P^2})$ and $\F=\O_{\P^1}(1)\boxtimes T_{\P^2}$ we are in case (x2) by \eqref{van-trip}.

Finally consider the case $M=V_7$ and $\F=\O_{\P^1}(1)\boxtimes (\O_{\P^2}(1) \oplus \O_{\P^2}(2))$ on $B=\P^1 \times \P^2$. For ease of notation we will set, for any $c, d \in \Z$, $\O_B(c,d)=\O_{\P^1}(c)\boxtimes \O_{\P^2}(d), \O_B(c) = \O_B(c,c)$ and $\H(c,d) = \H \otimes \O_B(c,d)$ for any sheaf $\H$ on $B$. Now \eqref{van-trip} gives in particular the vanishings
$$H^j(\G(-s))= H^j(\G(-1,-2))= H^j(\G(-2,-3))=0 \ \hbox{for} \ j \ge 0, 0 \le s \le 2.$$
Since the same vanishings hold for any direct summand of $\G$, we can assume that $\G$ is indecomposable. Note that the vanishings give that $\G(1)$ is an Ulrich bundle for $(B, \O_B(1))$, hence it is also ACM by Remark \ref{gen}(iii). Then \cite[Thm.~B]{fms} gives that either $\G \cong \O_{\P^1} \boxtimes \Omega_{\P^2}(1)$ or $\G$ fits into an exact sequence
$$0 \to \O_B(-1,0)^{\oplus a} \to \G \to \O_B(1,-1)^{\oplus b} \to 0.$$
In the first case we find the contradiction
$$0 = H^3(\G(-2,-3))= H^3(\O_{\P^1}(-2)\boxtimes \Omega_{\P^2}(-2)) \cong H^1(\O_{\P^1}(-2)) \otimes H^2(\Omega_{\P^2}(-2))\ne 0.$$
In the second case we have an exact sequence
$$0 \to \O_B(-2)^{\oplus a} \to \G(-1,-2) \to \O_B(0,-3)^{\oplus b} \to 0.$$
But $H^2(\G(-1,-2))=H^3(\O_B(-2))=0$ giving, for $b \ne 0$, the contradiction $H^2(\O_B(0,-3))=0$. Therefore $b = 0$ and $\G=\O_B(-1,0)^{\oplus a}$, giving the contradiction $0=H^3(\G(-2,-3))=H^3(\O_B(-3))^{\oplus a} \ne 0$. Thus this case does not occur. This proves Claim \ref{b21}.
\end{proof}

To finish the proof of Theorem \ref{main3}, it remains to consider the Mukai fourfolds not of product type, which, by \cite[Thm.~7]{mu}, are linear sections of the varieties listed in \cite[Ex.~1-9]{mu}. 

\begin{claim}
\label{b21bis}
If $X$ is a Mukai fourfold not of product type, then $(X, \O_X(1), \E)$ is as in (x3)-(x5).
\end{claim}
\begin{proof} 
We refer to \cite[Thm.~7]{mu}, \cite[Table 0.3]{wi1}. 

In example 1 we have that $X$ is a double cover $f : X \to \P^2 \times \P^2$ ramified along a divisor of type $(2,2)$ and $\O_X(1)=f^*L, L=\O_{\P^2}(1) \boxtimes \O_{\P^2}(1)$. For $i=1, 2$ let $p_i : \P^2 \times \P^2 \to \P^2$ be the projections and set $q_i= f \circ p_i : X \to \P^2$. Let $Y$ be a smooth hyperplane section of $X$ and set $h_i = {q_i}_{|Y} : Y \to \P^2$. By Lefschetz's theorem we know that the restriction map $\Pic(X) \to \Pic(Y)$ is an isomorphism. On the other hand, $Y$ is just the Fano threefold listed as No.~6b in Mori-Mukai's list \cite[Table 2]{mm1} and it follows from \cite[Thm.~5.1 and proof of Thm.~1.7]{mm2} that $\Pic(Y)$ is generated by $h_i^*(\O_{\P^2}(1)), i=1,2$. Therefore $\Pic(X)$ is generated by $A = q_1^*(\O_{\P^2}(1))$ and $B=q_2^*(\O_{\P^2}(1))$. Now $\det \E = aA + bB$, for some $a, b \in \Z$, hence
$$0 = c_1(\E)^4 = 12a^2b^2.$$
Therefore either $a=0$ or $b=0$, giving the contradiction $c_1(\E)^3=0$.
Thus example 1 is excluded.

In examples 2 and 4, setting $Q=Q_3$, we have that $X$ is the hyperplane section of $\P^2 \times Z$, where $Z=\P^3$ (respectively $Q$) under the Segre embedding given by $L = \O_{\P^2}(1) \boxtimes M$, with $M = \O_{\P^3}(2)$ (resp. $M=\O_Q(1))$ and $\O_X(1)=(\O_{\P^2}(1) \boxtimes \O_{\P^3}(1))_{|X}$ (resp. $\O_X(1) = L_{|X}$). Let $p_i, i=1, 2$ be the two projections on $\P^2 \times Z$ and let $q={p_2}_{|X}$. By Lefschetz's theorem we know that $\Pic(X)$ is generated by $A_{|X}$ and $B_{|X}$ where $A=p_1^*(\O_{\P^2}(1))$ and $B=p_2^*(\O_{\P^3}(1))$ (resp. $B=p_2^*(\O_Q(1))$). Now $\det \E = aA_{|X} + bB_{|X}$, for some $a, b \in \Z$. Hence, in both cases,
$$0 = c_1(\E)^4 = \sum\limits_{j=0}^4 \binom{4}{j}a^jb^{4-j}A^j\cdot B^{4-j}X=4ab^2(b+3a).$$
If $a \ne 0$, then either $b=0$ and $\det \E = aA$, hence $c_1(\E)^3=0$, a contradiction, or $b=-3a$ and $\det \E = a(A - 3B)$. But this is not nef: if $a>0$, let $y \in \P^2$ and choose a curve $C$ in the surface $(\{y\} \times Z)\cap X$. Then $A \cdot C=0, B \cdot C = \frac{1}{2} L \cdot C > 0$ (resp. $B \cdot C = L \cdot C > 0$), hence $\det \E \cdot C =-3aB \cdot C < 0$; if $a<0$, let $z \in Z$ and let $C=(\P^2 \times \{z\}) \cap X$. Then $B \cdot C=0, A \cdot C = L \cdot C > 0$, hence $\det \E \cdot C =aL \cdot C < 0$. Since $\det \E$ is globally generated, this is a contradiction. Therefore $a=0, \det \E = q^*(\O_Z(b))$ and it follows by \cite[Lemma 5.1]{lo} that there is a rank $r$ vector bundle $\G$ on $Z$ such that $\E \cong q^*\G$. 

When $Z=Q$, exactly as in the case of the hyperplane section of $\P^2 \times \P^3$ (proof of Claim \ref{serve}), we see that $H^i(\E(-pH))=0$ for all $i \ge 0$ and $1 \le p \le 4$, together with the K\"unneth formula implies that $H^j(\G(-2)(-s))=0$ for all $j \ge 0$ and $1 \le s \le 3$. But then $\G(-2)$ is an Ulrich vector bundle for $(Q, \O_Q(1))$, hence $r$ is even and $\G \cong \mathcal S(2)^{\oplus (\frac{r}{2})}$ by \cite[Lemma 3.2(iv)]{lms}, so that we are in case (x3). Hence we are done with example 4. 

When $Z=\P^3$ we will reach a contradiction. To this end set $\L = \O_{\P^2}(1) \boxtimes \O_{\P^3}(1)$ and consider the exact sequence
$$0 \to (p_2^*\G)(-p\L-X) \to (p_2^*\G)(-p\L) \to \E(-pH) \to 0.$$
Since $H^i(\E(-pH))=0$ for all $i \ge 0$ and $1 \le p \le 4$ we deduce that
$$H^i((p_2^*\G)(-p\L-X)) \cong H^i((p_2^*\G)(-p\L)) \ \hbox{for all} \ i \ge 0 \ \hbox{and} \ 1 \le p \le 4.$$
Now the K\"unneth formula gives that
$$h^0(\O_{\P^2}(p-2))h^{i-2}(\G(-p-2))=h^0(\O_{\P^2}(p-3))h^{i-2}(\G(-p)) \ \hbox{for} \ i \ge 2 \ \hbox{and} \ 1 \le p \le 4$$
and one easily sees that this implies that $H^j(\G(-4))=H^j(\G(-6))=0$ and $h^j(\G(-3))=3h^j(\G(-5))$ for all $j \ge 0$. Setting $\H = \G(-3)$ we find
\begin{equation}
\label{coo}
H^j(\H(-1))=H^j(\H(-3))=0 \ \hbox{and} \ h^j(\H)=3h^j(\H(-2)) \ \hbox{for all} \ j \ge 0.
\end{equation}
Now let $M$ be a plane in $\P^3$ and consider, for $l \in \Z$, the exact sequences
\begin{equation}
\label{exa}
0 \to \H(-l-1) \to \H(-l) \to \H(-l)_{|M} \to 0.
\end{equation}
We have $H^0(\H(-1))=0$ by \eqref{coo}, hence also $H^0(\H(-2))=0$ and then \eqref{coo} gives that \begin{equation}
\label{exa2}
h^0(\H)=3h^0(\H(-2))=0. 
\end{equation}
Now $H^1(\H(-1))=0$ by \eqref{coo}, hence \eqref{exa} with $l=0$ implies that $H^0(\H_{|M})=0$, hence also $H^0(\H(-1)_{|M})=0$. Setting $l=1$ in \eqref{exa} and using again $H^1(\H(-1))=0$, we deduce that 
$H^1(\H(-2))=0$. We have $H^3(\H(-3))=0$ by \eqref{coo}, hence also $H^3(\H(-2))=0$. Since $H^2(\H(-1))=0$ by \eqref{coo}, then \eqref{exa} with $l=1$ implies that $H^2(\H(-1)_{|M})=0$, hence also $H^2(\H_{|M})=0$. Therefore \eqref{exa} with $l=0$ and \eqref{coo} imply that $H^2(\H)=0$, whence also $H^2(\H(-2))=0$ by \eqref{coo}. Thus we have proved that $H^j(\H(-2))=0$ for $j \ge 0$ and together with \eqref{exa} this implies that $\H$ is an Ulrich vector bundle on $\P^3$. But this contradicts \eqref{exa2}.

Thus we are done with example 2.

In example 3 we have that $X$ is twice a hyperplane section of $\P^3 \times \P^3$ under the Segre embedding given by $L = \O_{\P^3}(1) \boxtimes \O_{\P^3}(1)$ and $\O_X(1) = L_{|X}$. Let $p_i, i=1, 2$ be the two projections on $\P^3 \times \P^3$ and let $p = {p_1}_{|X}$ and $q={p_2}_{|X}$. By Lefschetz's theorem we know that $\Pic(X)$ is generated by $A=p^*(\O_{\P^3}(1))$ and $B=q^*(\O_{\P^3}(1))$ with $A^4=B^4=0, A \cdot B^3 = A^3 \cdot B=1$ and $A^2 \cdot B^2=2$. Now $\det \E = aA + bB$, for some $a, b \in \Z$, hence
$$0 = c_1(\E)^4 = 4ab(b^2+3ab+a^2).$$
and, the case $b=0$ being completely similar, we can assume that $a=0$ and $\det \E=q^*(\O_{\P^3}(b))$. It follows by \cite[Lemma 5.1]{lo} that there is a rank $r$ vector bundle $\G$ on $\P^3$ such that $\E \cong q^*\G$. We now claim that $\G(-3)$ is an Ulrich vector bundle for $(\P^3, \O_{\P^3}(1))$. To this end observe that, as is well known, we can see $X$ as the hyperplane section of $\P(T_{\P^3})$ embedded by the tautological line bundle $\xi$ and $q := (p_2)_{|\P(T_{\P^3})} : \P(T_{\P^3}) \to \P^3$ is just its projection map. Thus $\E \cong (q^*\G)_{|X}$ and we have an exact sequence
$$0 \to (q^*\G)(-(s+1)\xi) \to (q^*\G)(-s\xi) \to \E(-sH) \to 0.$$
Since $H^i(\E(-sH))=0$ for all $i \ge 0$ and $1 \le s \le 4$ we deduce that
$$H^i((q^*\G)(-(s+1)\xi)) \cong H^i((q^*\G)(-s\xi)) \ \hbox{for all} \ i \ge 0 \ \hbox{and} \ 1 \le s \le 4.$$
Setting $j=i-2$, the Leray spectral sequence implies that 
$$H^j(\G \otimes R^2q_*(-(s+1)\xi)) \cong H^j(\G \otimes R^2q_*(-s\xi)) \ \hbox{for all} \ j \ge 0 \ \hbox{and} \ 1 \le s \le 4.$$
Since $R^2q_*(-2\xi)=0$ we deduce that $H^j(\G \otimes R^2q_*(-h\xi))=0$ for all $j \ge 0$ and $3 \le h \le 5$. On the other hand, $R^2q_*(-h\xi) \cong (S^{h-3} \Omega_{\P^3})(-4)$ and therefore we have that $H^j(\G(-4) \otimes S^k \Omega_{\P^3})=0$ for all $j \ge 0$ and $0 \le k \le 2$. But this is condition (6.4) in \cite{lo} (applied to $\G(-4)$), and it is proved there that this implies that $\G(-3)$ is an Ulrich vector bundle for $(\P^3, \O_{\P^3}(1))$. Hence $\G \cong \O_{\P^3}(3)^{\oplus r}$ by Remark \ref{kno1} and we are in case (x4). Thus we are done with example 3.

In example 5 we have that $X$ is the blow-up of $Q=Q_4$ along a conic. Let $\varepsilon : X \to Q$ be the blow-up map with exceptional divisor $E$ and let $\widetilde H = \varepsilon^*(\O_Q(1))$. We have that $\det \E = a\widetilde H + bE$, for some $a, b \in \Z$. It is easily checked that $E^4=6, \widetilde H\cdot E^3=2, \widetilde H^i\cdot E^{4-i}=0$ for $i =2, 3$ and $\widetilde H^4=2$. But then
$$0 = c_1(\E)^4 = 6b^4+8ab^3+2a^4$$
implies that either $a=b=0$, that is $c_1(\E)=0$, a contradiction, or $b=-a$, and in this case $\det \E=a(\widetilde H-E)$, hence $c_1(\E)^3=0$, again a contradiction. Thus example 5 is excluded.

In example 6 we have, by \cite[Table 0.3 and Thm.~1.1(iii)]{wi1}, \cite[Cor.~page 206]{sw} (or \cite[Thm.~6.5]{mos}), that $(X,\O_X(1))$ has two linear $\P^1$-bundle structures $p: X \cong \P(\F) \to B$ with $\F \cong \N(2)$ or $\F \cong \mathcal S(1)$ where $\N$ is a null-correlation bundle on $B=\P^3$ and $\mathcal S$ is the spinor bundle on $B=Q_3$. We first prove that case (B.1) does not occur. In fact, the Dichotomy Lemma applied to $p$ in the case $B=Q_3$ and Lemma \ref{fe}(ii) give that we are in case (emb). But this implies that $p_{|F_{x_0}} : \P^k=F_{x_0} \to Q_3$ is an embedding, hence that $k=2$. Therefore $p_{|F_{x_0}}$ is the composition of the Veronese embedding $v_s$ of $\P^2$ with an isomorphic projection. By Severi's theorem \cite{se} on projection of smooth surfaces it follows that $s=2$. But it is well known that the Veronese surface in $\P^4$ is not contained in a smooth quadric. 
Therefore we are in case (B.2) and Claim \ref{casoii-muk} gives that $(X,\O_X(1),\E)$ is a linear Ulrich triple, so that $\E \cong p^*(\G(4))$, in the case $B=\P^3$, or $p^*(\G(3))$, in the case $B=Q_3$, where $\G$ is a rank $r$ vector bundle on $B$ satisfying \eqref{van-trip}. When $B=Q_3$ we get case (x5). We now prove that, still in case (B.2), $B=\P^3$ does not occur. In fact, we have that $X \cong \P(\N)$ and setting $\xi= \O_{\P(\N)}(1), R = p^*(\O_{\P^3}(1))$, then $\O_X(1)=\xi+2R$. It is easily verified that $\xi^4=\xi^2 \cdot R^2=R^4=0, \xi^3 \cdot R=-1, \xi \cdot R^3=1$ and $H^4=24$. Now consider the surface $S$ complete intersection of two general divisors in $|H|$. Then $S$ is a K3 surface and $\E_{|S}$ is an Ulrich vector bundle for $(S,H_{|S})$ by Remark \ref{gen}(iv). We can write $\det \E=c_1R$ for some $c_1 \in \Z$. It follows by \cite[Prop.~2.1]{c} that
\begin{equation}
\label{c11}
c_1(\E) \cdot H^3 = c_1(\E_{|S}) \cdot H_{|S}=\frac{3r}{2}H_{|S}^2=36r
\end{equation}
and
\begin{equation}
\label{c2}
c_2(\E_{|S})=\frac{1}{2}c_1(\E_{|S})^2-r(H_{|S}^2-\chi(\O_S))=\frac{1}{2}c_1(\E_{|S})^2-22r.
\end{equation}
From \eqref{c11} we get
$$36r=c_1R \cdot (\xi+2R)^3=11c_1$$
giving $c_1=\frac{36r}{11}$. But $\E_{|S}$ is $\mu$-semistable by \cite[Thm.~2.9]{ch}, hence, using \eqref{c2}, Bogomolov's inequality gives
$$0 \le 2rc_2(\E_{|S})-(r-1)c_1(\E_{|S})^2 = 4c_1^2-44r^2=-\frac{140r^2}{121}$$
a contradiction. Thus the case $X \cong \P(\N)$ does not occur and we are done with example 6.

To see example 7, consider $\varepsilon : Y \to \P^5$ the blow-up along a line with exceptional divisor $E$ and let $\widetilde H = \varepsilon^*(\O_{\P^5}(1))$. Set $L = 2\widetilde H - E$. Then $X$ is a smooth divisor in $|L|$ and $\O_X(1) = L _{|X}$. By Lefschetz's theorem we know that $\Pic(X)$ is generated by $\widetilde H_{|X}$ and $E_{|X}$. It is easily checked that $E^5=-4, \widetilde H\cdot E^4=-1, \widetilde H^5=1$ and $\widetilde H^i\cdot E^{5-i}=0$ for $2 \le i \le 4$. Now $\det \E = a\widetilde H _{|X} + bE_{|X}$, for some $a, b \in \Z$ and
$$0 = c_1(\E)^4 = \sum\limits_{j=0}^4 \binom{4}{j}a^jb^{4-j}\widetilde H^j\cdot E^{4-j}(2\widetilde H - E)=2b^4+2a^4+4ab^3.$$
If $a=0$ we get that $b=0$, that is $c_1(\E)=0$, a contradiction. If $a \ne 0$ then $b=-a$ and $\det \E = a(\widetilde H - E)_{|X}$.  Now observe that $Y \cong \P(\O_{\P^3}^{\oplus 2} \oplus \O_{\P^3}(1))$ and the bundle morphism $q : Y \to \P^3$ is just the morphism defined by $|\widetilde H - E|$. Moreover $L = \xi + q^*(\O_{\P^3}(1))$, where $\xi$ is the tautological line bundle. If we set $h=q_{|X} : X \to \P^3$ we deduce from $\det \E = a(\widetilde H - E)_{|X}$ and \cite[Lemma 5.1]{lo} that there is a rank $r$ vector bundle $\G$ on $\P^3$ such that $\E \cong h^*\G$. Hence $\E \cong (q^*\G)_{|X}$ and we have an exact sequence
$$0 \to (q^*\G)(-(p+1)L) \to (q^*\G)(-pL) \to \E(-pH) \to 0.$$
Since $H^i(\E(-pH))=0$ for all $i \ge 0$ and $1 \le p \le 4$ we deduce that
$$H^i((q^*\G)(-(p+1)L)) \cong H^i((q^*\G)(-pL)) \ \hbox{for all} \ i \ge 0 \ \hbox{and} \ 1 \le p \le 4.$$
Setting $j=i-2$, the Leray spectral sequence implies that 
$$H^j(\G(-p-1) \otimes R^2q_*(-(p+1)\xi)) \cong H^j(\G(-p) \otimes R^2q_*(-p\xi)) \ \hbox{for all} \ j \ge 0 \ \hbox{and} \ 1 \le p \le 4.$$
Since $R^2q_*(-2\xi)=0$ we deduce that $H^j(\G(-h) \otimes R^2q_*(-h\xi))=0$ for all $j \ge 0$ and $3 \le h \le 5$. On the other hand, $R^2q_*(-h\xi) \cong (S^{h-3}(\O_{\P^3}^{\oplus 2} \oplus \O_{\P^3}(-1))(-1)$ and therefore we have that $H^j(\G(-s))=0$ for all $j \ge 0$ and $4 \le s \le 8$. In particular $\G(-3)$ is an Ulrich vector bundle for $(\P^3, \O_{\P^3}(1))$, hence $\G \cong \O_{\P^3}(3)^{\oplus r}$ by Remark \ref{kno1}. But this gives the contradiction 
$$0 =  H^3(\G(-7)) = H^3(\O_{\P^3}(-4)^{\oplus r}) \ne 0.$$ 
Thus example 7 is excluded.

In examples 8 and 9 we first claim that $(X,\O_X(1),\E)$ is a linear Ulrich triple. In fact, we already know that $(X, \O_X(1)) \cong (\P(\F),\O_{\P(\F)}(1))$ where $\F = \O_B(1) \oplus \O_B(m)$, with $B=\P^3$ (respectively $B=Q_3$), $m=3$ (resp. $m=2$). Now let $p : X \cong \P(\O_B \oplus \O_B(m-1)) \to B$ be the projection map, let $\xi$ be the tautological line bundle and let $R=p^*(\O_B(1))$. Set $\det \E = a\xi + bR$, for some $a, b \in \Z$. If $a=0$ we get that $\det \E = p^*(\O_B(b))$ and it follows by \cite[Lemmas 4.1 and 5.1]{lo} that $(X,\O_X(1),\E)$ is a linear Ulrich triple. Assume that $a \ne 0$. In example 8 we have that $\xi^i \cdot R^{4-i}=2$ for $1 \le i \le 4$ and $R^4=0$. Hence 
$$0 = c_1(\E)^4 =2a^4+8a^3b+12a^2b^2+8ab^3$$
so that $b=-\frac{1}{2}a$ and $2 \det \E = a(2\xi - R)$. Let $f$ be a fiber of $p$, so that $0 \le c_1(\E) \cdot f = 2a$, hence $a>0$. But $Q_3 =: Q \cong \P(\O_Q) \subset X$ and $\xi_{|\P(\O_Q)} \cong \O_Q$, while $R_{|\P(\O_Q)} \cong \O_Q(1)$, hence $(2 \det \E)_{|\P(\O_Q)} = \O_Q(-a)$, a contradiction since $\det \E$ is globally generated. In example 9 we have that $\xi^4=8, \xi^3 \cdot R =4, \xi^2 \cdot R^2=2, \xi \cdot R^3=1$ and $R^4=0$. Now 
$$0 = c_1(\E)^4 =8a^4+16a^3b+12a^2b^2+4ab^3$$
hence $b=-a$ and $\det \E = a(\xi - R)$. But $\P^3 \cong \P(\O_{\P^3}) \subset X$ and $\xi_{|\P(\O_{\P^3})} \cong \O_{\P^3}$, while $R_{|\P(\O_{\P^3})} \cong \O_{\P^3}(1)$, hence $(\det \E)_{|\P(\O_{\P^3})} = \O_{\P^3}(-1)$ is not globally generated, a contradiction. Thus we have proved that $(X,\O_X(1),\E)$ is a linear Ulrich triple. Hence $\E \cong p^*(\G(m+1))$, where $\G$ is a rank $r$ vector bundle on $B$ such that $H^j(\G \otimes S^k \F^*)=0$ for $j \ge 0, 0 \le k \le 2$. This gives in particular that $H^j(\G(-s))=0$ for $j \ge 0$ and $0 \le s \le 3$. Hence $\G(1)$ is an Ulrich vector bundle for $(B,\O_B(1))$. When $B=\P^3$ we get by Remark \ref{kno1} that $\G \cong \O_{\P^3}(-1)^{\oplus r}$. But this gives the contradiction
$$0 = H^3(\G(-3))=H^3(\O_{\P^3}(-4))^{\oplus r} \ne 0.$$
Instead, when $B=Q_3$, we get by \cite[Lemma 3.2(iv)]{lms} that $\G \cong \mathcal S(-1)^{\oplus (\frac{r}{2})} \cong S^{\oplus (\frac{r}{2})}$. But this gives the contradiction
$$0 = h^3(\G(-3))=h^3(S(-3))^{\oplus (\frac{r}{2})}=h^0(S^*)^{\oplus (\frac{r}{2})}=h^0(\mathcal S)^{\oplus (\frac{r}{2})} \ne 0.$$
Thus examples 8 and 9 are excluded.

This proves Claim \ref{b21bis}.
\end{proof}
This concludes the proof of Theorem \ref{main3} in Case (B). Therefore the proof of Theorem \ref{main3} is complete. 
\end{proof}
\renewcommand{\proofname}{Proof}


Finally, our last two results.

\renewcommand{\proofname}{Proof of Corollary \ref{main4}}
\begin{proof}
If $(X,\O_X(1),\E)$ is as in (i) or as in (iii), it follows by Remark \ref{kno1} and \cite[Thm.~1]{lms} and Theorem \ref{main3} that $\E$ is Ulrich not big. If $(X,\O_X(1),\E)$ is as in (ii1)-(ii3), it follows by \cite[Prop.~6.1 and 6.2]{lo} for (ii1) and (ii2) and by \cite[(3.5)]{b1} for (ii3), that $\E$ is Ulrich not big. 

Vice versa assume that $\E$ is Ulrich not big. 

Suppose first that $\rho(X) = 1$. 

If $(X,\O_X(1)) \cong (\P^n, \O_{\P^n}(1))$ we are in case (i1) by Remark \ref{kno1}. If $(X,\O_X(1)) \not\cong (\P^n, \O_{\P^n}(1))$ then $c_1(\E)^n>0$ by \cite[Lemma 3.2]{lo} and we are in case (i2) by \cite[Thms.~1 and 2]{lm} and Theorem \ref{main3}. This proves (i).

Now suppose that $(X, \O_X(1))$ is a del Pezzo variety.

If $n=2$, since del Pezzo surfaces are not covered by lines, there is no such $\E$. 

If $n=3$ it follows by \cite[Thm.~3]{lm} that $c_1(\E)^3=0, c_1(\E)^2 \ne 0$. According to the classification of del Pezzo $3$-folds (see for example \cite[\S 1]{lp}, \cite{fu1}), we get that $X$ is either $\P^1 \times \P^1 \times \P^1$ or $\P(T_{\P^2})$ or $V_7$. In the first two instances we are in cases (ii1) and (ii2) by \cite[Prop.~6.1 and 6.2]{lo}. Now consider the third one, that is we have $\varepsilon : X=V_7 \to \P^3$ the blow-up map with exceptional divisor $E$ and let $\widetilde H = \varepsilon^*(\O_{\P^3}(1))$. Hence $\det \E = a\widetilde H + bE$, for some $a, b \in \Z$. It is easily checked that $\widetilde H^3=E^3=1, \widetilde H^i\cdot E^{3-i}=0$ for $i =1, 2$. But then
$$0 = c_1(\E)^3 = a^3+b^3$$
implies that $b=-a$, that is $\det \E = a(\widetilde H - E)$. As is well known, we have that $X \cong \P(\O_{\P^2} \oplus \O_{\P^2}(1))$ and the morphism induced by $|\widetilde H - E|$ is just the projection map $p : X \to \P^2$. We deduce by \cite[Lemma 5.1]{lo} that there is a rank $r$ vector bundle $\G$ on $\P^2$ such that $\E \cong p^*(\G(2))$. But now \cite[Lemma 4.1]{lo} gives that $H^i(\G(-s)) = 0$ for all $i \ge 0$ and $1 \le s \le 3$. In particular $\G$ is an Ulrich vector bundle for $(\P^2, \O_{\P^2}(1))$, hence $\G \cong \O_{\P^2}^{\oplus r}$ by Remark \ref{kno1}. But this gives the contradiction $0 =  H^2(\G(-3)) = H^2(\O_{\P^2}(-3)^{\oplus r}) \ne 0$. Thus this case is excluded.

If $n=4$, using the classification of del Pezzo $4$-folds (see for example \cite[\S 1]{lp}, \cite{fu1}) and case (i), we see that $X=\P^2 \times \P^2 \subset \P^8$. Then we are in case (ii3) by \cite[Thm.~3]{lms}. This proves (ii).

Now suppose that $(X, \O_X(1))$ is a Mukai variety.

We have that $n = 4$ since the Mukai $n$-folds are not covered by lines for $n \le 3$. Therefore we are in cases (x1)-(x5) of Theorem \ref{main3}. This proves (iii).
\end{proof}
\renewcommand{\proofname}{Proof}

\renewcommand{\proofname}{Proof of Corollary \ref{main5}}
\begin{proof}
If $(X,\O_X(1),\E)$ is as in (i) or (ii), it follows by Corollary \ref{main4}, Theorem \ref{main3}, \cite[(3.5)]{b1} and \cite[Prop.~6.2]{lo} that $\E$ is Ulrich with $\det \E$ not big. 

Vice versa assume that $\E$ is Ulrich. Since $\E$ is globally generated and $\det \E$ is not big, we have that $c_1(\E)^n=0$, hence $\rho(X) \ge 2$ and $\E$ is not big.

If $(X, \O_X(1))$ is a del Pezzo $n$-fold, then either $n \le 4$ and we get cases (ii1)-(ii3) of Corollary \ref{main4} by the same corollary, or $n \ge 5$. But in the latter case the classification of del Pezzo $n$-folds (see for example \cite[\S 1]{lp}, \cite{fu1}) gives that $\rho(X)=1$, hence this case does not occur. This proves (i).

Now suppose that $(X, \O_X(1))$ is a Mukai $n$-fold.

If $n \le 4$ we get cases (x1)-(x5) of Theorem \ref{main3} by Corollary \ref{main4}(iii).

If $n \ge 5$, Mukai's classification \cite[Thm.~7]{mu} gives that $X$ is as in examples 3, 4 and 7 in \cite[Ex.~2]{mu}.

In example 3 we have that either $X = \P^3 \times \P^3$ or $X = \P(T_{\P^3})$. In the second case $\E$ is as in (ii2) by \cite[Prop.~6.2]{lo}. In the first case $c_1(\E)^6=0$ immediately gives that $\det \E= p^*(\O_{\P^3}(a))$, for some $a \in \Z$, where $p : \P^3 \times \P^3 \to \P^3$ is one of the two projections. It follows by \cite[Lemmas 4.1 and 5.1]{lo} and Remark \ref{kno1} that $\E$ is as in (ii1). 

In example 4 we have that $X = \P^2 \times Q_3$ and $c_1(\E)^5=0$ immediately gives that either $\det \E= p_1^*(\O_{\P^2}(a))$ or $p_2^*(\O_{Q_3}(a))$, for some $a \in \Z$, where $p_i, i=1, 2$ are the projections. In the second case it follows by \cite[Lemmas 4.1 and 5.1]{lo} and \cite[Lemma 3.2(iv)]{lms} that $\E$ is as in (ii3). In the first case it follows by \cite[Lemma 5.1]{lo} that there is a rank $r$ vector bundle $\G$ on $\P^2$ such that $\E \cong p_1^*\G$. Now the vanishings $H^i(\E(-pH))=0$ for all $i \ge 0$ and $1 \le p \le 5$, together with the K\"unneth formula imply that $H^j(\G(-2)(-s))=0$ for all $j \ge 0$ and $1 \le s \le 3$. But then $\G(-2)$ is an Ulrich vector bundle for $(\P^2, \O_{\P^2}(1))$ and Remark \ref{kno1} gives that $\G \cong \O_{\P^2}(2)^{\oplus r}$, thus giving the contradiction $0 = H^2(\G(-5))=H^2(\O_{\P^2}(-3))\ne 0$.

In example 7 we have that $X$ is the blow-up $\varepsilon : X \to \P^5$ along a line with exceptional divisor $E$ and let $\widetilde H = \varepsilon^*(\O_{\P^5}(1))$. As in the proof of Theorem \ref{main3} (see Claim \ref{b21bis}) we have that $E^5=-4, \widetilde H\cdot E^4=-1, \widetilde H^5=1$ and $\widetilde H^j \cdot E^{5-j}=0$ for $2 \le j \le 4$. Now $\det \E = a\widetilde H + bE$, for some $a, b \in \Z$ and
$$0 = c_1(\E)^5 = \sum\limits_{j=0}^5 \binom{5}{j}a^jb^{5-j}\widetilde H^j\cdot E^{5-j}=a^5-5ab^4-4b^5.$$
If $a=0$ we get that $b=0$, that is $c_1(\E)=0$, a contradiction. If $a \ne 0$ then $b=-a$ and $\det \E = a(\widetilde H - E)$. Again as in the proof of Theorem \ref{main3} (see Claim \ref{b21bis}), using the morphism $p : X \cong \P(\O_{\P^3}^{\oplus 2} \oplus \O_{\P^3}(1)) \to \P^3$ we deduce by \cite[Lemma 5.1]{lo} that there is a rank $r$ vector bundle $\G$ on $\P^3$ such that $\E \cong p^*(\G(3))$. Now \cite[Lemma 4.1]{lo} gives that $H^i(\G(-s)) = 0$ for all $i \ge 0$ and $1 \le s \le 4$. In particular $\G$ is an Ulrich vector bundle for $(\P^3, \O_{\P^3}(1))$, hence $\G \cong \O_{\P^3}^{\oplus r}$ by Remark \ref{kno1}. But this gives the contradiction 
$0 =  H^3(\G(-4)) = H^3(\O_{\P^3}(-4)^{\oplus r}) \ne 0$. Thus example 7 is excluded.

This proves (ii).
\end{proof}
\renewcommand{\proofname}{Proof}

\section{Examples}

In this section we will give some examples that are significant both with respect to Theorem \ref{main3} (for the statement but also for the method of proof) and to the fact that they do not appear in lower dimension.

\begin{example} (cfr. cases (v2) and (v3) in Theorem \ref{main3})
\label{altro}

Let $C$ be a smooth curve, let $L$ be a very ample line bundle on $C$, let $\G$ be an Ulrich vector bundle for $(C,L)$ and let $X=\P^2 \times C \times \P^1$ and $\O_X(1) = \O_{\P^2}(1) \boxtimes L \boxtimes \O_{\P^1}(1)$. For case (v3), let $p : X \to \P^2 \times C$ be the projection and let $\E = p^*(\O_{\P^2}(2) \boxtimes \G(L))$. Then $\E$ is Ulrich on $X$ by \cite[(3.5)]{b1} and $c_1(\E)^4=0$. Similarly, for case (v2), let $q : X \to C \times \P^2$ be the projection and let $\E = q^*(\G(2L) \boxtimes \O_{\P^1}(3))$.  
\end{example}

\begin{example} (cfr. case (vi2) in Theorem \ref{main3})
\label{secondo-bis}

Let $B$ be a smooth irreducible curve and let $L$ be a very ample line bundle on $B$ such that $K_B+2L$ is very ample. Let $Y = \P^1 \times \P^1 \times \P^1$ and let $\L = \O_{\P^1}(1) \boxtimes \O_{\P^1}(1) \boxtimes\O_{\P^1}(1)$. Let $X = B \times Y$, $\O_X(1) = L \boxtimes \L$ and let $p_1 : X \to B$ be the first projection. Then $K_X+2H  = p_1^*(K_B+2L)$, hence $p_1=\varphi_{K_X+2H}$. Thus $\varphi_{K_X+2H}$ gives a del Pezzo fibration on $X$ over $B$ with all fibers $\P^1 \times \P^1 \times \P^1$. Let $\F$ be any Ulrich vector bundle for $(B, L)$ and let $\G=\O_{\P^1}(2) \boxtimes \mathcal S'$ on $Y$. Then $\E:= \F(3L) \boxtimes \G$ is an Ulrich vector bundle on $X$ by \cite[(3.5)]{b1} with $c_1(\E)^4=0, c_1(\E)^3 \ne 0$.
\end{example}

\begin{example} (cfr. case (vi2) in Theorem \ref{main3})
\label{quarto-bis}

Let $Y$ be a smooth irreducible divisor of type $(1,2)$ on $\P^1 \times \P^3$ and let $\O_Y(1)=(\O_{\P^1}(1) \boxtimes \O_{\P^3}(1))_{|Y}$. Then ${\pi_1}_{|Y} : Y \to \P^1$ is a quadric fibration whose fibers are either smooth or a cone in $\P^3$ and there are $12$ such cones by \cite[\S 2]{lant}. Let $X = \P^1 \times Y$ and $\O_X(1) = (\O_{\P^1}(1) \boxtimes \O_Y(1))_{|X}$. Let $\G$ be an Ulrich vector bundle for $(Y,\O_Y(1))$ and let $\E = q^*(\G(1))$, where $q : X \to Y$ is the restriction of the second projection. It follows by \cite[Lemma 4.1]{lo} that $\E$ is an Ulrich vector bundle on $X$ and we are in case  (vi2) in Theorem \ref{main3}.
\end{example}

\begin{example} (cfr. case (vi3) in Theorem \ref{main3})
\label{secondo}

Let $Y = \P^1 \times \P^2 \times \P^2$ and let $L = \O_{\P^1}(1) \boxtimes \O_{\P^2}(1) \boxtimes\O_{\P^2}(1)$. Let $X \in |L|$ be smooth and irreducible and let $\O_X(1)=L_{|X}$. Let $p_1 : X \to \P^1$ be the restriction of the first projection on $Y$. Then $K_X+2H=p_1^*(\O_{\P^1}(1))$, hence $p_1=\varphi_{K_X+2H}$. A general fiber $F$ of $p_1$ is a hyperplane section of the Segre embedding of $\P^2 \times \P^2$, that is $F \cong \P(T_{\P^2})$. Thus $\varphi_{K_X+2H}$ gives a del Pezzo fibration on $X$ over $\P^1$ with general fiber $\P(T_{\P^2})$. To construct $\E$ observe that if $p : Y \cong \P(\O_{\P^1 \times \P^2}^{\oplus 3}) \to \P^1 \times \P^2$ with tautological line bundle $\xi$, then $L= \xi + p^*M$, where $M= \O_{\P^1}(1) \boxtimes \O_{\P^2}(1)$. Let $\E'=p^*(\O_{\P^1}(4) \boxtimes \O_{\P^2}(2))$. Then $\E'$ is an Ulrich line bundle for $(Y, L)$ by \cite[(3.5)]{b1} and \cite[Lemma 4.1(ii)]{lo}. Hence $\E = \E'_{|X}$ is an Ulrich line bundle for $(X, \O_X(1))$ with $c_1(\E)^4=0, c_1(\E)^3 \ne 0$. Moreover $\Phi$ is just the composition of $p_{|X}: X \to \P^1 \times \P^2$ with the embedding of $\P^1 \times \P^2$ by $\O_{\P^1}(4) \boxtimes \O_{\P^2}(2)$. Finally,  $p_{|X}$ has $M^3=3$ fibers that are a linear $\P^2$ by \cite[Ex.~14.1.5]{bs2}, hence so does $\Phi$.
\end{example}

We do not know if there is an example of a triple $(X, \O_X(1), \E)$ with $c_1(\E)^4=0, c_1(\E)^3 \ne 0$ and with the del Pezzo fibration having as fibers the blow-up of $\P^3$ in a point.

\begin{example} (cfr. case (vii) in Theorem \ref{main3})
\label{terzo}

Let $X = \P^2 \times Q$, where $Q=Q_2$, let $\O_X(1) = \O_{\P^2}(2) \boxtimes \O_Q(1)$ and let $p_1 : X \to \P^2$ be the first projection. Then $K_X+2H=p_1^*(\O_{\P^2}(1))$, hence $p_1=\varphi_{K_X+2H}$ is a quadric fibration over $\P^2$. Let $\F$ be an Ulrich vector bundle on $(\P^2, \O_{\P^2}(2))$ and let $\E=\F(2) \boxtimes \mathcal S'$. Then $\E$ is an Ulrich vector bundle for $(X, \O_X(1))$ by \cite[(3.5)]{b1}. Moreover $c_1(\E)^4=0$ and $c_1(\E)^3 \ne 0$. 
\end{example}

\begin{example} (cfr. case (ix) in Theorem \ref{main3})
\label{quarto}

Let $Y$ be a smooth irreducible threefold and let $M$ be a very ample line bundle on $Y$ such that $K_Y + 3M$ is very ample. Let $Z= \P^2 \times Y$ and let $L = \O_{\P^2}(1) \boxtimes M$. Let $X \in |L|$ be smooth and irreducible and let $\O_X(1)=L_{|X}$. Let $p_2 : Z \to Y$ be the second projection and let $q = {p_2}_{|X}$. Then $K_X+2H = q^*(K_Y + 3M)$. Hence $q=\varphi_{K_X+2H}$ gives a structure of scroll over $Y$. Let $\G$ be an Ulrich vector bundle for $(Y, M)$ and let $\E'=p_2^*(\G(2M))$. Then $\E'$ is an Ulrich vector bundle for $(Y, L)$ by \cite[Lemma 4.1]{lo}. Hence $\E = \E'_{|X}$ is an Ulrich vector bundle for $(X, \O_X(1))$ by Remark \ref{gen}(iv). Since $\E =  q^*(\G(2M))$ it follows that $\Phi$ factorizes through $q : X \to Y$. Finally, $q$ has $M^3>0$ fibers that are a linear $\P^2$ by \cite[Ex.~14.1.5]{bs2}. Therefore $c_1(\E)^4=0, c_1(\E)^3 \ne 0$ and $\Phi$ has $M^3$ fibers that are a linear $\P^2$.
\end{example}

\begin{example} (cfr. case (x1) in Theorem \ref{main3})
\label{nuovo-due}

On a Fano $3$-fold $M$ of index $2$, there are several examples of Ulrich vector bundles for $(M,L)$, where $K_M=-2L$. See for instance \cite{b1, cmp}.
\end{example}

\begin{example} (cfr. case (x2) in Theorem \ref{main3})
\label{nuovo}

Let $\F=\O_{\P^1}(1)\boxtimes T_{\P^2}, (X, \O_X(1))=(\P(\F), \O_{\P(\F)}(1))$ and let $p :  X \to \P^1 \times \P^2$ be the projection map. It is easily verified that $\E =p^*(\O_{\P^1}(3)\boxtimes \O_{\P^2}(2))$ is an Ulrich bundle as in (x2) of Theorem \ref{main3}.
\end{example}

\begin{example} (cfr. case (x5) in Theorem \ref{main3})
\label{spi}

The following example was kindly suggested to us by D. Faenzi, whom we thank.

Let $Q \subset \P^4=\P(V)$ be a smooth quadric. We have a natural identification
\begin{equation}
\label{q10}
\Lambda^2 V \cong H^0(\Omega_{\P^4}(2)) \cong H^0(\Omega_Q(2)) \subset H^0(\Omega_{\P^4}(2)_{|Q}).
\end{equation}
We can define a vector bundle $\F$  by setting
$$\H^* = \cHom(\Omega_{\P^4}(2)_{|Q}, \O_Q) \otimes H^0(\Omega_{\P^4}(2))$$ 
and considering the exact sequence
$$0 \to \F \to \H^* \to \O_Q \to 0$$
where the morphism $\psi: \H^* \to \O_Q$ is defined locally by sending $\phi \otimes \sigma$ to $\phi(\sigma_{|Q})$. 

Now set $\G = \F^*(-1)$ so that we have an exact sequence
\begin{equation}
\label{effe}
0 \to \O_Q \to \H \to \G(1) \to 0
\end{equation}
defining a rank $39$ vector bundle $\G$ on $Q$. We have
\begin{claim}
\label{van} 
$H^j(\G(-2k) \otimes S^k \mathcal S)=0$ for all $j \ge 0, 0 \le k \le 2$.
\end{claim}
\begin{proof}
To see this set $\H_1 = \Omega_{\P^4}(2)_{|Q}$ and consider the twisted Euler sequence, for every $l \in \Z$,
\begin{equation}
\label{q1}
0 \to \H_1(l) \to H^0(\O_Q(1)) \otimes \O_Q(l+1) \to \O_Q(l+2) \to 0.
\end{equation}
For $l=-1$ we see that $H^j(\H_1(-1))=0$ for $j \ge 0$, hence also $H^j(\H(-1))=0$ for $j \ge 0$. Then \eqref{effe} tensored by $\O_Q(-1)$ gives that 
$$H^j(\G)=0 \ \hbox{for} \ j \ge 0.$$
Choosing $l=-3$ in \eqref{q1} and tensoring with $\mathcal S$, we get the exact sequence
$$0 \to \H_1(-3) \otimes \mathcal S \to \mathcal S(-2)^{\oplus 5} \to \mathcal S(-1) \to 0.$$
Using the fact that $\mathcal S$ is Ulrich, we deduce that $H^j(\H_1(-3) \otimes \mathcal S)=0$ for $j \ge 0$, hence also $H^j(\H(-3) \otimes \mathcal S)=0$ for $j \ge 0$. Then \eqref{effe} tensored by $\mathcal S(-3)$ gives that 
$$H^j(\G(-2) \otimes \mathcal S)=0 \ \hbox{for} \ j \ge 0.$$
To finish the proof of  \eqref{van} it remains to prove that 
\begin{equation}
\label{q-1}
H^j(\G(-4) \otimes S^2 \mathcal S)=0 \ \hbox{for} \ j \ge 0.
\end{equation}
To this end we collect some well-known vanishings, that can be easily obtained using \cite{o} and restricting to the hyperplane section.

\begin{subclaim}
\label{van2}
\hskip 3cm

\begin{itemize}
\item[(i)] $H^0((S^2 \mathcal S)(l))=0$ for $l \le -2$.
\item[(ii)] $H^1((S^2 \mathcal S)(l))=0$ for $l \ne -2$.
\item[(iii)] $H^2((S^2 \mathcal S)(l))=0$ for $l \ne -3$.
\item[(iv)] $h^2((S^2 \mathcal S)(-3))=1$.
\item[(v)] $H^3((S^2 \mathcal S)(-4))=0$.
\item[(vi)] $h^3((S^2 \mathcal S)(-5))=10$.
\end{itemize}
\end{subclaim}
Setting $l=-5$ in \eqref{q1} and tensoring with $S^2 \mathcal S$, we get the exact sequence
\begin{equation}
\label{q8}
0 \to \H_1(-5) \otimes S^2 \mathcal S \to (S^2 \mathcal S)(-4)^{\oplus 5} \to (S^2 \mathcal S)(-3) \to 0
\end{equation}
and we deduce by Subclaim \ref{van2}(i), (ii) and (iii) that $H^j(\H_1(-5) \otimes S^2 \mathcal S)=0$ for $j=0, 1, 2$, hence also 
\begin{equation}
\label{q5}
H^j(\H(-5) \otimes S^2 \mathcal S)=0 \ \hbox{for} \ j=0, 1, 2.
\end{equation}
Now tensoring \eqref{effe} with $(S^2 \mathcal S)(-5)$ we get the exact sequence
\begin{equation}
\label{q4}
0 \to (S^2 \mathcal S)(-5) \to \H(-5) \otimes S^2 \mathcal S \to \G(-4) \otimes S^2 \mathcal S \to 0
\end{equation}
and applying \eqref{q5} and Subclaim \ref{van2}(ii) and (iii) we get that 
$$H^j(\G(-4) \otimes S^2 \mathcal S)=0 \ \hbox{for} \ j = 0, 1.$$
Moreover \eqref{q5} and \eqref{q4} give rise to the exact sequence
\begin{equation}
\label{q6}
0 \to H^2(\G(-4) \otimes S^2 \mathcal S) \to H^3((S^2 \mathcal S)(-5)) \to H^3(\H(-5) \otimes S^2 \mathcal S) \to H^3(\G(-4) \otimes S^2 \mathcal S) \to 0.
\end{equation}
Now note that from \eqref{q8} we have, applying Subclaim \ref{van2}(iii), (iv) and (v), that $h^3(\H_1(-5) \otimes S^2 \mathcal S)=h^2((S^2 \mathcal S)(-3))=1$, hence
$$h^3(\H(-5) \otimes S^2 \mathcal S))=10.$$
Since $h^3((S^2 \mathcal S)(-5))=10$ by Subclaim \ref{van2}(vi), to complete the proof it remains to show that the morphism 
$$H^3((S^2 \mathcal S)(-5)) \to H^3(\H(-5) \otimes S^2 \mathcal S)$$
is injective, or, by Serre's duality, that
$$H^0(\H^* \otimes S^2 \mathcal S) \to H^0(S^2 \mathcal S)$$
is surjective.

To see the latter, using the isomorphism 
$$\cHom(\Omega_{\P^4}(2)_{|Q}, \O_Q) \otimes \Omega_Q(2) \cong \cHom(\Omega_{\P^4}(2)_{|Q}, \Omega_Q(2))$$ 
and recalling that $S^2 \mathcal S \cong \Omega_Q(2)$ by \cite[Ex.~1.5]{o}, we have that the map $H^0(\H^* \otimes S^2 \mathcal S) \to H^0(S^2 \mathcal S)$ 
can be identified, with our choices, with the map
$$\Hom({\Omega_{\P^4}}_{|Q}(2), \Omega_Q(2)) \otimes H^0(\Omega_{\P^4}(2)) \to H^0(\Omega_Q(2))$$
which sends $\phi \otimes \sigma$ to $\phi(\sigma_{|Q})$. This map is clearly surjective by \eqref{q10}. Hence it is an isomorphism and therefore \eqref{q6} gives that
$$H^j(\G(-4) \otimes S^2 \mathcal S)=0 \ \hbox{for} \ j = 2, 3.$$
This proves \eqref{q-1} and the claim.
\end{proof}
\end{example}

\begin{example} (cfr. cases (xii)-(xiii) in Theorem \ref{main3})
\label{primo}

Let $B$ be a smooth irreducible variety of dimension $b=1, 2$ and let $\F$ be a rank $5-b$ very ample vector bundle on $B$. Let $X = \P(\F)$ with tautological line bundle $\O_X(1)$ and projection $p: X \to B$. Let $M$ be a line bundle on $B$ such that $H^i(M)=0$ for all $i \ge 0$ and set $\E= \Omega_{X/B}(2H+p^*M)$. When $b=2$ suppose also that $H^i(\F(M-\det\F))=0$ for all $i \ge 0$. Then $\E$ is an Ulrich vector bundle for $(X, \O_X(1))$, $\E$ is not big, $\E_{|f} \cong \Omega_f(2)$ and, in many cases, $c_1(\E)^4>0$. This is shown for $b=1$ in \cite[Lemma 4.1]{lm}. With the same method it can be shown for $b=2$. An example for $b=2$ can be obtained by picking a very ample line bundle $L$ on $B$, $\F= L^{\oplus 3}$ and $M=\L(-2L)$ where $\L$ is an Ulrich line bundle for $(B, 2L)$. Explicitly one can take $B=\P^1\times \P^1, L=\O_{\P^1}(1) \boxtimes \O_{\P^1}(1)$ and $M = \O_{\P^1}(-1) \boxtimes \O_{\P^1}(1)$. Note that for $b=2$ we have by \eqref{ellesse} that $\nu(\E)=r+2$. Moreover, in order to get restrictions $\E_{|f}$ with trivial summands, one can add to $\E$ direct summands of type $p^*(\L(\det \F))$, where $\L$ is a line bundle on $B$ such that $H^j(\L \otimes S^k \F^*)=0$ for $j \ge 0, 0 \le k \le b-1$.
\end{example}

\begin{example} (cfr. case (xii) in Theorem \ref{main3})
\label{primobis}

Let $X = \P^1 \times \P^3$ and $\O_X(1)=\O_{\P^1}(1) \boxtimes \O_{\P^3}(1)$. It is easily seen that the vector bundle
$$\E =  [\O_{\P^1}(2) \boxtimes (T_{\P^3}(-1))^{\oplus 2}] \oplus  [\O_{\P^1}(3) \boxtimes \O_{\P^3}]^{\oplus (r-6)}$$
is Ulrich, $c_1(\E)^4 > 0, \nu(\E)=r+2$ and $\E_{|f} = T_{\P^3}(-1)^{\oplus 2} \oplus  \O_{\P^3}^{\oplus (r-6)}$. This is the last possible case in Theorem \ref{main3}(xii), as in \cite[Thm.~1(v)]{su}.
\end{example}

We do not know if the case with restriction $\N(1) \oplus \O_{\P^3}^{\oplus (r-2)}$ actually occurs in Theorem \ref{main3}(xii).

\begin{example} (cfr. case (xiv) in Theorem \ref{main3})
\label{settimo}

Let $B$ be a smooth irreducible curve, let $L$ be a very ample line bundle on $B$ and let $Q=Q_3$. Let $X = B \times Q$ and let $\O_X(1) = L \boxtimes \O_Q(1)$. Then the first projection $p_1 : X \to B$ is a quadric fibration associated to $K_X+3H$. Let $\L$ be an Ulrich line bundle for $(B,L)$ and let $\E=\L(3L) \boxtimes \mathcal S$. Then $\E$ is an Ulrich rank $2$ relative spinor bundle for $(X, \O_X(1))$ by \cite[(3.5)]{b1} and  \cite[Lemma 3.2(iii)]{lms}. Note that $c_1(\E)^4>0$. Moreover $\E$ is not big by \cite[Prop.~3.3(iii)]{lms} and \cite[Lemma 2.4]{lm}.
\end{example}

\begin{landscape}
\begin{table}
\caption{}
\begin{tabular}{ |p{.7cm} |p{2.9cm}| p{3.58cm}| p{4.3cm}| p{3.1cm}| p{3.3cm}| p{2.7cm}| p{1.4cm}|}
\hline
\multicolumn{8}{|c|}{{\bf Cases in Theorem \ref{main3} with $c_1(\E)^4 = 0$}} \\
\hline
Case & $X$ & $\O_X(1)$ & $\E$ & $B$ (if linear Ulrich triple) & case Lemma \ref{aggiu}, $\phi_{\tau}$ & $p$ (Def. \ref{trip}) or $q$ & Example \\
\hline
(i) & $\P^4$ & $\O_{\P^4}(1)$ & $\O_{\P^4}^{\oplus r}$ & & (a) & & \\
\hline
(ii1) & $\P^1 \times \P^3$ & $\O_{\P^1}(1) \boxtimes \O_{\P^3}(1)$ & $p^*(\O_{\P^3}(1))^{\oplus r}$ & $\P^3$ & $(b.2), \P^1 \times \P^3 \to \P^1$ & $\P^1 \times \P^3 \to \P^3$ & \\
\hline
(ii2) & $\P(\F)$ & $\O_{\P(\F)}(1)$ & $p^*(\G(\det \F))$ & curve & $(b.2), p$ & bundle map & \\
\hline
(iii) & $\P^2 \times \P^2$ & $\O_{\P^2}(1) \boxtimes \O_{\P^2}(1)$ & $p^*(\O_{\P^2}(2))^{\oplus r}$ & $\P^2$ & $(c.1)$ & $\P^2 \times \P^2 \to \P^2$ & \\
\hline
(iv) & $\P^1 \times Q_3$ & $\O_{\P^1}(1) \boxtimes \O_{Q_3}(1)$ & $p^*(\mathcal S(1))^{\oplus (\frac{r}{2})}$ & $Q_3$ & $(c.2), \P^1 \times Q_3 \to \P^1$ & $\P^1 \times Q_3 \to Q_3$ & \\
\hline
(v1) & $(\P^2 \times \P^3) \cap H$ & $(\O_{\P^2}(1) \boxtimes \O_{\P^3}(1))_{|X}$ & $q^*(\O_{\P^3}(2))^{\oplus r}$ & & $(c.3), X \to \P^2$ projection linear $\P^2$-bundle & $X \to \P^3$ projection & \\
\hline
(v2) & $\P(\F)$ & $\O_{\P(\F)}(1)$ & $p^*(\G(\det \F))$ & surface & $(c.3), p$ & bundle map & \ref{altro} \\
\hline
(v3) & $\P(\F)$ & $\O_{\P(\F)}(1)$ & $p^*(\G(\det \F))$ & $3$-fold & $(c.3)$, linear $\P^2$-bundle over surface & bundle map & \ref{altro} \\
\hline
(vi1) & $\P(\F)$ & $\O_{\P(\F)}(1)$ & $p^*(\G(\det \F))$ & $3$-fold, del Pezzo fibration, fibers $\P^2$ & $(d.2)$, del Pezzo fibration, fibers $V_7$ & bundle map & \\
\hline
(vi2) & $\P(\F)$ & $\O_{\P(\F)}(1)$ & $p^*(\G(\det \F))$ & $3$-fold, del Pezzo fibration, smooth fibers $\P^1 \times \P^1$ & $(d.2)$, del Pezzo fibration, smooth fibers $\P^1 \times \P^1 \times \P^1$ & bundle map & \ref{secondo-bis}, \ref{quarto-bis} \\
\hline
(vi3) & del Pezzo fibration, smooth fibers $\P(T_{\P^2})$ & del Pezzo fibration & $\E_{|\P(T_{\P^2})}$ pull-back from $\P^2$ & & $(d.2)$, del Pezzo fibration, smooth fibers $\P(T_{\P^2})$ & & \ref{secondo} \\
\hline
(vii) & $\P(\F)$ & $\O_{\P(\F)}(1)$ & $p^*(\G(\det \F))$ & 3-fold & $(d.3)$, quadric fibration over surface & bundle map & \ref{terzo} \\
\hline
(viii) & $\P(\F)$ & $\O_{\P(\F)}(1)$ & $p^*(\G(\det \F))$ & 3-fold & $(d.4), p$ & bundle map & \\
\hline
(ix) & scroll over a normal threefold & scroll & $\E_{|\P^1}$ trivial & & $(d.5)$, scroll over a normal threefold & & \ref{quarto} \\
\hline
(x1) & $\P^1 \times M$, $K_M=-2L$ & $\O_{\P^1}(1) \boxtimes L$ & $p^*(\G(L))$, $\G$ Ulrich for $(M, L)$ & $M$ & $(d.1)$ & $\P^1 \times M \to M$ & \ref{nuovo-due} \\
\hline
(x2) & $\P^1 \times \P(T_{\P^2})$ & $\O_{\P^1}(1) \boxtimes \O_{\P(T_{\P^2})}(1)$ & $p^*(\G \otimes$ $(\O_{\P^1}(2)\boxtimes \O_{\P^2}(3)))$ & $\P^1 \times \P^2$ & $(d.1)$ & $\P(\O_{\P^1}(1)\boxtimes T_{\P^2})$ $\to \P^1 \times \P^2$ & \ref{nuovo} \\
\hline
(x3) & $(\P^2 \times Q_3) \cap H$ & $(\O_{\P^2}(1) \boxtimes \O_{Q_3}(1))_{|X}$ & $q^*(\mathcal S(2))^{\oplus (\frac{r}{2})}$ & & $(d.1)$ & $X \to Q_3$ projection & \\
\hline
(x4) & $(\P^3 \times \P^3) \cap H \cap H'$ & $(\O_{\P^3}(1) \boxtimes \O_{\P^3}(1))_{|X}$ & $q^*(\O_{\P^3}(3))^{\oplus r}$ & & $(d.1)$ & $X \to \P^3$ projection & \\
\hline
(x5) & $\P(\mathcal S)$ & $\O_{\P(\mathcal S)}(1)\otimes p^*(\O_{Q_3}(1))$ & $p^*(\G(3))$ & $Q_3$ & $(d.1)$ & bundle map & \ref{spi} \\
\hline
\end{tabular}
\label{tab1}
\end{table}
\end{landscape}

\begin{landscape}
\begin{table}
\caption{}
\begin{tabular}{ |p{.7cm} |p{1.7cm}| p{2.8cm}| p{8.2cm}| p{1.1cm}| p{3.3cm}| p{2.0cm}| p{2.0cm}|}
\hline
\multicolumn{8}{|c|}{{\bf Cases in Theorem \ref{main3} with $c_1(\E)^4 > 0$}} \\
\hline
Case & $X$ & $\O_X(1)$ & $\E$ & $B$ & case Lemma \ref{aggiu}, $\phi_{\tau}$ & $p$ & Example \\
\hline
(xi) & $Q_4$ & $\O_{Q_4}(1)$ & $\mathcal S', \mathcal S'', \mathcal S' \oplus \mathcal S''$ & & $(b.1)$ & & \\
\hline
(xii) & $\P(\F)$ & $\O_{\P(\F)}(1)$ & $\E_{|f}= T_{\P^3}(-1)\oplus \O_{\P^3}^{\oplus (r-3)}$, \hskip-.06cm $\Omega_{\P^3}(2) \oplus \O_{\P^3}^{\oplus (r-3)}$, $\N(1) \oplus \O_{\P^3}^{\oplus (r-2)}$, or quotient \hskip3.0cm $0 \to \O_{\P^3}(-1)^{\oplus 2} \to \O_{\P^3}^{\oplus (r+2)} \to \E_{|f} \to 0$ & curve & $(b.2), p$ & bundle map & \ref{primo}, \ref{primobis} \\
\hline
(xiii) & $\P(\F)$ & $\O_{\P(\F)}(1)$ & $\E_{|f} \cong T_{\P^2}(-1)\oplus \O_{\P^2}^{\oplus (r-2)}$ & surface & $(c.3), p$ & bundle map & \ref{primo} \\
\hline
(xiv) & quadric fibration & quadric fibration & $\E_{|f}$ spinor bundle on general $f$ & curve & $(c.2), p$ & & \ref{settimo} \\
\hline
\end{tabular}
\label{tab2}
\end{table}
\end{landscape}

\end{document}